\documentclass{article}

\title{Martingale central limit theorem for random multiplicative functions}
\author{Ofir Gorodetsky, Mo Dick Wong}
\date{}

\usepackage[margin=1in]{geometry}
\usepackage{amssymb}
\usepackage{amsfonts}
\usepackage{amsthm}
\usepackage{amsmath}
\usepackage{hyperref}
\usepackage{cleveref}
\usepackage{mathtools}
\usepackage{url}
\usepackage{xcolor}
\usepackage{enumitem}

\theoremstyle{plain}
\newtheorem{thm}{Theorem}[section]
\newtheorem{lem}[thm]{Lemma}  
\newtheorem{proposition}[thm]{Proposition}
\newtheorem{cor}[thm]{Corollary}
\newtheorem{definition}[thm]{Definition}
\newtheorem{conj}[thm]{Conjecture}
\theoremstyle{remark}
\newtheorem{rem}{Remark}[section]

\newcommand{\QQ}{\mathbb{Q}}
\newcommand{\PP}{\mathbb{P}}
\newcommand{\RR}{\mathbb{R}}
\newcommand{\CC}{\mathbb{C}}
\newcommand{\ZZ}{\mathbb{Z}}

\newcommand{\EE}{\mathbb{E}}
\newcommand{\NN}{\mathbb{N}}

\newcommand{\Da}{\mathcal{D}}
\newcommand{\Fa}{\mathcal{F}}
\newcommand{\Ga}{\mathcal{G}}
\newcommand{\Na}{\mathcal{N}}
\newcommand{\Ma}{\mathcal{M}}
\newcommand{\Ea}{\mathcal{E}}
\newcommand{\Oa}{\mathcal{O}}
\newcommand{\Aa}{\mathcal{A}}

\newcommand{\fname}{H}
\newcommand{\Lname}{D}
\newcommand{\Lsum}{R}

\newcommand{\Sf}{\mathrm{sf}}
\newcommand{\RadSym}{\beta}
\newcommand{\Rad}{\beta}

\newcommand{\OurEpsilon}{\varepsilon}
\newcommand{\asymptoticO}{O}

\numberwithin{equation}{section}

\begin{document}

\maketitle
\begin{abstract}
Let $\alpha$ be a Steinhaus or a Rademacher random multiplicative function. For a wide class of multiplicative functions $f$ we show that the sum $\sum_{n \le x}\alpha(n) f(n)$, normalised to have mean square $1$, has a non-Gaussian limiting distribution. More precisely, we establish a generalised central limit theorem with random variance determined by the total mass of a random measure associated with $\alpha f$.

Our result applies to $d_z$, the $z$-th divisor function, as long as $z$ is strictly between $0$ and $\tfrac{1}{\sqrt{2}}$. Other examples of  admissible $f$-s  include any multiplicative indicator function with the property that $f(p)=1$ holds for a set of primes of density strictly between $0$ and $\tfrac{1}{2}$.
\end{abstract}

\section{Introduction}
We denote by $\Ma$ the set of multiplicative functions:
\[ \Ma := \{ f \colon \NN \to \CC : (n,m)=1 \implies f(nm)=f(n)f(m)\}.\]
Let $\alpha\colon \NN \to \CC$ be a Steinhaus random multiplicative function. This is a multiplicative function such that $(\alpha(p))_{p}$ are i.i.d.~random variables uniformly distributed on $\{ z \in \CC: |z|=1\}$, and $\alpha(n) := \prod_{p}\alpha(p)^{a_p}$ if $n$ factorizes into primes as $n=\prod_{p} p^{a_p}$. It is easy to see that
\begin{equation}\label{eq:orth}
\EE [\alpha(n)\overline{\alpha}(m)]=\delta_{n,m}
\end{equation}
for all $n,m \in \NN$. Given $f \in \Ma$ we consider the random sum
\[ S_x =S_{x,f}:=\frac{1}{\sqrt{\sum_{n\le x} |f(n)|^2}}\sum_{n\le x} \alpha(n) f(n).\]
By \eqref{eq:orth}, $\EE [ |S_x|^2]=1$. 
We describe the limiting distribution of $S_x$ as $x\to \infty$ for a wide class of $f$-s. 

Given $\theta \in \CC$ we denote by $\mathbf{P}_\theta$ the class of functions
$g\colon \NN \to \CC$ that satisfy the following: the sum $\pi_{g}(t):=\sum_{p\le t} g(p)$ can be written as
\begin{align} \label{eq:condition-P}
\pi_{g}(t) = \theta \frac{t}{\log t} + \Ea_g(t),
\end{align}

\noindent where $\Ea_g(t) = o(t / \log t)$ as $t \to \infty$ and $\int_2^\infty t^{-2} |\Ea_g(t)|dt < \infty$.\footnote{
Let $\mathrm{Li}(t)=\int_{2}^{t} \tfrac{dt}{\log t}=\tfrac{t}{ \log t}  + \asymptoticO(\tfrac{t}{\log^2 t})$. Abusing notation, we may write \eqref{eq:condition-P} equivalently as
$\pi_{g}(t) = \theta \mathrm{Li}(t) + \Ea_g(t)$ for $t\ge 2$, where $\Ea_g(t) = o(t / \log t)$ is such that $\int_2^\infty t^{-2} |\Ea_g(t)|dt < \infty$. We will use both formulations interchangeably.}

\begin{thm}\label{thm:main}
Let $f \in \Ma$. Suppose the following conditions are satisfied.
\begin{itemize}
\item[(a)] There is $\theta \in (0,\tfrac{1}{2})$ such that $|f|^2 \in \mathbf{P}_{\theta}$.
\item[(b)] There exists $c>0$ such that $\sum_{p} |f(p)|^3/p^{3/2-c}<\infty$,  and $|f(p^k)|^2= O(2^{k(1-c)})$ holds for all $k\ge 2$ and primes $p$.
\end{itemize}
Then we have
\begin{align}\label{eq:dist}
    S_x \xrightarrow[x \to \infty]{d} \sqrt{V_\infty} \ G
\end{align}
\noindent where $\displaystyle V_\infty:= \frac{1}{2\pi}\int_{\RR} \frac{m_\infty(ds)}{|\frac{1}{2} + is|^2}$ is almost surely finite and strictly positive (see \Cref{thm:mc-convergence} for the definition of $m_\infty$), and is independent of $G \sim \Na_\CC(0, 1)$. 
Moreover, the convergence in law is stable in the sense of \Cref{def:stable}.
\end{thm}
\begin{rem}
Without getting into the technical details here, let us mention that the random measure $m_\infty(ds)$ in the definition of $V_\infty$ may be constructed as the large-$x$ limit of
\begin{align}\label{eq:mxinf}
    m_{x, \infty}(ds) := \prod_{p \le x} \frac{\left|1 + \sum_{k \ge 1} \frac{\alpha(p)^k f(p^k)}{p^{k(1/2+is)}}\right|^2}{\EE\left[\left|1 + \sum_{k \ge 1} \frac{\alpha(p)^k f(p^k)}{p^{k(1/2+is)}}\right|^2\right]} ds
\end{align}

\noindent in probability (see \Cref{sec:convergence_measure} for the precise meaning), i.e., the random variance $V_\infty$ is closely related to the infinite Euler product associated to $\alpha(n) f(n)$. Even though such an infinite product does not make sense pointwise on the critical line, a weak interpretation can be taken such that operations like integration against nice test functions could be justified. While not established in the current article, it is not hard to believe that such limiting measure cannot be absolutely continuous with respect to the Lebesgue measure (or else it would contradict our intuition that the infinite Euler product could not exist almost everywhere as a density function on the critical line). Later in \Cref{rem:support-ppt}, we shall sketch the proof that $m_\infty$ does not contain any Dirac mass. In other words, $m_\infty(ds)$ is a singular continuous measure on $\RR$ and its support property is closely related to (multi-)fractal analysis.
\end{rem}
\begin{rem}
Under the assumptions of \Cref{thm:main}, a classical result of Wirsing \cite{Wirsing} implies that 
\[ \sum_{n \le x} |f(n)|^2 \sim \frac{e^{-\gamma\theta}}{\Gamma(\theta)} \frac{x}{\log x} \prod_{p\le x} \left(\sum_{i=0}^{\infty} \frac{|f(p^i)|^2}{p^{i}}\right) \sim \frac{x (\log x)^{\theta-1}}{\Gamma(\theta)} \prod_{p} \left(\sum_{i=0}^{\infty} \frac{|f(p^i)|^2}{p^{i}}\right)\left(1-\frac{1}{p}\right)^{\theta}\]
holds as $x \to \infty$, where $\gamma$ is the Euler--Mascheroni constant. 
\end{rem}
% The number-theoretic input behind \Cref{thm:main} is rudimentary, and consists mainly of classical results of Wirsing \cite{Wirsing} and de Bruijn and van Lint \cite{dBvL} on nonnegative multiplicative functions (which we recall in \Cref{app:mean}). The probabilistic input required is mainly the martingale central limit theorem \cite{HH1980} and a probabilistic Tauberian theorem \cite{BW2023}.
Using \Cref{thm:main} we also obtain the convergence of moments.
\begin{cor}\label{cor:mom-convergence}
Under the same setting as \Cref{thm:main}, for any fixed $0 \le q \le 1$, 
\begin{align} \label{eq:mom-convergence}
    \lim_{x \to \infty} \EE\left[|S_x|^{2q}\right] = \Gamma(1+q) \EE\left[V_\infty^q\right].
\end{align}
\end{cor}

\begin{rem}\label{rem:mom-how}
Note that $\Gamma(1+q)$ corresponds to the $2q$-th absolute moment of a standard complex Gaussian random variable. The convergence \eqref{eq:mom-convergence} is expected to hold for all fixed $0 \le q < 1/\theta$, and the uniform integrability of $|S_x|^{2q}$ is the only ingredient needed (apart from \Cref{thm:main}) in order to establish such claim. With a refinement of certain estimates in the present article, one could obtain \eqref{eq:mom-convergence} for all $q \le 2$ almost immediately. We shall discuss this extension in \Cref{subsec:mom-convergence}, where we shall also explain why uniform integrability (and hence convergence of moments) is expected for the entire range $0 \le q < 1/\theta$.
\end{rem}
 A positive integer is called squarefree if it is indivisible by a perfect square other than $1$. Let $\mu^2$ be the indicator of squarefree integers and let 
 \[ \Ma^{\Sf} \subseteq \Ma\]
 be the set of \textit{real-valued} multiplicative function vanishing on non-squarefrees.  Let $\Rad \colon \NN \to \RR$ be a Rademacher random multiplicative function. This is a multiplicative function, supported on squarefree integers, such that $(\Rad(p))_p$ are i.i.d.~random variables taking the values $+1$ and $-1$ with probability $\frac{1}{2}$, and $\Rad(n) := \prod_{p \mid n} \Rad(p)$ if $\mu^2(n)=1$. It is easy to see that
\begin{equation}\label{eq:orthRad} 
\EE [ \Rad(n) \Rad(m)] = \delta_{n,m}\mu^2(n)
\end{equation}
for all $n,m \in \NN$. Given $f \in \Ma^{\Sf}$ we study
\[ S^{\RadSym}_x =S^{\RadSym}_{x,f}:=\frac{1}{\sqrt{\sum_{n\le x} f(n)^2}}\sum_{n\le x} \Rad(n) f(n).\]
By \eqref{eq:orthRad}, $\EE [ (S^{\RadSym}_x)^2]=1$. 
We describe the limiting distribution of $S^{\RadSym}_x$ as $x\to \infty$ for a wide class of $f$-s. 
\begin{thm}\label{thm:mainRad}
Let $f \in \Ma^{\Sf}$. Suppose the following conditions are satisfied.
\begin{itemize}
\item[(a)] There is $\theta \in (0,\frac{1}{2})$ such that $f^2 \in \mathbf{P}_{\theta}$.
\item[(b)] There exists $c>0$ such that $\sum_{p} |f(p)|^3/p^{3/2-c}<\infty$.
\end{itemize}
Then we have
\begin{align}\label{eq:distRad}
    S^{\RadSym}_x \xrightarrow[x \to \infty]{d} \sqrt{V_\infty^\RadSym} \ G
\end{align}
\noindent where $\displaystyle V_\infty^\RadSym:= \frac{1}{2\pi}\int_{\RR} \frac{m_\infty^\RadSym(ds)}{|\frac{1}{2} + is|^2}$ is almost surely finite and strictly positive (see \Cref{thm:mc-convergence-rad} for the definition of $m^{\RadSym}_\infty$), and is independent of $G \sim \Na(0, 1)$. Moreover, the convergence in law is stable in the sense of \Cref{def:stable}.
\end{thm}
\begin{cor}\label{cor:mom-convergenceRad}
Under the same setting as \Cref{thm:mainRad}, for any fixed $0 \le q \le 1$, 
\begin{align} \label{eq:mom-convergenceRad}
    \lim_{x \to \infty} \EE\left[|S^{\RadSym}_x|^{2q}\right] = \frac{2^q}{\sqrt{\pi}}\Gamma\left(\frac{1}{2}+q\right) \EE\left[(V_\infty^\RadSym)^q\right].
\end{align}
\end{cor}
The proofs of \Cref{thm:mainRad} and \Cref{cor:mom-convergenceRad} are very similar to those of \Cref{thm:main} and \Cref{cor:mom-convergenceRad}, and the necessary changes are described in \Cref{sec:rad}.
\subsection{Examples and previous works}
We single out some $f$-s to which our result applies.
\begin{itemize}
    \item $d_z$, the $z$-th divisor function, if $z \in (0,\tfrac{1}{\sqrt{2}})$. It is defined via $\sum_{n} d_z(n)/n^s=\zeta(s)^z$, or more explicitly as $d_z(p^i) = \binom{z+i-1}{i}$ on prime powers. Condition (a) of \Cref{thm:main} holds as a consequence of the prime number theorem with (weak) error term. The same is true for $z^{\omega}$ and $z^{\Omega}$ if  $z \in (0,\tfrac{1}{\sqrt{2}})$, where $\omega(n)$ is the number of distinct prime factors of $n$ and $\Omega(n)$ is the number of prime factors with multiplicity.
    \item Given a modulus $q$ and a subset $S \subseteq (\ZZ/q\ZZ)^{\times}$ let $P_{q,S}$ be the set of primes congruent to $a\bmod q$ for some $a \in S$. The indicator function of integers which are products of primes from $P_{q,S}$ satisfies the conditions of \Cref{thm:main} as long as $|S|<\frac{1}{2}|(\ZZ/q\ZZ)^{\times}|$. This is a consequence of the prime number theorem in arithmetic progressions with (weak) error term.
    \item Let $K/\QQ$ be a number field of degree greater than $2$. The indicator function of integers $n$ such that the ideal $(n)$ is representable as $\mathrm{Nm}(I)$ for some ideal in the ring of integers $\Oa_K$ satisfies the conditions of \Cref{thm:main}. This is a consequence of Chebotarev's density theorem with (weak) error term. 
    \item Let $b$ be the indicator of sums of two squares. For every $t\in (0,1)$, $b \cdot t^{\omega}$ satisfies the conditions  of \Cref{thm:main}.
    \item Any $f \in \Ma$ that takes values in $[-1,1]$ and such that $\sum_{p \le t}f(p)^2 = \theta t/\log t +O(t/\log^{1+\OurEpsilon}t)$ holds for some $\OurEpsilon>0$ and $\theta \in (0,\tfrac{1}{2})$ satisfies the conditions  of \Cref{thm:main}.
\end{itemize}
Various authors established central limit theorems for \begin{equation}\label{eq:2sums}\frac{\sum_{n \le x} \alpha(n)f(n)}{\sqrt{\sum_{n\le x}|f(n)|^2}}, \qquad \frac{\sum_{n \le x} \Rad(n)f(n)}{\sqrt{\sum_{n\le x}f(n)^2}}
\end{equation}
for particular $f$-s, but \Cref{thm:main} is the first to exhibit a nontrivial non-Gaussian limiting distribution. Previous works studied the cases where $f$ is the indicator of integers with $k$ prime factors \cite{Hough,HLD,Aggarwal}, indicator of short intervals \cite{Chatterjee,Pandey,SX} and indicator of polynomial values \cite{Najnudel,KSX,Paucity}. The recent work of Soundararajan and Xu \cite[Corollary 3.2 and Theorem 9.1]{SX} supplies a sufficient criterion for $\sum_{n \in \Aa} \alpha(n)/\sqrt{|\Aa|}$ and $\sum_{n \in \Aa} \Rad(n)/\sqrt{|\Aa|}$ to have a limiting Gaussian distribution, where $\Aa=\Aa_x \subseteq \NN$ are finite subsets varying as $x\to \infty$.

One of the approaches previously used to studying the sums  \eqref{eq:2sums} is McLeish's martingale central limit theorem \cite{McLeish}. Its use in number theory was pioneered by Harper \cite{HLD} and since then has been used in \cite{Aggarwal,KSX,SX}.
Despite being more flexible than the Feller--Lindeberg central limit theorem (which requires independence), McLeish's 1974 result only produces Gaussian limiting distribution and cannot deduce the type of results presented in the current paper. Inspired by the work of Najnudel, Paquette and Simm \cite{NPS}, we employ a more general martingale central limit theorem (see \Cref{lem:mCLT}) which has the following crucial distinctions. For one thing, the class of limiting distributions is given by the family of Gaussian mixture models: the limiting random variable can be written in the form $\sqrt{V_{\infty}}G$ for some general nonnegative random variable $V_{\infty}$ (independent of $G$, where $G$ is Gaussian) which may be interpreted as a random variance. For another, the distributional convergence can be upgraded to the stronger `stable' sense (see \Cref{def:stable}), and $V_\infty$ is measurable with respect to the $\sigma$-algebra in the original probability space. In layman's terms, the random variance can be determined from the value of the random multiplicative function, and indeed much of the difficulty in the present work lies in identifying and constructing $V_{\infty}$. The interested readers may consult \cite[Chapter 3.1--3.2]{HH1980} for a historical account and further discussions of martingale central limit theorems.

In possibly the most interesting case, $f \equiv 1$, Harper \cite{Har2020} showed that  $\sum_{n \le x}\alpha(n)/\sqrt{x}$ and $\sum_{n \le x}\Rad(n)/\sqrt{x}$ tend in distribution to $0$. Harper's arguments were adapted  to a model problem and simplified by Soundararajan and Zaman \cite{SZ}. A vast generalization of the bounds in \cite{SZ} was established recently by Gu and Zhang \cite{GZ}. See also Najnudel, Paquette and Simm \cite{NPS} for a different derivation of some of the bounds in \cite{SZ}.

If $f=f_k$ is the indicator function of integers with $k$ prime factors and $k \asymp \log \log x$ then Harper showed in \cite[Theorem 2]{HLD} that $\sum_{n \le x}\Rad(n)f(n)/\sqrt{\sum_{n \le x}f(n)^2}$ cannot have a Gaussian limiting distribution.
\subsection{Conjectures}
We expect that \Cref{thm:main} holds, with possibly slight changes in our assumption for the twist $f$, for any fixed $\theta$ arbitrary close to $1$. 
Indeed, many results in the current article may be extended to larger values of $\theta$ with more refined analysis, but the main obstacle to achieving full extension with our approach is a crucial $L^2$ approximation (see \Cref{prop:l2} in \Cref{subsec:overview}) which leads to the technical condition that $\theta \in (0, \tfrac{1}{2})$. We make the following modest conjecture.
\begin{conj}\label{conj:l1}
Let $f \in \Ma$. Suppose the following conditions are satisfied.
\begin{itemize}
\item[(a)] There is $\theta \in (0,1)$ such that $\sum_{p \le t} |f(p)|^2 = \theta \frac{t}{\log t}+ O\left(\frac{t}{\log^2 t}\right)$. 
\item[(b)] There exists $c>0$ such that $|f(p^k)|^2 = O(2^{k(1-c)})$ and $|f(p)|<1/c$ for all $k\ge 2$ and primes $p$.
\end{itemize}
Then the stable convergence \eqref{eq:dist} holds and we have \eqref{eq:mom-convergence} for any fixed $q \in [0, 1/\theta).$
\end{conj}

We believe that \Cref{conj:l1} will continue to hold when $\theta = 1$. Given the now-resolved Helson's conjecture, this claim may appear to contradict the better-than-squareroot cancellation phenomenon that was first established in the work of Harper \cite{Har2020}, where uniform two-sided estimates were obtained for low moments of 
\begin{equation}\label{eq:norm} 
(\log \log x)^{\frac{1}{4}} S_{x,f}
\end{equation}

\noindent when $f \equiv 1$. Those who are familiar with the theory of multiplicative chaos, however, may be able to quickly point out that $\theta = 1$ corresponds to the so-called critical regime, where one would expect $V_\infty:= \frac{1}{2\pi} \int_{\RR} |\tfrac{1}{2} + is|^{-2} m_\infty(ds)$ to be equal to $0$ almost surely. As such, the conclusion \eqref{eq:dist} will be merely reduced to e.g., $x^{-1/2} \sum_{n \le x} \alpha(n) \xrightarrow[x \to \infty]{p} 0$.

Given the uniform moment estimates, Harper asked in \cite[p.~12]{Har2020} whether one could establish a nontrivial distributional limit for \eqref{eq:norm} as $x \to \infty$. With his method of proof, Harper suggested that the limiting distribution might be closely related to the asymptotic behaviour of $\int_{-1/2}^{1/2} m_{x,\infty}(ds)$ where $m_{x,\infty}$ is defined in \eqref{eq:mxinf}. Here we refine Harper's question, and provide the following precise conjecture.
\begin{conj}\label{conj:critical}
If $f \in \{\mathbf{1}, \mu\}$, then
\begin{align}\label{eq:dist2}
  \frac{(\log \log x)^{\frac{1}{4}}}{\sqrt{x}}  \sum_{n \le x} \alpha(n)f(n) \xrightarrow[x \to \infty]{d} \sqrt{V_\infty^{\mathrm{critical}}} \ G.
\end{align}

\noindent Here, the random variable $V_\infty^{\mathrm{critical}}$ is almost surely finite, strictly positive, independent of $G\sim \Na_{\CC}(0,1)$, and is given by
\begin{align}\label{eq:Vcritical}
V_{\infty}^{\mathrm{critical}}
:= \lim_{x \to \infty} V_x^{\mathrm{critical}}
:= \lim_{x \to \infty} \frac{1}{2\pi}\int_{\RR} \frac{\sqrt{\log \log x}~ m_{x, \infty}(ds)}{|\frac{1}{2} + is|^2}
= \frac{1}{2\pi}\int_{\RR} \frac{m_{\infty}^{\mathrm{critical}}(ds)}{|\frac{1}{2} + is|^2}
\end{align}

\noindent where $m_{\infty}^{\mathrm{critical}}(ds) := \lim_{x \to \infty} \sqrt{\log \log x} ~ m_{x, \infty}(ds)$, and all the limits in \eqref{eq:Vcritical} are interpreted in the sense of convergence in probability (see \Cref{sec:convergence_measure} for details of convergence of random measures). Moreover, the convergence \eqref{eq:dist2} is stable, and we have
\begin{align}\label{eq:conjmom}
    \lim_{x \to \infty} \EE\left[ \left|\frac{(\log \log x)^{\frac{1}{4}}}{\sqrt{x}}  \sum_{n \le x} \alpha(n)f(n)\right|^{2q} \right] = \Gamma(1+q) \EE\left[ \left(V_\infty^{\mathrm{critical}} \right)^q\right]
\end{align}

\noindent for any fixed $q \in [0, 1)$.

\end{conj}
In a model problem, Aggarwal et al.~\cite{AggarwalConjectural} conjectured that $\lim_{N \to \infty}\EE |(\log N)^{1/4} A(N) |$ exists, where $A(N)$ is a model for $S_{x,1}$ defined originally in \cite{NPS,SZ}. This is related to the $q=\tfrac{1}{2}$ case of \eqref{eq:conjmom}, and our conjecture supports \cite[Conjecture 1.1]{Aggarwal}.

We complement our conjectures with some numerical evidence. The simulation experiment is based on $f(n) := \theta^{\Omega(n)/2}$, and in both \Cref{fig:sim1} and \Cref{fig:sim2} the subplots correspond to $\theta = 0.49$ (top-left), $\theta = 0.64$ (top-right), $\theta = 0.81$ (bottom-left) and $\theta = 1$ (bottom-right) respectively. 

\Cref{fig:sim1} shows the empirical distribution of $\Re \sum_{n \le x} \alpha(n) \theta^{\Omega(n)/2}$ (blue). The black curve corresponds to the Gaussian density with matching 2nd moment, whereas the red curve corresponds to the (empirical) distribution of $\sqrt{\frac{1}{2\pi} \int_{\RR} \frac{m_{x, \infty}(ds)}{|\frac{1}{2} + is|^2}} \Re(G)$ where $\Re(G) \sim \Na(0, 1/2)$ is independent. Despite the modest size of $x = 10000$, the red curve produces a surprisingly good fit for all values of $\theta$ considered.

\Cref{fig:sim2} shows the empirical distribution of the random variance $\sqrt{\frac{1}{2\pi} \int_{\RR} \frac{m_{x, \infty}(ds)}{|\frac{1}{2} + is|^2}}$. We note that the plots may not truly reflect the statistical properties of the limiting random variable $V_\infty$ given the choice of $x = 10000$, as the rate of convergence may be polylogarithmic in $x$. We anticipate light tail near $0$ and heavy tail at infinity to emerge in the empirical distribution as $x \to \infty$.

\begin{figure}
\begin{center}
\includegraphics[width=0.4\textwidth]{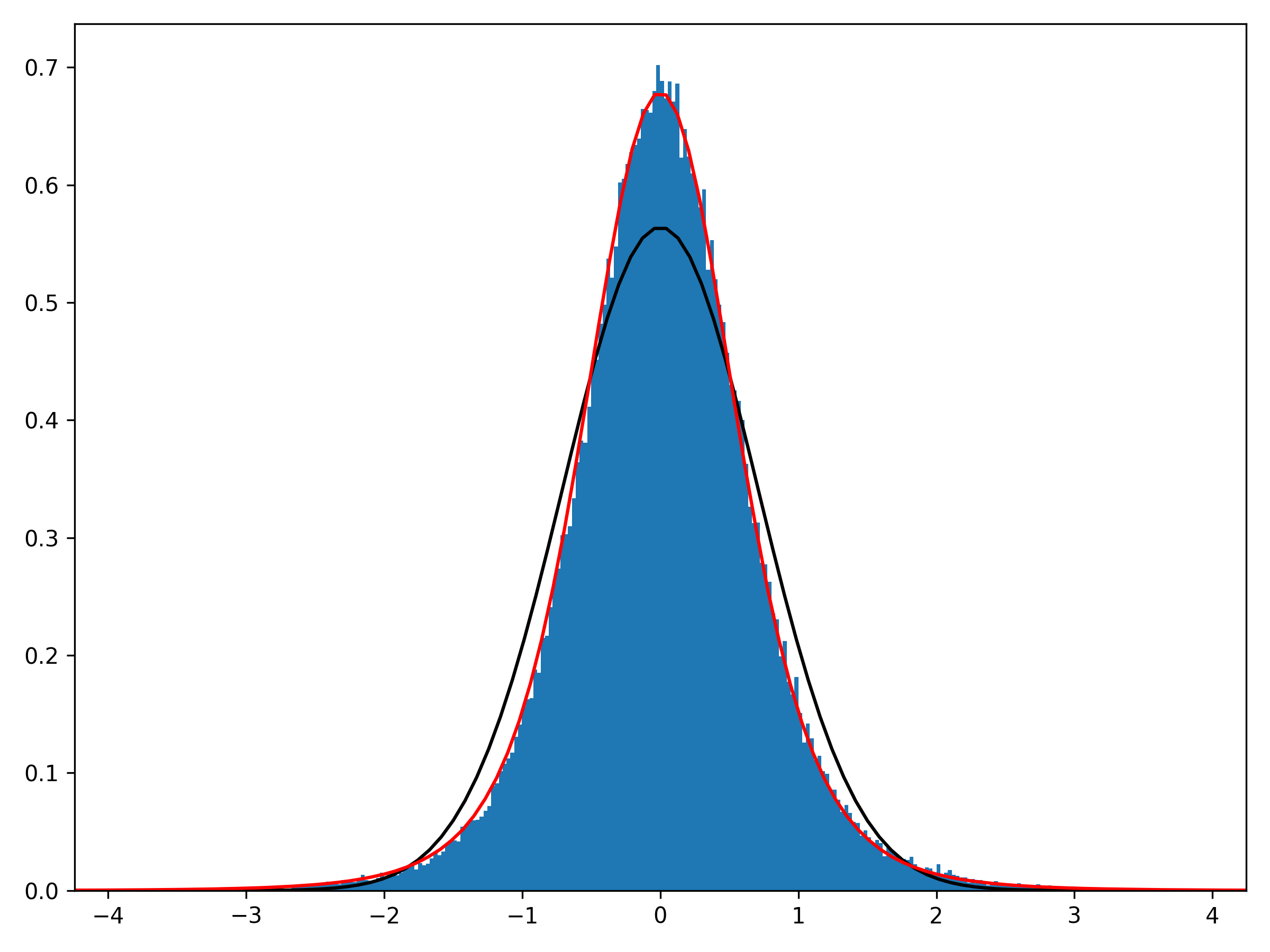}\hspace{1cm}
\includegraphics[width=0.4\textwidth]{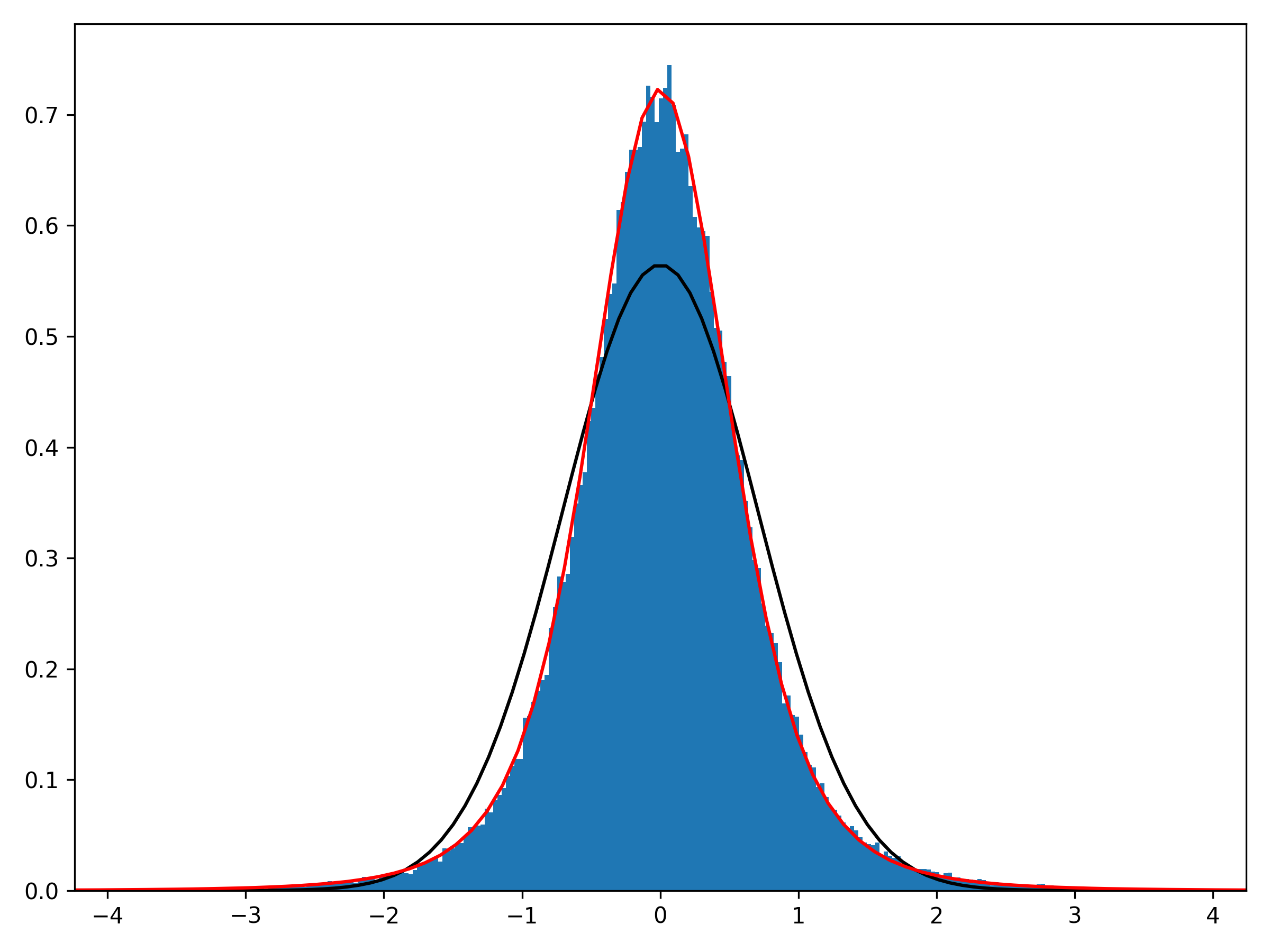}\\
\includegraphics[width=0.4\textwidth]{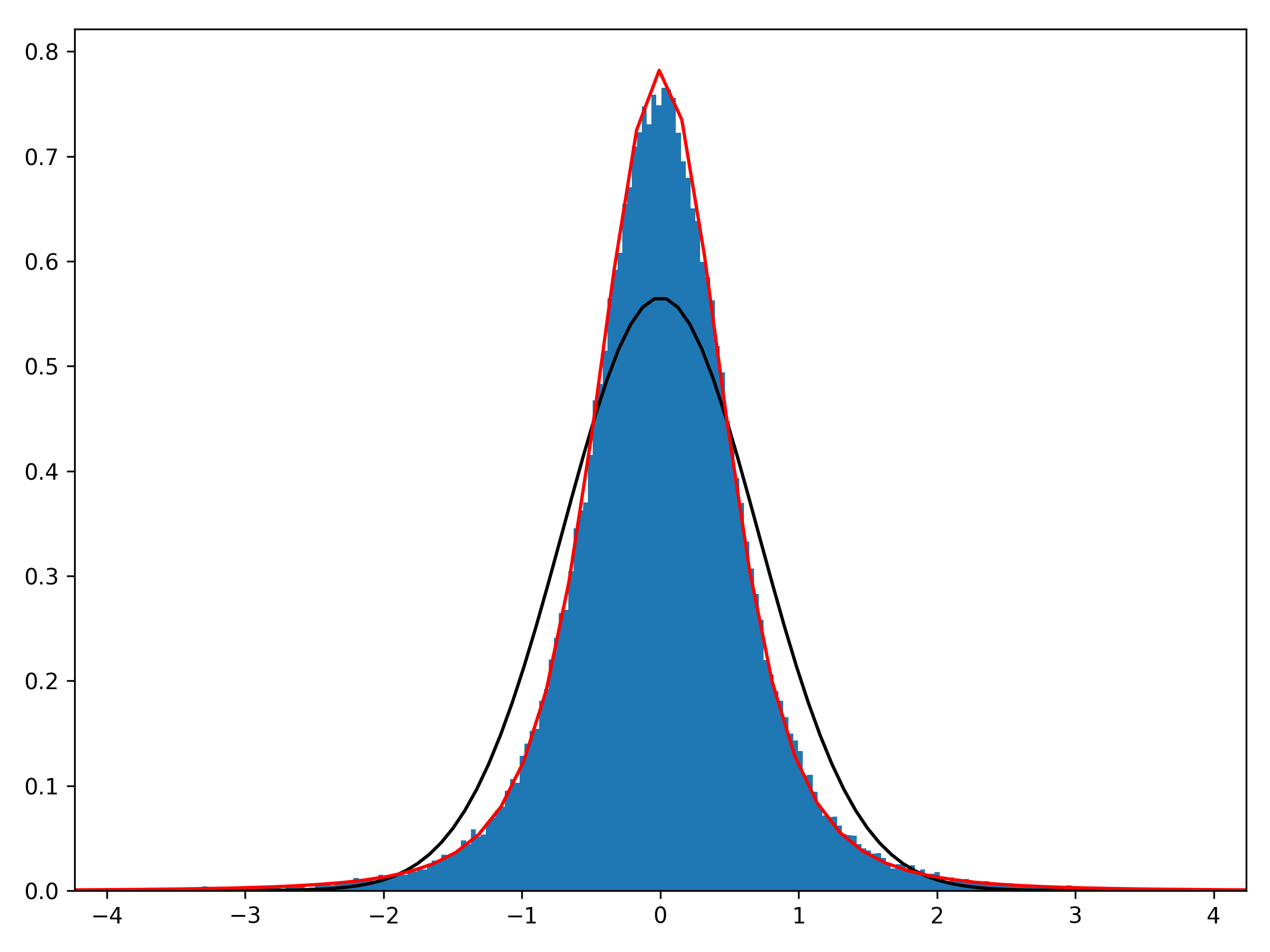}\hspace{1cm}
\includegraphics[width=0.4\textwidth]{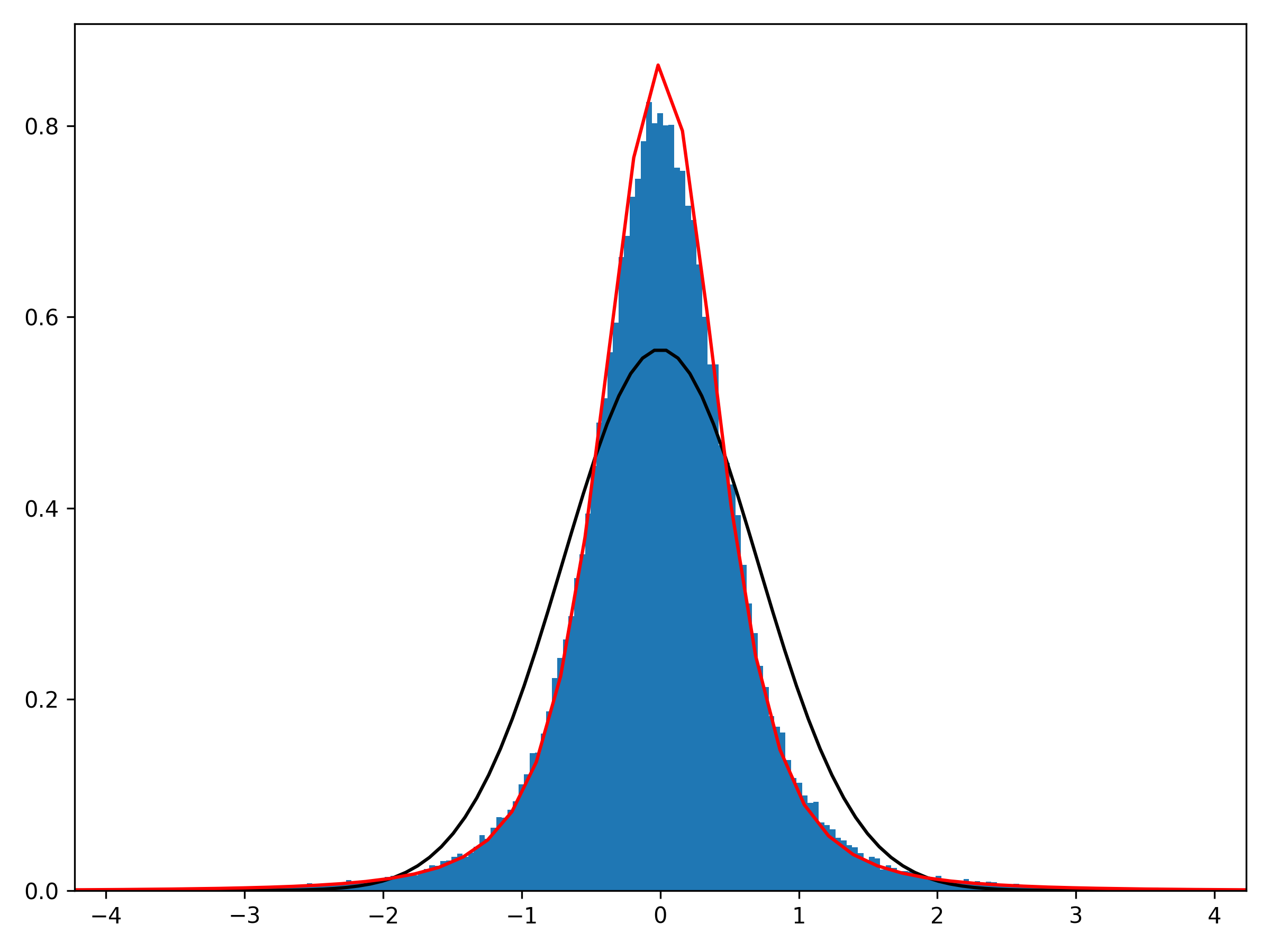}
\end{center}
\caption{\label{fig:sim1} Histogram of 120000 realisations of $\Re \sum_{n \le x} \alpha(n) \theta^{\Omega(n)/2}$, $x = 10000$.}
\end{figure}

\begin{figure}
\begin{center}
\includegraphics[width=0.4\textwidth]{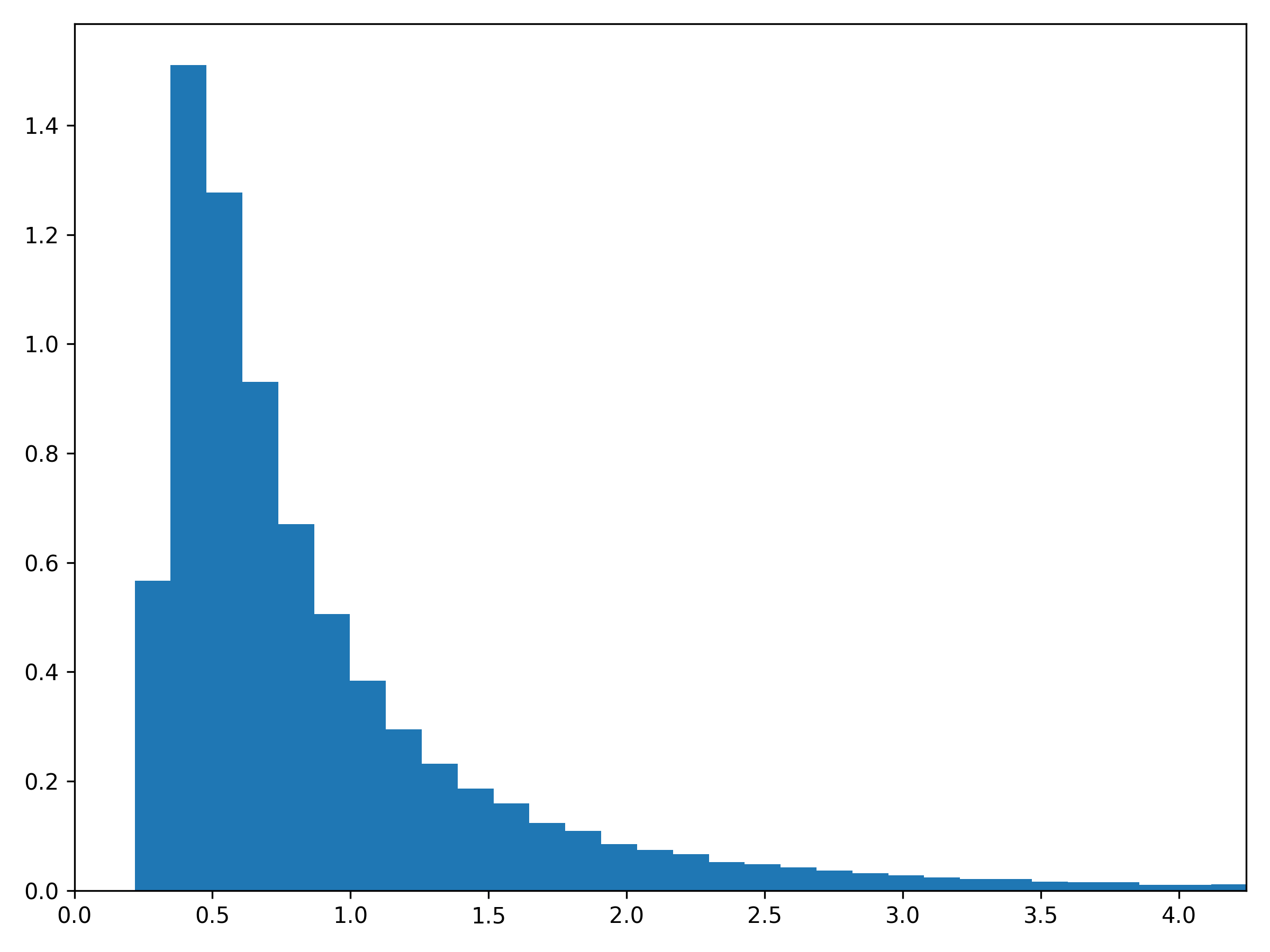} \hspace{1cm}
\includegraphics[width=0.4\textwidth]{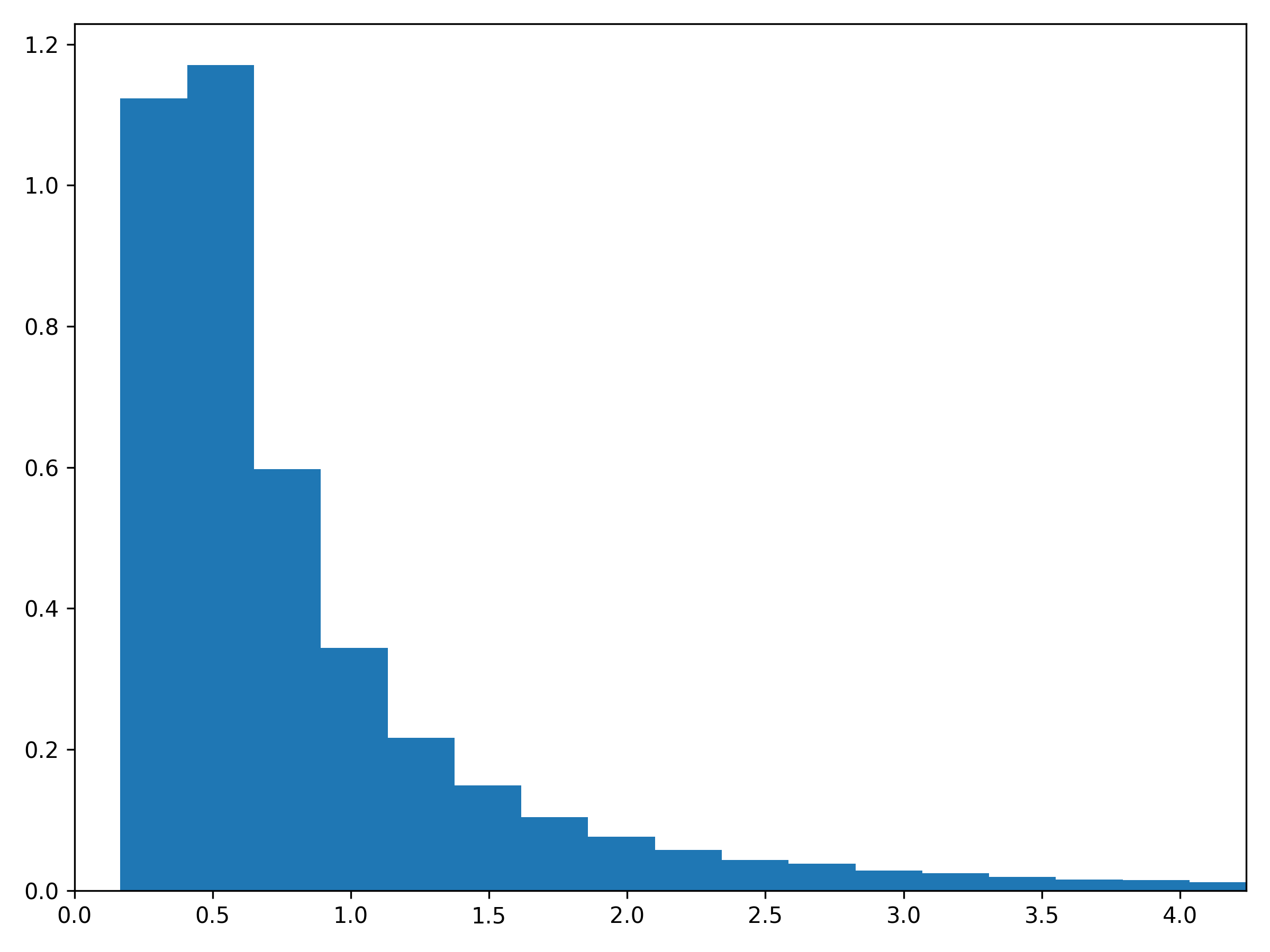}\\
\includegraphics[width=0.4\textwidth]{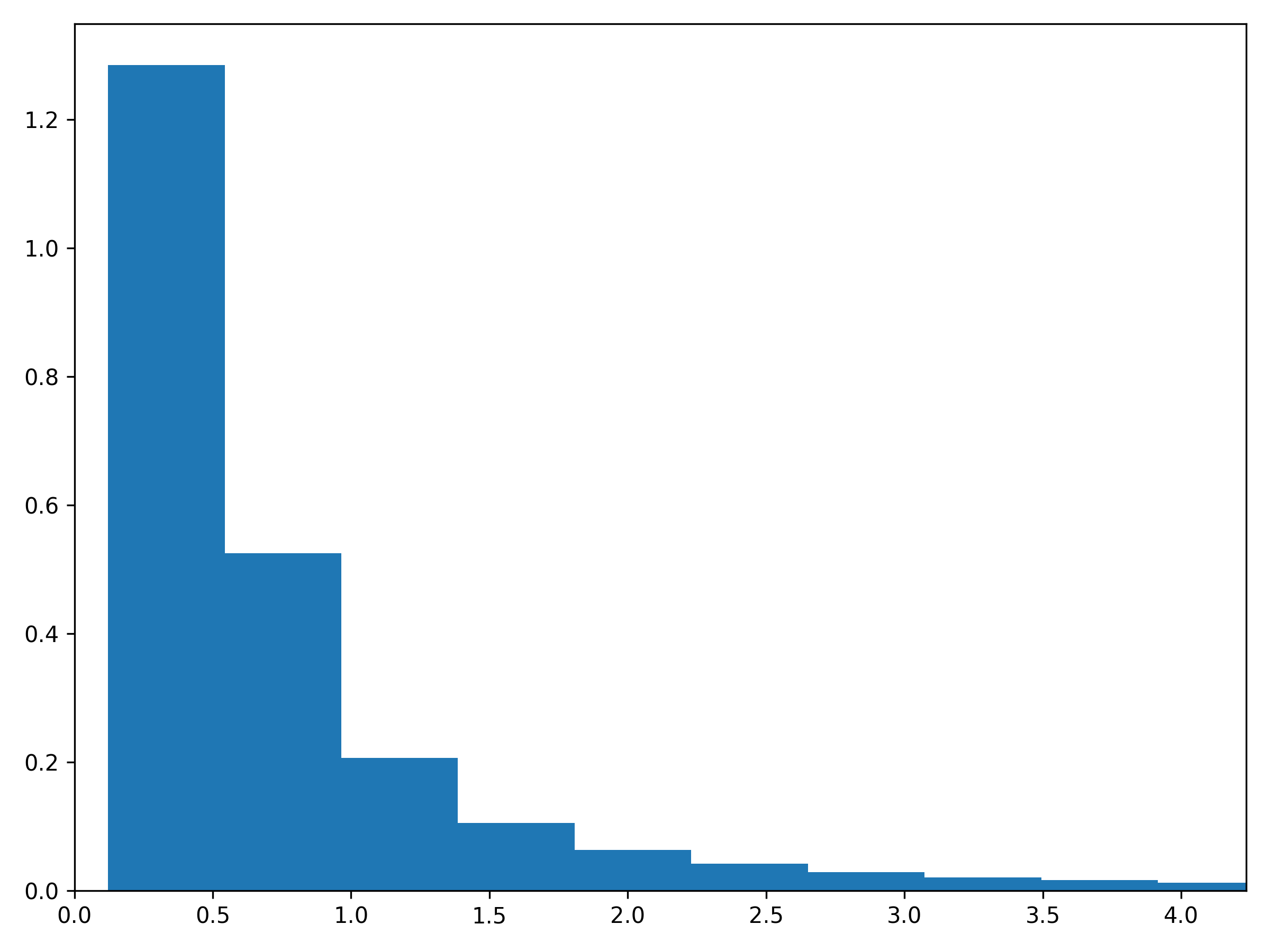} \hspace{1cm}
\includegraphics[width=0.4\textwidth]{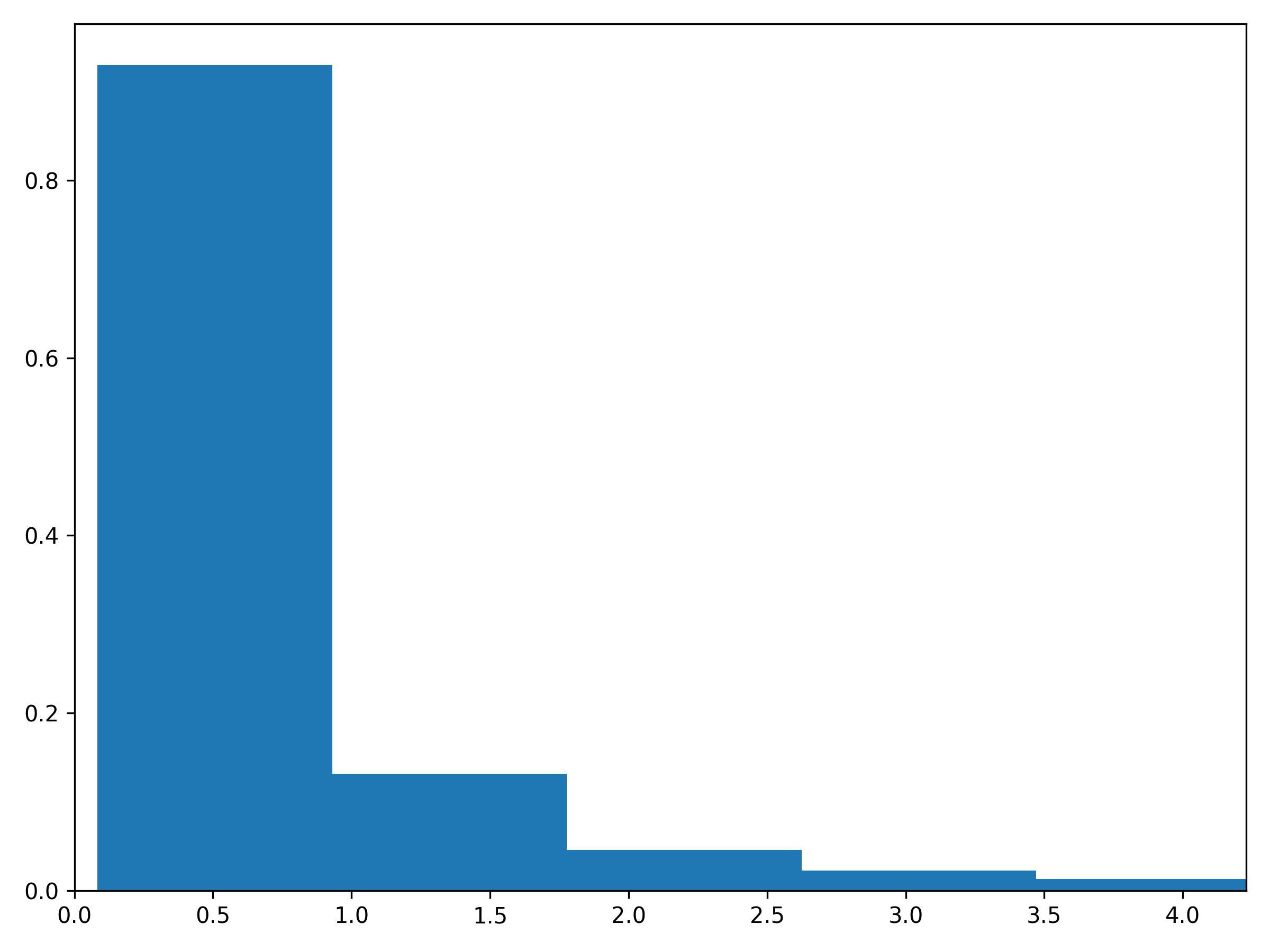}
\end{center}
\caption{\label{fig:sim2} Histogram of 120000 realisations of $\frac{1}{2\pi} \int_{\RR} \frac{m_{x, \infty}(ds)}{|\frac{1}{2} + is|^2}$, $x = 10000$.}
\end{figure}

\section*{Acknowledgments}
O.G.~has received funding from the European Research Council (ERC) under the European Union's Horizon 2020 research and innovation programme (grant agreement number 851318). We thank Adam Harper for comments that improved the introduction. The first author thanks Anurag Sahay and Max Xu for engaging discussions.

\section{Structure of the proof of \texorpdfstring{\Cref{thm:main}}{Theorem \ref{thm:main}}}\label{sec:mainproof}
\subsection{Analogy with the work of Najnudel, Paquette and Simm}
Let $\theta>0$. Let $(N_i)_{i=1}^{\infty}$ be i.i.d.~complex random variables distributed according to $N_i \sim \Na_{\CC}(0,1)$. Najnudel, Paquette and Simm \cite{NPS} investigated the distribution of the random variable $c_n$, defined as the coefficient of $u^n$ in the power series
\begin{equation}\label{eq:cndef} \exp\bigg( \sqrt{\theta}\sum_{i=1}^{\infty} N_i \frac{u^i}{\sqrt{i}}\bigg).
\end{equation}
This was motivated by the study of secular coefficients in the circular beta ensemble. It turns out that the asymptotic behaviour of $c_n$ as $n \to \infty$ is connected to the Gaussian multiplicative chaos measure on the unit circle, formally defined as
\begin{align}\label{eq:form}
    \mathrm{GMC}_{\theta}(d\vartheta):=\lim_{r \to 1^-} (1-r^2)^{\theta}\exp\bigg( 2\sqrt{\theta}\Re\sum_{k=1}^{\infty} N_k \frac{(re^{i\vartheta})^k}{\sqrt{k}}\bigg)d\vartheta, \qquad \vartheta \in [0, 2\pi],
\end{align}

\noindent when $\theta \in (0,1)$ (see \cite[Section 1.3]{NPS} for a discussion of \eqref{eq:form}). When $\theta \in (0,\tfrac{1}{2})$, they proved that \cite[Theorem 1.8]{NPS}
\begin{align}\label{eq:distNPS}
    \frac{c_n}{\EE[|c_n|^2]^{1/2}} \xrightarrow[n \to \infty]{d} \sqrt{V_\infty} \ G
\end{align}
\noindent where $\displaystyle V_\infty$ has law 
\[ V_{\infty}\overset{d}{=}\frac{1}{2\pi} \int_{0}^{2\pi} \mathrm{GMC}_{\theta}(d\vartheta)\]
and is independent of $G \sim \Na_\CC(0, 1)$. Their main tool in proving \eqref{eq:distNPS} is the martingale central limit theorem, and our proof has a lot in parallel with theirs. To see the analogy between $S_x$ and $c_n/\EE[|c_n|^2]^{1/2}$, we note that the generating function of $\alpha f$, i.e., its Dirichlet series, is given by
\[ \sum_{n=1}^{\infty} \frac{\alpha(n)f(n)}{n^s}=\prod_{p} \left( \sum_{i=0}^{\infty} \frac{\alpha(p^i)f(p^i)}{p^{is}}\right).\]
Using the central limit theorem and condition (a) in \Cref{thm:main}, 
\[ \sqrt{\frac{\log n}{n}}\sum_{n\le p< 2n} \alpha(p)f(p) \xrightarrow[n \to \infty]{d}  \sqrt{\theta}\ G\]
where $G \sim \Na_\CC(0, 1)$. Thus, informally,  
\[ \sum_{n=1}^{\infty} \frac{\alpha(n)f(n)}{n^s}\approx \exp \big(\sum_{p} \frac{\alpha(p)f(p)}{p^s}\big) \approx \exp\bigg( \sum_{i=1}^{\infty} 2^{-is}\sum_{2^i\le p<2^{i+1}} \alpha(p)f(p)\bigg) \approx \exp\bigg( \sqrt{\theta} \sum_{i=1}^{\infty} N_i 2^{-is} \sqrt{\frac{2^i}{\log(2^i)}} \bigg),\]
which resembles the power series in \eqref{eq:cndef}.
\subsection{The (complex) martingale central limit theorem}
The following result is deduced in \Cref{app:cmCLT} from the real-valued martingale central limit theorem \cite[Theorem 3.2 and Corollary 3.1]{HH1980}. It is the main tool behind the proof of \Cref{thm:main}.
\begin{lem}[Martingale central limit theorem]\label{lem:mCLT}
For each $n$, let $(M_{n,j})_{j\le k_n}$ be a complex-valued, mean-zero and square integrable martingale with respect to the filtration $(\Fa_{j})_j$, and $\Delta_{n, j} := M_{n, j} - M_{n, j-1}$ be the corresponding martingale differences. Suppose the following conditions are satisfied.
\begin{itemize}
    \item[(a)] The conditional covariances converge, i.e.,
    \begin{align}
    \label{eq:condcov1}
 \sum_{j =1}^{k_n} \EE\left[\Delta_{n, j}^2 | \Fa_{j-1}\right] &\xrightarrow[n \to \infty]{p} 0\\
 \label{eq:condcov2}
\qquad \text{and}\qquad V_{n} := \sum_{j =1}^{k_n} \EE\left[|\Delta_{n, j}|^2 | \Fa_{j-1}\right] & \xrightarrow[n \to \infty]{p} V_\infty.
    \end{align}

    \item[(b)] The conditional Lindeberg condition holds: for any $\delta > 0$,
    \begin{align*}
        \sum_{j = 1}^{k_n} \EE\left[|\Delta_{n, j}|^2 1_{\{|\Delta_{n, j}| > \delta\}} | \Fa_{j-1}\right] \xrightarrow[n \to \infty]{p} 0.
    \end{align*}
\end{itemize}

\noindent Then 
\begin{align*}
    M_{n, k_n} \xrightarrow[n \to \infty]{d} \sqrt{V_\infty} \ G
\end{align*}

\noindent where $G \sim \Na_{\CC}(0,1)$ is independent of $V_\infty$, and the convergence in law is also stable in the sense of \Cref{def:stable}.
\end{lem}
\subsection{A class of functions}
Given $\theta>0$ we introduce a subset $\Ma_{\theta}\subseteq \Ma$ of all functions $g \in \Ma$ which satisfy the following mild conditions:
\begin{enumerate}
\item For every $n \in \NN$, $g(n) \ge 0$.
    \item As $t \to \infty$, 
    \begin{equation}\label{eq:loglim}
    \sum_{p\le t} \frac{g(p)\log p}{p}\sim \theta \log t.
    \end{equation}
\item For $t \ge 2$, $\sum_{p \le t} g(p) \ll \frac{t}{\log t}$.
\item The following series converges: 
    \begin{equation}\label{eq:series}
    \sum_{p} \bigg(\frac{g^2(p)}{p^2}+\sum_{i \ge 2} \frac{g(p^i)}{p^i}\bigg)<\infty.
    \end{equation}
\item If $\theta\le 1$ then we also impose that for $t \ge 2$,\[\sum_{p, i \ge 2,\, p^i \le t}g(p^i) \ll \frac{t}{\log t}.\]
\item As $x \to \infty$,
\begin{equation}\label{eq:flatsum} \sum_{n \le x} h(n) \sim \frac{e^{-\gamma\theta}}{\Gamma(\theta)}\frac{x}{\log x}\prod_{p \le x} \left( \sum_{i=0}^{\infty} \frac{h(p^i)}{p^i}\right)
    \end{equation}
    holds for every $h$ of the form $h(\cdot)=g(\cdot)\mathbf{1}_{(\cdot,d)=1}$ for some $d \in \NN$.
\end{enumerate}
If $f \in \Ma$ satisfies the conditions in \Cref{thm:main} then necessarily $|f|^2\in \Ma_{\theta}$ as a direct consequence of Wirsing's theorem \cite{Wirsing} (see \Cref{app:mean} for details).

Let $P(1) = 1$, and for $n >1$  denote by $P(n)$ the largest prime factor of $n$. Of particular importance to us is the fact that if $g \in \Ma_{\theta}$ then one can compute the asymptotics of \[\sum_{n \le x,\, P(n)\le x^{\OurEpsilon}} h(n)\]
for any given $\OurEpsilon>0$ and any $h$ of the form $h(\cdot)=g(\cdot)\mathbf{1}_{(\cdot,d)=1}$ as an immediate consequence of a classical result of de Bruijn and van Lint \cite{dBvL} (see \Cref{lem:sum}). 
\subsection{Overview of the proof}\label{subsec:overview}
The proof of \Cref{thm:main} will consist of 4 steps, inspired by \cite{NPS}. Each step will require (in addition to  $f \in \Ma$) slightly different restrictions on $f$. The conditions in \Cref{thm:main} will be easily seen to imply all these restrictions. After describing the steps, we will show how, together with \Cref{lem:mCLT}, they imply \Cref{thm:main}. We shall use the following $\sigma$-algebras:
\begin{align}
\Fa_y := \sigma(\alpha(p), p\le y),\qquad \Fa_{y^-} := \sigma(\alpha(p), p <y),
\end{align}
generated by the random variables $\{\alpha(p), p \le y\}$ and  $\{\alpha(p), p < y\}$, respectively.

\paragraph{Step 1: truncation.} \mbox{}\\
Instead of attacking the problem directly, we introduce a truncation parameter $\OurEpsilon>0$ and define
\[ S_{x,\OurEpsilon} :=\frac{1}{\sqrt{\sum_{n \le x} |f(n)|^2}}\sum_{\substack{n\le x\\ P(n) \ge x^{\OurEpsilon}\\ P(n)^2 \nmid n}} \alpha(n) f(n).\]
We write
\begin{align}\label{eq:eps-truncated-S}
S_{x,\OurEpsilon} = \sum_{x^{\OurEpsilon}\le p\le x} Z_{x,p}, \qquad  Z_{x,p} := \frac{1}{\sqrt{\sum_{n \le x} |f(n)|^2}}\sum_{\substack{\substack{n\le x\\ P(n)=p\\ p^2 \nmid n}}} \alpha(n) f(n).
\end{align}
We achieve some technical simplification in using these truncated sums from the beginning, and the following lemma shows we lose little by doing so.  
\begin{lem}\label{lem:neg}
		Let $\theta>0$. Let $f\in \Ma$ be a function such that $|f|^2 \in \Ma_{\theta}$ and
\begin{equation}\label{eq:ppasump}
\sum_{p,i\ge 2,\, p^i \le x}  |f(p^i)|^2 =o\left( \frac{x}{(\log x)^2}\right)
\end{equation}
as $x \to \infty$.  Then $\limsup_{\OurEpsilon\to 0^+} \limsup_{x \to \infty}\EE\left[|S_x - S_{x,\OurEpsilon}|^2\right]=0$.
\end{lem}
Ultimately, we apply \Cref{lem:mCLT} to $S_{x,\OurEpsilon}$.

\paragraph{Step 2: Lindeberg condition.} \mbox{} \\
The proof of the following lemma will be given in \Cref{sec:Lindeberg}. 
\begin{lem}\label{lem:lind}
	Let $\theta >0$. Let $f\in \Ma$ be a function such that $|f|^2 \in \Ma_{\theta}$ and
\begin{align}
\label{eq:sumpi}&\sum_{p} \bigg(\sum_{i \ge 2} \frac{|f(p^i)|^2i}{p^i}\bigg)^2<\infty,\\
\label{eq:passum}&\sum_{p \le x} \left(|f(p^2)|^2+|f(p)|^4\right) \ll x^2(\log x)^{-2\theta-4}.
\end{align}
 Then $\limsup_{x \to \infty}\sum_{x^{\OurEpsilon}\le p\le x} \EE [|Z_{x,p}|^4] =0$ for every $\OurEpsilon>0$.
\end{lem}
We claim \Cref{lem:lind} implies that for any $\delta > 0$,
\begin{align} \label{eq:checkLind}
    \sum_{x^{\OurEpsilon} \le p \le x} \EE\left[|Z_{x, p}|^2 1_{\{|Z_{x, p}| > \delta \}}| \Fa_{p^-}\right] \xrightarrow[x \to \infty]{p} 0.
\end{align}
\noindent Indeed, the left-hand side of the above  can be upper bounded by
\begin{align*}
     \delta^{-2} \sum_{x^\OurEpsilon \le p \le x} \EE\left[|Z_{x, p}|^4| \Fa_{p^-}\right]
\end{align*}
\noindent using Cauchy--Schwarz, and
\begin{align*}
    \EE \bigg[\sum_{x^\OurEpsilon \le p \le x} \EE\left[|Z_{x, p}|^4| \Fa_{p^-}\right]\bigg] =\sum_{ x^{\OurEpsilon} \le p \le x} \EE\left[|Z_{x, p}|^4\right]
\end{align*}
\noindent vanishes when we first send $x \to \infty$ and then $\OurEpsilon \to 0^+$ by \Cref{lem:lind}. This shows that the convergence \eqref{eq:checkLind} holds in $L^1$, and in particular in probability. Observe \eqref{eq:checkLind} corresponds to the Lindeberg condition in \Cref{lem:mCLT}.

\paragraph{Step 3: approximating the conditional variance.}
\mbox{}\\
Let
\[ T_{x,\OurEpsilon} :=\sum_{x^{\OurEpsilon}\le p\le x} \EE\left[ |Z_{x,p}|^2 \mid \Fa_{p^-}\right]\]
and
\begin{align}\label{eq:Uxdefinition}
U_x := \bigg(\sum_{n\le x}|f(n)|^2\left(\frac{1}{n}-\frac{1}{x}\right)\bigg)^{-1}\int_{1}^{x} \frac{|s_t|^2}{t}dt, \qquad  s_t := \frac{1}{\sqrt{t}} \sum_{n\le t}  \alpha(n)f(n).
\end{align}
Given $\theta>0$, define $\rho_{\theta} \colon (0,\infty)\to (0,\infty)$ via $\rho_{\theta}(t)=t^{\theta-1}/\Gamma(\theta)$ for $t\le 1$ and by
\[t\rho_{\theta}(t)=\theta \int_{t-1}^{t}\rho_{\theta}(v)dv\]
for $t>1$. See Smida \cite{Smida} for an asymptotic investigation of $\rho_{\theta}$, which tends rapidly to $0$. Let
\begin{equation}\label{eq:Ceps}
C_{\OurEpsilon}:=\theta \int_{\OurEpsilon}^{1} \frac{(1-v)^{\theta-1}}{v} \frac{\Gamma(\theta)\rho_{\theta}\left( \frac{1-v}{v}\right)}{((1-v)/v)^{\theta-1}}dv.
\end{equation}
It was shown in \cite[Equation~(4.38)]{NPS} that $\lim_{\OurEpsilon\to 0^+}C_{\OurEpsilon}=1$.
\begin{proposition}\label{prop:l2}
Let $\theta \in (0,\frac{1}{2})$. Let $f\in \Ma$ be a function such that $|f|^2 \in \Ma_{\theta}$. Suppose there exists $c>0$ such that $|f(p)|=O(p^{\frac{1}{2}-c})$  and $|f(p^k)|^2= O(2^{k(1-c)})$ for all $k\ge 2$ and primes  $p$. Then $\lim_{x \to \infty}\EE[ |T_{x,\OurEpsilon}-C_{\OurEpsilon} U_x|^2] =0$.
\end{proposition}
When establishing \Cref{prop:l2} we shall assume $x$ is sufficiently large  as this implies that the denominator $\sum_{n \le x} |f(n)|^2(1/n - 1/x)\ge 0$ in $U_x$ is not $0$. 
The denominator was chosen as it forces \[\EE [U_x] \equiv 1.\]
Indeed, $\EE [|s_t|^2 ] = \sum_{n \le t} |f(n)|^2/t$ by \eqref{eq:orth} and so 
\[ \EE \left[\int_{1}^{x} \frac{|s_t|^2}{t} dt\right] = \sum_{n \le x} |f(n)|^2  \int_{n}^{x} \frac{dt}{t^2} = \sum_{n \le x} |f(n)|^2 \left(\frac{1}{n}-\frac{1}{x}\right).\]
\paragraph{Step 4: convergence of conditional variance.}
\mbox{}\\
To conclude the proof of \Cref{thm:main}, we need to identify the limit of $U_x$ which is the subject of \Cref{sec:mc-element}. As we shall see in \Cref{sec:mc-element}, the task of identifying the limit of $U_x$ is closely related to understanding the random Euler product
\begin{align*}
    A_y(s) := \prod_{p \le y} \bigg(1 + \sum_{k \ge 1} \frac{\alpha(p)^k f(p^k)  }{p^{ks}}\bigg)=\sum_{P(n) \le y} \frac{\alpha(n)f(n)}{n^s},
\end{align*}

\noindent and will require the following probabilistic ingredient.
\begin{thm}\label{thm:mc-convergence}
Let $\theta \in (0, \frac{1}{2})$, and $f\in \Ma$ be a function such that $|f|^2 \in \mathbf{P}_\theta$ and
\begin{align}\label{eq:f-summability}
 \sum_p \bigg(\frac{|f(p)|^3}{p^{3/2}} 
+ \frac{|f(p^2)|^2}{p^{2}} +\sum_{k \ge 3} \frac{|f(p^k)|^2}{p^{k/2}} \bigg) < \infty.
\end{align}

\noindent Write 
\begin{align}\label{eq:mc-measure}
\sigma_t = \tfrac{1}{2}\left(1 + \frac{1}{\log t}\right) \qquad \text{and} \qquad     m_{y, t}(ds) := \frac{|A_y(\sigma_t + is)|^2}{\EE\left[|A_y(\sigma_t + is)|^2\right]}ds, \qquad s \in \RR.
\end{align}

\noindent Then there exists a nontrivial random Radon measure $m_\infty(ds)$ on $\RR$  such that the following are true: for any bounded interval $I$ and any test function $h \in C(I)$, we have
\begin{align}\label{eq:thm:mc-convergence}
   (i) \quad m_{y, \infty}(h) \xrightarrow[y \to \infty]{L^2} m_\infty(h)
    \qquad \text{and} \qquad
  (ii) \quad   m_{\infty, t}(h) \xrightarrow[t \to \infty]{L^2} m_\infty(h)
\end{align}

\noindent and in particular the above convergence also holds in probability. Moreover,  the limiting measure $m_\infty(ds)$ is supported on $\RR$ almost surely.
\end{thm}

We will obtain from \Cref{thm:mc-convergence} the following result.
\begin{lem}\label{lem:limit-cond-var}
Under the same setting as \Cref{thm:mc-convergence}, we have
\begin{align}\label{eq:Uxlim}
    U_x \xrightarrow[x \to \infty]{p} \frac{1}{2\pi} \int_{\RR} \frac{m_\infty(ds)}{|\frac{1}{2} + is|^2}
\end{align}
\noindent and the limit is almost surely finite and strictly positive. 
\end{lem}
For the moment let us assume the validity of all the results presented in the current section. We are ready to explain:
\begin{proof}[Proof of \Cref{thm:main}]
Let $\OurEpsilon > 0$ be fixed. Since $Z_{x, p}$ is linear in $\alpha(p)$ by construction, for any $x > 0$ we automatically obtain $\sum_{x^{\OurEpsilon} \le p \le x} \EE[Z_{x, p}^2 | \Fa_{p^-}] =  0$ almost surely. Moreover, we have
\begin{align*}
    T_{x,\OurEpsilon} :=\sum_{x^{\OurEpsilon}\le p\le x} \EE\left[ |Z_{x,p}|^2 \mid \Fa_{p^-}\right]
    \xrightarrow[x \to \infty]{p} \frac{C_\OurEpsilon}{2\pi} \int_{\RR} \frac{m_\infty(ds)}{|\tfrac{1}{2} + is|^2} = C_\OurEpsilon V_\infty
\end{align*}

\noindent by \Cref{prop:l2} and \Cref{lem:limit-cond-var}. Combining this with \Cref{lem:lind} (which implies the conditional Lindeberg condition), we see that the martingale sequence associated to $S_{x, \OurEpsilon}$ satisfies all assumptions in \Cref{lem:mCLT}, and thus
$S_{x, \OurEpsilon} \xrightarrow[x \to \infty]{d}  \sqrt{C_\OurEpsilon V_\infty} G$ where the distributional convergence is also stable.

To establish the stable convergence of $S_{x}$, it is helpful to use the formulation \eqref{eq:stable3} in the language of weak convergence: we would like to show that
\begin{align*} 
    \lim_{x \to \infty} \EE\left[Y \widetilde{h}(S_x)\right]
    = \EE\left[Y\widetilde{h}(\sqrt{V_\infty} G) \right]
\end{align*}

\noindent for any bounded $\Fa_\infty$-measurable random variable $Y$ and any bounded continuous function $\widetilde{h}\colon \CC \to \RR$. By a density argument, it suffices to establish this claim for bounded Lipschitz functions $\widetilde{h}$. Consider
\begin{align*}
    \left|\EE\left[Y \widetilde{h}(S_x)\right] - 
    \EE\left[Y\widetilde{h}(\sqrt{V_\infty} G) \right]\right|
    & \le \left|\EE\left[Y \widetilde{h}(S_{x, \OurEpsilon})\right] - 
    \EE\left[Y\widetilde{h}(\sqrt{C_\OurEpsilon V_\infty} G) \right]\right|\\
    & \qquad + \EE\left[Y \left|\widetilde{h}(S_x) - \widetilde{h}(S_{x, \OurEpsilon})\right|\right]
    + \EE\left[Y \left|\widetilde{h}(\sqrt{V_\infty} G) - \widetilde{h}(\sqrt{C_\OurEpsilon V_\infty} G)\right|\right].
\end{align*}

\noindent We see that the first term on the right-hand side converges to $0$ as $x \to \infty$ as a consequence of the stable convergence of $S_{x, \OurEpsilon}$, whereas the third term converges to $0$ as $\OurEpsilon \to 0^+$ as a consequence of dominated convergence. Meanwhile, the middle term is bounded by
\begin{align*}
    \|\widetilde{h}\|_{\mathrm{Lip}} \EE[Y^2]^{1/2} 
     \EE\left[\left|S_x - S_{x, \OurEpsilon}\right|^2\right]^{\frac{1}{2}}
\end{align*}

\noindent which also vanishes in the limit as $x \to \infty$ and then $\OurEpsilon \to 0^+$ by \Cref{lem:neg}. This completes the proof of \Cref{thm:main}.
\end{proof}
\paragraph{Outline.} The remainder of the article is organised as follows.
\begin{itemize}
    \item In \Cref{sec:physical}, we shall carry out the first three steps, proving \Cref{lem:neg}, \Cref{lem:lind} and \Cref{prop:l2}.
    \item  In \Cref{sec:mc-element}, we shall carry out the fourth step, establishing \Cref{thm:mc-convergence} and \Cref{lem:limit-cond-var}. Towards the end we shall also provide the proof of \Cref{cor:mom-convergence}, and discuss extension to larger values of $q$.
    \item In \Cref{sec:rad}, we treat the Rademacher case, and highlight the changes needed for the proof of \Cref{thm:mainRad} and \Cref{cor:mom-convergenceRad}.
    \item Finally, we collect various results about mean values of nonnegative multiplicative functions and abstract probability ingredients in the two appendices.
\end{itemize}

\section{Proofs of \texorpdfstring{\Cref{lem:neg}}{Lemma \ref{lem:neg}}, \texorpdfstring{\Cref{lem:lind}}{Lemma \ref{lem:lind}} and \texorpdfstring{\Cref{prop:l2}}{Proposition \ref{prop:l2}}}\label{sec:physical}
\subsection{Truncation: proof of \texorpdfstring{\Cref{lem:neg}}{Lemma \ref{lem:neg}}}
By \eqref{eq:orth},
\[ \EE\left[|S_x - S_{x,\OurEpsilon}|^2\right] =(\sum_{n\le x} |f(n)|^2)^{-1} \sum_{\substack{n \le x\\ P(n) < x^{\OurEpsilon} \text{ or } P(n)^2 \mid n}} |f(n)|^2 = A_{1,\OurEpsilon} + A_{2,\OurEpsilon}\]
where
\begin{equation}\label{eq:A12}
A_{1,\OurEpsilon}:=(\sum_{n\le x} |f(n)|^2)^{-1} \sum_{\substack{n \le x\\ P(n) < x^{\OurEpsilon}}}|f(n)|^2 ,\qquad
A_{2,\OurEpsilon}:=(\sum_{n\le x} |f(n)|^2)^{-1} \sum_{\substack{n \le x\\ P(n) \ge x^{\OurEpsilon} \\ P(n)^2 \mid n}}|f(n)|^2.
\end{equation}
By \Cref{lem:sum} with $g=|f|^2$,
\[ \limsup_{x \to \infty} A_{1,\OurEpsilon} = \Gamma(\theta)\rho_{\theta}(1/\OurEpsilon)\OurEpsilon^{\theta-1}\]
which tends to $0$ as $\OurEpsilon\to 0^+$ \cite{Smida}. In particular, $\limsup_{\OurEpsilon\to 0^+}\limsup_{x \to \infty} A_{1,\OurEpsilon} =0 $.\footnote{By making slightly different assumptions we may avoid \Cref{lem:sum}. Applying \Cref{lem:upper} with $g(\cdot) =|f(\cdot)|^2 \mathbf{1}_{P(\cdot) \le x^{\OurEpsilon}}$ we obtain
$\sum_{n \le x,\, P(n) \le x^{\OurEpsilon}}|f(n)|^2\ll \tfrac{x}{\log x}\prod_{p \le x^{\OurEpsilon}} ( \sum_{i=0}^{\infty} |f(p^i)|^2/p^i)$; this upper bound and \eqref{eq:flatsum} with $g=|f|^2$ imply
 $A_{1,\OurEpsilon} \ll \exp(-\sum_{x^{\OurEpsilon}<p\le x} |f(p)|^2/p) \ll \OurEpsilon^{\theta}$ if $x$ is large enough in terms of $\OurEpsilon$.}
 To estimate $A_{2,\OurEpsilon}$ we denote $P(n)$ by $p$, and the multiplicity of $P(n)$ in $n$ by $i$ (i.e., $P(n)^i \mid n$ but $P(n)^{i+1} \nmid i$), so
\begin{equation}	\label{eq:A2}
	\begin{split}
	A_{2,\OurEpsilon} &= (\sum_{n \le x}|f(n)|^2)^{-1} \sum_{p \ge x^{\OurEpsilon}} \sum_{i\ge 2} |f(p^i)|^2 \sum_{\substack{m\le x/p^i \\ P(m)<p}}|f(m)|^2\\
	&\le(\sum_{n \le x}|f(n)|^2)^{-1} \sum_{p \ge x^{\OurEpsilon}} \sum_{i\ge 2} |f(p^i)|^2 \sum_{m\le x/p^i}|f(m)|^2
	\end{split}
\end{equation}
where $m$ stands for $n/p^i$. Considering $p^i \le \sqrt{x}$ and $p^i>\sqrt{x}$ separately in \eqref{eq:A2}, we deduce using \eqref{eq:flatsum} applied twice (with $f=|g|^2$) that
\[ 	A_{2,\OurEpsilon}\ll \sum_{\substack{p \ge x^{\OurEpsilon},\,i \ge 2\\ p^i \le \sqrt{x}}}  \frac{|f(p^i)|^2}{p^i} + \log x\sum_{\substack{p \ge x^{\OurEpsilon},\,i \ge 2\\ x\ge p^i > \sqrt{x}}} \frac{|f(p^i)|^2}{p^i}.\]
We bound the last sums using \eqref{eq:ppasump}, finding that for any fixed $\OurEpsilon>0$, $\limsup_{x \to \infty} A_{2,\OurEpsilon} =0 $.
\subsection{Lindeberg condition: proof of \texorpdfstring{\Cref{lem:lind}}{Lemma \ref{lem:lind}}}
\label{sec:Lindeberg}
By \eqref{eq:orth},
\begin{equation}\label{eq:4thsum}
\sum_{x^{\OurEpsilon}\le p\le x} \EE\left[|Z_{x,p}|^4\right]=(\sum_{n\le x}|f(n)|^2)^{-2}\sum_{x^{\OurEpsilon}\le p\le x}  \sum_{\substack{ab=cd\\ a,b,c,d \le x \\ P(a)=P(b)=P(c)=P(d)=p\\ P(a)^2 \nmid a,b,c,d}} f(a)f(b)\overline{f(c)f(d)}.
\end{equation}
The inner sum in the right-hand side of \eqref{eq:4thsum} can be written as 
\[\sum_{\substack{ab=cd\\ a,b,c,d \le x \\ P(a)=P(b)=P(c)=P(d)=p \\P(a)^2 \nmid a,b,c,d}} f(a)f(b)\overline{f(c)f(d)}=\sum_{\substack{m \le x^2\\P(m)=p}}h(m), \qquad h(m):=\bigg|\sum_{\substack{ab=m\\ P(a)=P(b)\\a,b\le x\\ P(m)^2 \nmid a,b}}f(a)f(b)\bigg|^2.\]
We write $p^2 \mid \mid m$ to indicate that $p^2$ divides $m$ but $p^3$ does not. We have the pointwise bound
$h(m) \le (|f|*|f|)^2(m) \cdot \mathbf{1}_{P(m)^2 \mid\mid  m}$ so that
\begin{equation*} \sum_{\substack{ab=cd\\ a,b,c,d \le x \\ P(a)=P(b)=P(c)=P(d)=p\\ P(a)^2 \nmid a,b,c,d}} f(a)f(b)\overline{f(c)f(d)} \le \sum_{\substack{m \le x^2\\P(m)=p\\ P(m)^2 \mid\mid  m}}\fname(m), \qquad \fname := (|f|*|f|)^2.
\end{equation*}
It follows that 
\begin{align}\label{eq:magicupper}
\sum_{x^{\OurEpsilon}\le p\le x}  \sum_{\substack{ab=cd\\ a,b,c,d \le x \\ P(a)=P(b)=P(c)=P(d)=p\\ P(a)^2 \nmid a,b,c,d}} f(a)f(b)\overline{f(c)f(d)}\le \sum_{x^{\OurEpsilon}\le p\le x}\sum_{\substack{p^2 \mid \mid m \le x^2\\ P(m)\le p}} \fname(m) = \sum_{x^{\OurEpsilon}\le p \le x}\fname(p^2) \sum_{\substack{m' \le x^2/p^2\\ P(m')<p}}\fname(m')
\end{align}
where $m'$ stands for $m/p^2$. Next,
\begin{equation}\label{eq:trivupper}
\sum_{\substack{m' \le x^2/p^2\\ P(m')<p}}\fname(m') \le \sum_{m' \le x^2/p^2}\fname(m') \le \frac{x^2}{p^2} \sum_{m'\le x^2/p^2} \frac{\fname(m')}{m'} \le \frac{x^2}{p^2} \prod_{q\le x^2/p^2} \left(\sum_{i=0}^{\infty} \frac{\fname(q^i)}{q^i}\right).
\end{equation}
Our assumptions on $|f(p)|$ and $|f(p^i)|$ in \eqref{eq:series} (with $g=|f|^2$) and \eqref{eq:sumpi} imply $\sum_{q,i \ge 2} \fname(q^i)/q^i < \infty$, and so
\begin{equation}\label{eq:nohigh}
\prod_{q\le x^2/p^2} \left(\sum_{i=0}^{\infty} \frac{\fname(q^i)}{q^i}\right)\ll \prod_{q \le x^2/p^2} \left(1+ \frac{\fname(q)}{q}\right) .
\end{equation}
Since $\fname(q) = 4|f(q)|^2$, we have
\[\prod_{q \le x^2/p^2} \bigg(1+\frac{\fname(q)}{q}\bigg) \le \prod_{q \le x^2} \bigg(1+\frac{\fname(q)}{q}\bigg)  = (\log x)^{4\theta+o(1)}\]
by \eqref{eq:loglim} and \eqref{eq:series} (with $g=|f|^2$). Thus, \eqref{eq:trivupper} and \eqref{eq:nohigh} imply
 \begin{align}\label{eq:Hb}
 	\sum_{x^{\OurEpsilon}\le p \le x}\fname(p^2) \sum_{\substack{m' \le x^2/p^2\\ P(m')<p}}\fname(m')\ll  x^2 (\log x)^{4\theta+o(1)}\sum_{x^{\OurEpsilon}\le p \le x}\frac{\fname(p^2)}{p^2}.
 \end{align}
By definition of $\fname$, $\fname(p^2) \ll |f(p^2)|^2 + |f(p)|^4$. Our assumptions on $f(p)$ and $f(p^2)$ in \eqref{eq:passum} show that the $p$-sum in the right-hand side of \eqref{eq:Hb} is $\ll_{\OurEpsilon} (\log x)^{-2\theta-3}$. We conclude by plugging \eqref{eq:Hb} in \eqref{eq:magicupper} and then \eqref{eq:magicupper} in \eqref{eq:4thsum}, and noting $\sum_{n \le x} |f(n)|^2 = x (\log x)^{\theta-1+o(1)}$ by \eqref{eq:flatsum} with $h=|f|^2$ and \Cref{lem:slow}.

\subsection{Approximating conditional variance: proof of \texorpdfstring{\Cref{prop:l2}}{Proposition \ref{prop:l2}}}\label{subsec:l2computation}
Given $f \in \Ma$ we define
\[ \Lsum:= \sum_{(n_1,n_2)=1} \frac{1}{\max\{n_1,n_2\}^2} \prod_{i=1}^{2} \bigg|\sum_{g_i \mid n_i^{\infty}} \frac{\overline{f(g_i)}f(g_in_i)}{g_i} \bigg|^2 \bigg(\sum_{h \mid (n_1n_2)^{\infty} }\frac{|f^2(h)|}{h} \bigg)^{-2},\]
where the outer sum is over all pairs of coprime positive integers $n_1$ and $n_2$. The notation `$g \mid n^{\infty}$' indicates that $g$ ranges over all positive integers whose prime factors divide $n$, e.g., if $n=p$ then $g$ ranges over all powers of $p$. If $f$ is completely multiplicative then
\[ \Lsum = \sum_{(n_1,n_2)=1} \frac{|f(n_1)|^2 |f(n_2)|^2}{\max\{n_1,n_2\}^2}.\]
The proof of \Cref{prop:l2} is broken into three parts:
\begin{lem}\label{lem:conv}
	Let $\theta \in (0,\frac{1}{2})$. Suppose $f \in \Ma$ satisfies $|f|^2 \in \Ma_{\theta}$, and that
 \begin{equation}\label{eq:Ldom}
		\Lsum':=\sum_{(n_1,n_2)=1} \frac{1}{\max\{n_1,n_2\}^2} \prod_{i=1}^{2} \bigg(\sum_{g_i \mid n_i^{\infty}} \frac{|\overline{f(g_i)}f(g_in_i)|}{g_i} \bigg)^2 \bigg(\sum_{h \mid (n_1n_2)^{\infty} }\frac{|f^2(h)|}{h} \bigg)^{-2}
	\end{equation}
converges.  Then, as $x \to \infty$, 
	\begin{align}
\label{eq:1stpart}		\EE[ U_x^2] &\to \Lsum,\\
\label{eq:2ndpart}		\EE [T_{x,\OurEpsilon} U_x] &\to C_{\OurEpsilon}\Lsum,\\
\label{eq:3rdpart}		\EE [T_{x,\OurEpsilon}^2] &\to C^2_{\OurEpsilon}\Lsum.
	\end{align}
\end{lem}
Note that $\Lsum'$ dominates $\Lsum$. Before embarking on the proof of \Cref{lem:conv}, we explain why \eqref{eq:Ldom} converges when the assumptions of \Cref{prop:l2} hold. In fact, we estimate the tail of \eqref{eq:Ldom}. Let
\[ \widetilde{f}(n) := \sum_{g \mid n^{\infty}} \frac{|\overline{f(g)}f(gn)|}{g}.\]
By the triangle inequality and the fact that $f(1)=1$,
\[ \sum_{\substack{(n_1,n_2)=1\\ \max\{n_1,n_2\} \ge T}} \frac{1}{\max\{n_1,n_2\}^2} \prod_{i=1}^{2} \big(\sum_{g_i \mid n_i^{\infty}} \frac{|\overline{f(g_i)}f(g_in_i)|}{g_i}\big)^2 \big(\sum_{h \mid (n_1n_2)^{\infty} }\frac{|f^2(h)|}{h} \big)^{-2} \le  \sum_{\max\{n_1,n_2\} \ge T} \frac{(\widetilde{f}(n_1)\widetilde{f}(n_2))^2}{\max\{n_1,n_2\}^2}.\]
By Cauchy--Schwarz, 
\begin{equation}\label{eq:cs}
\widetilde{f}(p^i)^2 \ll |f(p^i)|^2 + \sum_{j\ge 1}\frac{|f(p^j)|^2}{p^j} \sum_{j\ge 1}\frac{|f(p^{i+j})|^2}{p^{j}}, \quad (\widetilde{f}(p) - |f(p)|)^2\le \sum_{j \ge 1} \frac{|f(p^j)|^2}{p^j }\sum_{j \ge 1} \frac{|f(p^{j+1})|^2}{p^{j}}.
\end{equation}
From \eqref{eq:cs} we see that our assumptions on $f$ imply
\begin{equation}\label{eq:assumpf}
\begin{split}
&\sum_{p \le T} \widetilde{f}(p)^2 \log p \ll T \qquad (T \ge 2),\qquad \sum_{p \le T} \frac{\widetilde{f}(p)^2\log p}{p} \sim \theta \log T \qquad (T \to \infty),\\
&\sum_{p, i \ge 2} \frac{\widetilde{f}(p^i)^2\log(p^i)}{p^i}<\infty.
\end{split}
\end{equation}
By \Cref{lem:upper} and \eqref{eq:assumpf} we have
\begin{equation}\label{eq:tildsum}
\sum_{n \le T} \widetilde{f}(n)^2 \ll \frac{T}{\log T}\sum_{n \le T} \frac{\widetilde{f}(n)^2}{n}\le \frac{T}{\log T}\prod_{p \le T} \sum_{i=0}^{\infty}\bigg(\frac{\widetilde{f}(p^i)^2}{p^i}\bigg) \ll\frac{T}{\log T}\prod_{p \le T} \bigg(1+ \frac{\widetilde{f}(p)^2}{p}\bigg) \ll T(\log T)^{\theta -1+o(1)}
\end{equation}
uniformly for $T \ge 2$, and so
\[
\sum_{\max\{n_1,n_2\} \ge T} \frac{(\widetilde{f}(n_1)\widetilde{f}(n_2))^2}{\max\{n_1,n_2\}^2}\ll  \sum_{n \ge T} \frac{\widetilde{f}(n)^2 (\log (n+1))^{\theta-1+o(1)}}{n} \ll (\log T)^{2\theta-1+o(1)}\]
where we used \eqref{eq:tildsum} and the fact that $\theta<\tfrac{1}{2}$. Here $n$ stands for $\max\{n_1,n_2\}$.

\subsubsection{Proof of first part of \texorpdfstring{\Cref{lem:conv}}{Lemma \ref{lem:conv}}}
We introduce
\begin{equation}\label{eq:Fmx}
 F_m(x):= \sum_{\substack{h \le x\\(h,m)=1}} |f(h)|^2\left(\frac{1}{h}-\frac{1}{x}\right),\qquad S_1(x,m,d):= \frac{F_m(x/d)}{F_1(x)}.
\end{equation}
The proof of \eqref{eq:1stpart} relies on the following identity which holds for \textit{all} multiplicative functions.
\begin{lem}\label{lem:iden1}
Let $f \in \Ma$. We have
\begin{equation}\label{eq:Uxrep}
\EE [U_x^2] = \sum_{\substack{n_i \le x\\(n_1,n_2)=1}} \frac{1}{\max\{n_1,n_2\}^2}\left| \sum_{g_{i}\mid n_i^{\infty}} \frac{\overline{f(g_1)}f(g_{1}n_1)f(g_2)\overline{f(g_2n_2)}}{g_{1}g_{2}} S_1(x,n_1n_2,\max\{n_1,n_2\}g_1g_2)\right|^2.
\end{equation}
\end{lem}
\begin{proof}
By definition,
\[\EE [U_x^2 ]=F_1(x)^{-2} S, \qquad S:=\int_{1}^{x}\int_{1}^{x} \frac{\EE\left[|s_t|^2|s_r|^2\right]}{tr}dtdr.\]
By \eqref{eq:orth},
\[\EE \left[|s_t|^2|s_r|^2 \right]=\frac{1}{tr} \sum_{\substack{n_1,n_2 \le t \\ m_1, m_2 \le r\\ n_1 m_1 = n_2 m_2}}f(n_1)\overline{f(n_2)}f(m_1)\overline{f(m_2)}.\]
It follows that
\begin{align*}
S &= \sum_{\substack{n_1,n_2,m_1,m_2 \le x \\ n_1 m_1 = n_2 m_2}} f(n_1)f(m_1)\overline{f(n_2)f(m_2)} \int_{\max\{n_1,n_2\}}^{x} \int_{\max\{m_1,m_2\}}^{x}\frac{dtdr}{t^2r^2}\\
&= \sum_{\substack{n_1,n_2,m_1,m_2 \le x \\ n_1 m_1 = n_2 m_2}} f(n_1)f(m_1)\overline{f(n_2)f(m_2)} \left( \frac{1}{\max\{n_1,n_2\}}-\frac{1}{x}\right) \left( \frac{1}{\max\{m_1,m_2\}}-\frac{1}{x} \right).
\end{align*}
We parametrize the solutions to $n_1 m_1 = n_2 m_2$. We let $g_1:=(n_1,n_2)$ and $n'_i:=n_i/g_1$ ($i=1,2$). Since $(n'_1,n'_2)=1$ and $n'_1m_1=n'_2 m_2$ we must have $m_1 = g_2 n'_2$ and $m_2 = g_2 n'_1$ for some $g_2$. Conversely, given coprime $n'_1$ and $n'_2$, and arbitrary $g_1,g_2$, we can construct a solution to $n_1m_1=n_2m_2$ via $n_i=g_1n'_i$, $m_1 = g_2 n'_2$ and $m_2=g_2 n'_1$. So 
\[ S =\sum_{\substack{n'_1,n'_2,g_1,g_2\\(n'_1,n'_2)=1 \\ n'_1, n'_2 \le x/\max\{g_1,g_2\}}} \frac{f(g_1n'_1)f(g_2n'_2)\overline{f(g_2n'_1)f(g_1n'_2)}}{g_1 g_2} \prod_{j=1}^{2}\left( \frac{1}{\max\{n'_1,n'_2\}}-\frac{g_j}{x}\right).\]
For $1\le i,j\le 2$ we denote $g_{i,j}=(g_i,n^{\prime \infty}_j)$, i.e., the largest divisor of $g_i$ whose prime factors divide $n_j$. Since $(n'_1,n'_2)=1$ we have $(g_{i,1},g_{i,2})=1$ for $i=1,2$. We introduce $h_i = g_i/(g_{i,1}g_{i,2})$, which is the largest divisor of $g_i$ coprime to $n'_1n'_2$. Observe that
\begin{equation}\label{eq:factor}
f(g_i n'_1) =f(g_{i,1}n'_1) f(g_{i,2})f( h_i),\qquad f(g_i n'_2) =f(g_{i,2}n'_2) f(g_{i,1})f( h_i),
\end{equation}
for $i=1,2$, by multiplicativity of $f$. 
The sum $S$ transforms into
\begin{equation}\label{eq:toexp}
 S=\sum_{\substack{n'_i \le x\\(n'_1,n'_2)=1}} \frac{1}{\max\{n'_1,n'_2\}^2}\left| \sum_{g_{1,i}\mid n_i^{\prime \infty}} \frac{\overline{f(g_{1,1})}f(g_{1,1}n'_1)f(g_{1,2})\overline{f(g_{1,2}n'_2)}}{g_{1,1}g_{1,2}} F_{n_1n_2}\left(\frac{x}{\max\{n'_1,n'_2\}g_{1,1}g_{1,2}}\right)\right|^2
\end{equation}
where in each of the sums, $i$ ranges over $\{1,2\}$. Replacing the notation $n'_i$ with $n_i$, and $g_{1,i}$ with $g_i$, and dividing \eqref{eq:toexp} by $F_1(x)^2$ we conclude the proof.
 \end{proof}
We proceed to prove \eqref{eq:1stpart}, where $|f|^2 \in \Ma_{\theta}$ is assumed.  Note that $S_1(x,m,d)$ vanishes if $d \ge x$. Let 
\begin{equation}\label{eq:Lprime}
 \Lname_m:= \prod_{p \mid m} \bigg(\sum_{i=0}^{\infty} \frac{|f(p^i)|^2}{p^i}\bigg)^{-1} = \bigg(\sum_{h \mid m^{\infty}} \frac{|f^2(h)|}{h}\bigg)^{-1}.
 \end{equation} 
\noindent We make two claims: 
\begin{itemize}
    \item $S_1(x,m,d)$ is bounded uniformly for $d \le x$ and $m \ge 1$, and 
    \item $\lim_{x \to \infty} S_1(x,m,d)= \Lname_m$ for any fixed $m$ and $d$. 
    \end{itemize}
Given these claims, the required result will follow by taking $x$ to $\infty$ in \eqref{eq:Uxrep} and using dominated convergence (recall that \eqref{eq:Ldom} is finite by assumption). We proceed to establish the claims. Given $x_2 \ge x_1 \ge 1$ we necessarily have
\begin{equation}\label{eq:triv1}
0\le F_m(x_1) \le F_1(x_1) \le \sum_{n \le x_1 } \frac{|f(n)|^2}{n} \le \sum_{n \le x_2} \frac{|f(n)|^2}{n}.
\end{equation}
Moreover,
\begin{equation}\label{eq:triv2}
F_1(x) \asymp \sum_{n \le x} \frac{|f(n)|^2}{n}
\end{equation}
by \eqref{eq:flatsum} and \Cref{lem:log} (with $g=|f|^2$). So
\[ 0\le S_1(x,m,d) \le F_1(x/d)/F_1(x) \ll   1, \] 
proving the first claim. Similarly, \eqref{eq:flatsum} and \Cref{lem:log} imply that
\begin{equation}\label{eq:logapp} \sum_{n \le x}g(n) \left(\frac{1}{n}-\frac{1}{x}\right) \sim  \sum_{n \le x}\frac{g(n) }{n} \sim \frac{e^{-\gamma\theta }}{\Gamma(\theta+1)} \prod_{p \le x} \left( \sum_{i=0}^{\infty} \frac{g(p^i) }{p^i}\right)
\end{equation}
holds as $x \to \infty$ for all functions of the form $g(\cdot)=|f(\cdot)|^2\mathbf{1}_{(\cdot,m)=1}$ for some $m$. By two applications of \eqref{eq:logapp}, once with $g=|f|^2$ and a second time with $g(\cdot)=|f(\cdot)|^2 \mathbf{1}_{(\cdot,m)=1}$ and $x/d$ in place of $x$, we obtain
\begin{equation}\label{eq:S1lim}
\lim_{x \to \infty} S_1(x,m,d) =\lim_{x \to \infty} \prod_{x/d < p \le x} \left( \sum_{i=0}^{\infty} \frac{|f(p^i)|^2}{p^i}\right)^{-1} \prod_{p \mid m} \bigg(\sum_i \frac{|f(p^i)|^2}{p^i}\bigg)^{-1}=\prod_{p \mid m} \bigg(\sum_i \frac{|f(p^i)|^2}{p^i}\bigg)^{-1},
\end{equation}
proving the second claim.
\subsubsection{Proof of second part of \texorpdfstring{\Cref{lem:conv}}{Lemma \ref{lem:conv}}}
The proof of \eqref{eq:2ndpart} is based on the following lemma. We use the notation introduced in \eqref{eq:Fmx}.
\begin{lem}\label{lem:TxUxS2}
Let $f \in \Ma$. We have
\begin{equation}\label{eq:TxUxS2}
\begin{split}
\EE [T_{x,\OurEpsilon}U_x]&=\sum_{\substack{n_i\le x \\ (n_1,n_2)=1}} \frac{1}{\max\{n_1,n_2\}^2}  \sum_{\substack{g_{1,i}\mid n_i^{\infty}\\ g_{2,i}\mid n_i^{\infty}}} \frac{\overline{f(g_{1,1})}f(g_{1,1}n_1)f(g_{1,2})\overline{f(g_{1,2}n_2)}f(g_{2,1})\overline{f(g_{2,1}n_1)}\overline{f(g_{2,2})}f(g_{2,2}n_2)}{g_{1,1}g_{1,2}g_{2,1}g_{2,2}}\\ &\qquad \cdot S_2(x,n_1n_2,\max\{n_1,n_2\}g_{1,1}g_{1,2},\max\{n_1,n_2\}g_{2,1}g_{2,2}),
\end{split}
\end{equation}
where
\begin{align}
S_2(x,m,d_1,d_2) &:= S_{2,1}(x,m,d_1)S_{2,2}(x,m,d_2),\\ S_{2,1}(x,m,d) &:=   F_m(x/d)/F_1(x) = S_1(x,m,d),\\
S_{2,2}(x,m,d) &:= (\sum_{n\le x} |f(n)|^2)^{-1} \sum_{\substack{h \le x^{1-\OurEpsilon}/d\\(h,m)=1}} |f^2(h)|d \sum_{P(hm)+1,x^{\OurEpsilon}\le p \le \frac{x}{dh}}|f^2(p)|.
\end{align}
\end{lem}
\begin{proof}
Let $p$ be a prime. We have the following variant of \eqref{eq:orth}, proved in the same way:
\begin{equation}\label{eq:orthF}
	\EE \left[\alpha(n)\overline{\alpha}(m) \mid \Fa_{p^-}\right]=\alpha(n')\overline{\alpha}(m')\delta_{n'',m''}
\end{equation}
where $n'n''=n$, $m'm''$, and $n'$ and $m'$ are the largest divisors of $n$ and $m$ with $P(n'),P(m') \le p-1$. By \eqref{eq:orthF},
\begin{equation}\label{eq:zpsquared}
\begin{split}
\EE\left[ |Z_{x,p}|^2 \mid \Fa_{p^-}\right]&=\frac{1}{\sum_{n \le x}|f(n)|^2} \sum_{\substack{m_1,m_2 \le x\\ P(m_1)=P(m_2)=p\\p^2 \nmid m_1,m_2}} f(m_1)\overline{f(m_2)}\EE \left[\alpha(m_1)\overline{\alpha}(m_2) \mid \Fa_{p^-}\right]\\
&=\frac{1}{\sum_{n \le x}|f(n)|^2} |f^2(p)| \sum_{\substack{m_i' \le x/p\\ P(m'_i)\le p-1}} f(m'_1)\alpha(m'_1) \overline{f(m'_2)\alpha(m'_2)},
\end{split}
\end{equation}
where $i$ ranges over $\{1,2\}$, and the $m'_i$ stand for $m_i/p$. By definition,
\begin{equation}\label{eq:ETU}
\EE [T_{x,\OurEpsilon} U_x] = F_1(x)^{-1}\sum_{p \le x}\int_{1}^{x} \frac{1}{t} \EE \left[|s_t|^2 \EE \left[|Z_{x,p}|^2 \mid \Fa_{p^-}\right]\right] dt.
\end{equation}
By \eqref{eq:orth} and \eqref{eq:zpsquared},
\begin{equation}\label{eq:mixed}
\begin{split}
\EE &\left[|s_t|^2 \EE\left[|Z_{x,p}|^2 \mid \Fa_{p^-}\right]\right] \\
&=\frac{1}{t\sum_{n \le x}|f(n)^2|} \sum_{n_i \le t}\EE \bigg[f(n_1)\alpha(n_1)\overline{f(n_2)\alpha(n_2)}\sum_{\substack{m'_i \le x/p\\ P(m'_i) \le p-1}}|f^2(p)| f(m'_1)\alpha(m'_1)\overline{f(m'_2)\alpha(m'_2)}\bigg]\\ 
&=\frac{1}{t\sum_{n \le x}|f(n)|^2} |f^2(p)|\sum_{\substack{n_i\le t,\, m'_i\le x/p\\ P(m'_i)\le p-1\\ n_1m'_1=n_2 m'_2}} f(n_1)f( m'_1)\overline{f(n_2)f(m'_2)}.
\end{split}
\end{equation}
From \eqref{eq:ETU} and \eqref{eq:mixed},
\begin{align}\label{eq:etu} \EE[ T_{x,\OurEpsilon} U_x ]=F_1(x)^{-1} ( \sum_{n \le x} |f(n)|^2)^{-1} S
	\end{align}
for
\begin{equation*}
\begin{split}
S&:= \sum_{x^{\OurEpsilon}\le p \le x}|f^2(p)|\int_{1}^{x} t^{-2}\sum_{\substack{n_i \le t ,\,  m'_i \le x/p\\ P(m'_i)\le p-1\\ n_1m'_1=n_2 m'_2}} f(n_1)f(m'_1)\overline{f(n_2)f(m'_2)}dt\\
&= \sum_{x^{\OurEpsilon} \le p \le x}|f^2(p)|\sum_{\substack{n_i\le x,\, m'_i \le x/p\\ P(m'_i)\le p-1\\ n_1m'_1=n_2 m'_2}} f(n_1)f(m'_1)\overline{f(n_2)f(m'_2)} \left(\frac{1}{\max\{n_1,n_2\}} - \frac{1}{x}\right).
\end{split}
\end{equation*}
We now parametrize the solutions to $n_1m'_1=n_2m'_2$. Let $n'_i = n_i/g_1$ ($i=1,2$) where $g_1=(n_1,n_2)$. Then $m'_2 = g_2 n'_1$, $m'_1 = g_2 n'_2$ for $g_2= (m'_1,m'_2)$. Note that $P(n'_1)$, $P(n'_2)$ and $P(g_2)$ are $\le p-1$ if $P(m'_1)$, $P(m'_2)$ are $\le p-1$. Conversely, given $n'_1,n'_2,g_2$ such that $P(n'_1)$, $P(n'_2)$ and $P(g_2)$ are $\le p-1$ and a positive integer $g_1$ we can construct $n_i=n'_i g_1$, $m'_1= n'_2 g_2$ and $m'_2= n'_1g_2$ that will satisfy $n_1m'_1=n_2m'_2$ and $P(m'_i)\le p-1$. Using this parametrization, $S$ transforms into
\[S=\sum_{x^{\OurEpsilon} \le p \le x}|f^2(p)|\sum_{\substack{g_1 n'_i  \le x\\g_2 n'_i \le x/p\\ P(n'_i),P(g_2) \le p-1\\ (n'_1,n'_2)=1}} f(g_1n_1')f(g_2n'_2)\overline{f(g_2 n'_1 )f(g_1 n_2')} \left(\frac{1}{g_1\max\{n'_1,n'_2\}} - \frac{1}{x}\right).\]
We use the notation $g_{i,j}=(g_i,n^{\prime \infty}_j)$ and $h_i = g_i/(g_{i,1}g_{i,2})$, as in the proof of \Cref{lem:iden1}, as well as the notation $F_m$ defined in \eqref{eq:Fmx}. By \eqref{eq:factor}, the sum $S$ transforms into
\begin{equation}\label{eq:Stransform}
\begin{split}
	S=\sum_{x^{\OurEpsilon}\le p \le x}|f^2(p)|&\sum_{\substack{n'_i \le x\\P(n'_i)\le p-1\\(n'_1,n'_2)=1}} \frac{1}{\max\{n'_1,n'_2\}} \\
	& \sum_{g_{1,i}\mid n_i^{\prime \infty}} \frac{\overline{f(g_{1,1})}f(g_{1,1}n'_1)f(g_{1,2})\overline{f(g_{1,2}n'_2)}}{g_{1,1}g_{1,2}}F_{n'_1n'_2}\left(\frac{x}{\max\{n'_1,n'_2\}g_{1,1}g_{1,2}}\right) \\  &\sum_{g_{2,i}\mid n_i^{\prime \infty}}f(g_{2,1}) \overline{f(g_{2,1}n'_1)}\overline{f(g_{2,2})}f(g_{2,2}n'_2)\sum_{\substack{h_2 \le x/(\max\{n'_1,n'_2\}pg_{2,1}g_{2,2})\\(h_2,n'_1n'_2)=1\\ P(h_2)\le p-1}}|f(h_2)|^2
 \end{split}
\end{equation}
where $i$ ranges over $\{1,2\}$. Replacing the notation $n'_i$ with $n_i$ and $h_2$ with $h$, we obtain  \eqref{eq:TxUxS2} from \eqref{eq:Stransform} and \eqref{eq:etu}.
\end{proof}
We now establish \eqref{eq:2ndpart}, where $|f|^2 \in \Ma_{\theta}$ is assumed. Observe that $S_2(x,m,d_1,d_2)$ vanishes if $d_1>x$ or $d_2 >x^{1-\OurEpsilon}$.  Recall $\Lname_m$ was defined in \eqref{eq:Lprime}. We make two claims:
\begin{itemize}
\item $S_2(x,m,d_1,d_2) = O_{\OurEpsilon}(1)$ uniformly when $d_1 \le x$, $d_2 \le  x^{1-\OurEpsilon}$ hold simultaneously,
\item $\lim_{x \to \infty} S_2(x,m,d_1,d_2) = C_{\OurEpsilon}\Lname_m^{2}$ for every fixed $m,d_1,d_2$.
\end{itemize}
Given these claims, the required limit will follow by taking $x$ to $\infty$ in \eqref{eq:TxUxS2} and using dominated convergence. To establish boundedness, we consider $S_{2,1}$ and $S_{2,2}$ individually. We already saw $S_{2,1} \ll 1$ (by \eqref{eq:triv1}--\eqref{eq:triv2}). To bound $S_{2,2}$, we first bound the inner sum over $p$. The bound $\sum_{p \le t} |f(p)|^2 \ll t/\log (t+1)$ yields
\begin{align}
S_{2,2}(x,m,d) &\ll x(\sum_{n \le x}|f(n)|^2)^{-1} \sum_{ h \le x^{1-\OurEpsilon}/d}\frac{|f^2(h)| }{h \log(x/(dh) +1)}  \\
&\ll \frac{x}{\OurEpsilon\log x}(\sum_{n \le x}|f(n)|^2)^{-1} F_1(x^{1-\OurEpsilon}/d)\ll  \frac{x }{\OurEpsilon\log x}(\sum_{n \le x}|f(n)|^2)^{-1}F_1(x).
\end{align}
By \eqref{eq:flatsum} and \Cref{lem:log} with $g=|f|^2$ , $\sum_{n \le x} |f(n)|^2 \asymp x F_1(x)/\log x$ and so $S_{2,2}(x,m,d) \ll 1/\OurEpsilon$.

We now estimate $S_2$ asymptotically when $m$, $d_1$ and $d_2$ are fixed. We already estimated $S_{2,1}=S_1$ in \eqref{eq:S1lim}:
\begin{equation}
\lim_{x \to \infty} S_{2,1}(x,m,d) =\Lname_m.
\end{equation}
It remains to show $S_{2,2}$ tends to $C_{\OurEpsilon} \Lname_m$. By interchanging the order of summation we find that
\begin{equation}\label{eq:S2simp} S_{2,2}(x,m,d) =  \sum_{x^{\OurEpsilon} \le p \le \frac{x}{d}} |f^2(p)| \sum_{\substack{h \le x/(pd)\\ P(h)\le p-1\\(h,m)=1}} |f^2(h)|d/ (\sum_{n \le x}|f(n)|^2) 
\end{equation}
if $x$ is large enough with respect to $m$ and $\OurEpsilon$. Let $\delta>0$ be a parameter that eventually will be sent to $0^+$. We write
\[ S_{2,2}= S_{2,2,\le 1- \delta}+S_{2,2,>1- \delta}\]
where $S_{2,2,\le 1-\delta}$ is the contribution of $x^{\OurEpsilon} \le p \le x^{1-\delta}$ to \eqref{eq:S2simp}, while  $S_{2,2,>1- \delta}$ is the contribution of $x^{1-\delta}<p \le x/d$.  We define
\[ \widetilde{\rho_{\theta}}(t) :=  \Gamma(\theta)\rho_{\theta}(t)/t^{\theta-1}.\] 
To estimate the $h$-sum in \eqref{eq:S2simp} we apply \Cref{lem:sum} with $g(\cdot)=|f(\cdot)|^2\mathbf{1}_{(\cdot,m)=1}$, which  gives
\begin{equation}\label{eq:hsum} \sum_{\substack{h \le x/(pd) \\(h,m)=1 \\P(h)\le p-1}} |f^2(h)| \sim \widetilde{\rho_{\theta}}\left(\frac{\log(x/(pd))}{\log(p-1)}\right) \sum_{n \le x/(pd),\, (n,m) = 1} |f^2(n)|
\end{equation}
as $x \to \infty$, uniformly for $x^{\OurEpsilon} \le p \le x^{1-\delta}$. \Cref{lem:slow} and \eqref{eq:flatsum} applied twice, with $g=|f|^2$ and $g(\cdot)=|f(\cdot)|^2 \mathbf{1}_{(\cdot,m)=1}$, imply that
\begin{equation}\label{eq:shortvslong} \frac{\sum_{n \le x/(pd),\, (n,m)=1} |f^2(n)| }{\sum_{n \le x} |f^2(n)|} \sim \frac{\Lname_m}{pd}  \left(\frac{\log(1+x/(pd))}{\log (1+x)}\right)^{\theta-1} \frac{L(\log (x/(pd)))}{L(\log x)}
\end{equation}
as $x/(pd) \to \infty$, for an explicit slowly varying function $L=L_{|f^2|}\colon [0,\infty) \to (0,\infty)$ described in \Cref{lem:slow}. Recall that for any given $b \in (0,1)$, \begin{equation}\label{eq:unif}
L(tu)\sim L(u)
\end{equation}
holds as $u \to \infty$, \textit{uniformly} for $t \in [b,1/b]$. This can be deduced directly from the description of $L$ in \Cref{lem:slow}, but is also a general property of slowly varying functions \cite[p.~180]{Korevaar}. From \eqref{eq:hsum}, \eqref{eq:shortvslong} and \eqref{eq:unif},
\begin{equation}\label{eq:S2asymp}
S_{2,2,\le 1-\delta} =(1+o_{\delta}(1))\Lname_m \sum_{x^{\OurEpsilon} \le p \le x^{1-\delta}} \frac{|f^2(p)|\log p}{p}  \frac{( 1-\log p/\log x)^{1-\theta}}{\log p} \widetilde{\rho_{\theta}}\bigg( \frac{\log x}{\log p}-1\bigg),
\end{equation}
where, for fixed $\delta$, $o_{\delta}(1)\to 0$ as $x \to \infty$ by \eqref{eq:unif}. By integration by parts, the $p$-sum in \eqref{eq:S2asymp} tends to
\[ \theta \int_{x^{\OurEpsilon}}^{x^{1-\delta}} \frac{(1-\log t/\log x)^{1-\theta}}{t\log t} \widetilde{\rho_{\theta}}\bigg(\frac{\log x}{\log t}-1\bigg) dt  =\theta \int_{\OurEpsilon}^{1-\delta}  \frac{(1-v)^{1-\theta}}{v} \widetilde{\rho_{\theta}}\bigg(\frac{1-v}{v}\bigg)dv=C_{\OurEpsilon}+O(\delta) \]
as $x \to \infty$, 
where the implied constant in $O(\delta)$ depends on $\theta$ alone.
It remains to control $S_{2,2,>1-\delta}$.
Omitting the conditions $(h,m)=1$ and $P(h) \le p-1$ in \eqref{eq:S2simp} and using  \Cref{lem:slow} and \eqref{eq:flatsum} twice with $g=|f|^2$, we find that 
\begin{equation}\label{eq:S221minus} S_{2,2,>1-\delta}\ll \sum_{x^{1-\delta} < p \le \frac{x}{d}} \frac{|f^2(p)|}{p} \bigg( \frac{\log (1+x/(pd))}{\log (1+x)}\bigg)^{\theta-1} \frac{L(\log(x/(pd)))}{L(\log x)}. 
\end{equation}
For any given $a >0$, there exists $C_a>0$ such that
\begin{equation}\label{eq:Lratio} L(u)/L(v)\le (v/u)^a
\end{equation}
holds for all $v \ge u \ge C_a$ (this follows from the description of $L$ in \Cref{lem:slow}, and is also a general property of slowly varying functions \cite[p.~180]{Korevaar}). From \eqref{eq:S221minus} and \eqref{eq:Lratio},
\begin{equation}\label{eq:junk}  S_{2,2,>1-\delta} \ll  \sum_{x^{1-\delta} < p \le \frac{x}{d}} \frac{|f^2(p)|}{p} \bigg( \frac{\log (1+x/(pd))}{\log (1+x)}\bigg)^{\frac{\theta}{2}-1} \ll  \max\left\{\delta,\frac{1}{\log x}\right\}^{\frac{\theta}{2}}
\end{equation}
where we used \eqref{eq:Lratio} with $a=\theta/2$ in the first inequality, and $\sum_{p \le t} |f(p)|^2 \ll t/\log (t+1)$ in the second inequality.   Our estimates on $S_{2,2,\le 1-\delta}$ and $S_{2,2,>1-\delta}$ imply $\lim_{x \to \infty} S_{2,2}(x,m,d) = C_{\OurEpsilon} \Lname_m$ as needed.
\subsubsection{Proof of third part of \texorpdfstring{\Cref{lem:conv}}{Lemma \ref{lem:conv}}}
We start with the following general identity.
\begin{lem}\label{lem:txsquared}
Let $f \in \Ma$. We have
\begin{align}\label{eq:Txexp}
\EE [|T_{x,\OurEpsilon}|^2]&= \sum_{\substack{n_i \le x\\(n_1,n_2)=1}} \frac{1}{\max\{n_1,n_2\}^2}\left| \sum_{g_{i} \mid n^{\infty}_i}\frac{\overline{f(g_{1})}f(g_{1}n_1)f(g_{2})\overline{f(g_{2}n_2)}}{g_{1}g_2} S_3(x,n_1n_2,\max\{n_1,n_2\}g_{1}g_{2})\right|^2
\end{align}
where
\begin{align} 
\label{eq:S3def}S_3(x,m,d) &:= \sum_{\substack{\max\{x^{\OurEpsilon},P(m)+1\} \le p \le x/d}} \frac{|f^2(p)|}{p}  T_m\left(x,\frac{x}{pd},p\right),\\
\notag T_m(x,y,p) &:= \frac{\sum_{ h \le y,\,  P(h)\le p-1,\, (h,m)=1}|f^2(h)|/y}{\sum_{n \le x}|f(n)|^2/x}.
 \end{align}
\end{lem}
\begin{proof}
By definition,
\begin{equation}\label{eq:Tsquared}
\EE[ |T_{x,\OurEpsilon}|^2 ]=  \sum_{x^{\OurEpsilon}\le p,q \le x} \EE \left[\EE\left[ |Z_{x,p}|^2 \mid \Fa_{p^-}\right] \EE\left[ |Z_{x,q}|^2 \mid \Fa_{q^-}\right]\right].
\end{equation}
Given primes $p,q\le x$ we have, by \eqref{eq:zpsquared} and \eqref{eq:orth},
\begin{equation*}
\begin{split}
\EE \big[\EE\left[ |Z_{x,p}|^2 \mid \Fa_{p^-}\right] &\EE\left[ |Z_{x,q}|^2 \mid \Fa_{q^-}\right]\big] \\
&= \frac{|f^2(p)f^2(q)|}{(\sum_{n \le x}|f(n)|^2)^2}  \sum_{\substack{m'_i \le x/p,\, m''_i \le x/q \\ P(m'_i) \le p-1,\, P(m''_i)\le q-1\\ m'_1 m''_1 = m'_2 m''_2}}  f(m'_1)f(m''_1)\overline{f(m'_2)f(m''_2)}.
\end{split}
\end{equation*} 
Here $i$ ranges over $\{1,2\}$. Suppose without loss of generality that $p \le q$. We parametrize the solutions to $m'_1 m''_1 = m'_2 m''_2$: we let $g_1=(m'_1,m'_2)$ and introduce a variable $g_2$ with $P(g_2)\le q-1$, and then given $m'_i$ we parametrize $m''_i$ by $m''_1 = g_2 (m'_2/g_1)$, $m''_2 = g_2 (m'_1/g_1)$. Denoting $m'_i/g_1$ by $n'_i$, we can write
\begin{multline}\label{eq:epq1}
\EE \left[\EE\left[ |Z_{x,p}|^2 \mid \Fa_{p^-}\right] \EE\left[ |Z_{x,q}|^2 \mid \Fa_{q^-}\right]\right] = \frac{ |f^2(p)f^2( q)|}{(\sum_{n \le x}|f(n)|^2)^2} \sum_{\substack{P(n'_ig_1)\le p-1\\ P(g_2)\le q-1\\ (n'_1,n'_2)=1 \\ n'_i g_1 \le x/p \\ n'_i g_2 \le x/q}} f(g_1 n'_1)f(g_2 n'_2)\overline{f(g_1 n'_2 ) f(g_2 n'_1)}.
\end{multline}
We use the notation $g_{i,j}=(g_i,n^{\prime \infty}_j)$ and $h_i = g_i/(g_{i,1}g_{i,2})$ as in the proof of \Cref{lem:iden1}. By \eqref{eq:factor}, \eqref{eq:epq1} transforms into
\begin{equation}\label{eq:epq}
	\begin{split}
\EE \big[\EE\left[ |Z_{x,p}|^2 \mid \Fa_{p^-}\right]& \EE\left[ |Z_{x,q}|^2 \mid \Fa_{q^-}\right]\big] = \frac{|f^2(p)f^2( q)|}{(\sum_{n \le x}|f(n)|^2)^2} \\
& \sum_{\substack{P(n'_i)\le p-1\\  n'_i\le x/q\\(n'_1,n'_2)=1 }}\overline{f(g_{1,1})}f(g_{1,1}n'_1)f(g_{1,2})\overline{f(g_{1,2}n'_2)}f(g_{2,1})\overline{f(g_{2,1}n'_1)}\overline{f(g_{2,2})}f(g_{2,2}n'_2)\\
 & \sum_{\substack{P(h_1) \le p-1 \\ h_1 \le x/(p\max\{n'_1,n'_2\} g_{1,1}g_{1,2})\\(h_1,n'_1n'_2)=1}} |f(h_1)|^2     \sum_{\substack{P(h_2) \le q-1 \\ h_2 \le x/(q\max\{n'_1,n'_2\} g_{2,1}g_{2,2})\\(h_2,n'_1n'_2)=1}} |f(h_2)|^2.
 \end{split}
\end{equation}
Renaming $n'_i$ to $n_i$ and $g_{1,i}$ to $g_i$, we obtain from \eqref{eq:Tsquared} and \eqref{eq:epq} the required identity.
\end{proof}
We proceed to prove \eqref{eq:3rdpart}, where $|f|^2 \in \Ma_{\theta}$ is assumed. If $d>x$ then clearly $S_3(x,m,d)$ vanishes. We make two claims:
\begin{itemize}
    \item $S_3(x,m,d)=O_{\OurEpsilon}(1)$ holds uniformly for $d\le x$ and $m \ge 1$, and
    \item $\lim_{x \to \infty} S_3(x,m,d) =  C_{\OurEpsilon}\Lname_m$ for every fixed $m$ and $d$. 
\end{itemize}
Given these claims, the required result will follow by taking $x$ to $\infty$ in \eqref{eq:Txexp} and using dominated convergence.  We first explain boundedness. By omitting the conditions $P(h)\le p-1$ and $(h,m)=1$ in $T_m(x,y,p)$ and using \eqref{eq:flatsum} (with $g=|f|^2$), we find that
\begin{align*}
S_3(x,m,d) &\ll \sum_{x^{\OurEpsilon}\le p \le x/d} \frac{|f^2(p)|}{p} \frac{L(\log x -\log (pd))}{L(\log x)}\left(1-\frac{\log(pd)-1}{\log x}\right)^{\theta-1}\\
&\ll \sum_{x^{\OurEpsilon}\le p \le x/d} \frac{|f^2(p)|}{p} \left(1-\frac{\log(pd)-1}{\log x}\right)^{\frac{\theta}{2} - 1}
\end{align*}
for a slowly varying function $L=L_{|f^2|}\colon [0,\infty)\to (0,\infty)$ described in \Cref{lem:slow}; in the last inequality we made use of \eqref{eq:Lratio} with $a= \theta/2$. We continue by considering the interval in which $p$ lies:
\begin{align*}S_3(x,m,d) &\ll (\log x)^{1-\frac{\theta}{2} }\sum_{ 1 \le e^k \le x^{1-\OurEpsilon}/d} (k+1)^{\frac{\theta}{2} - 1} \sum_{p \in [x/(de^{k+1}),x/(de^{k})]} \frac{|f^2(p)|}{p}\\
& \ll (\log x)^{1-\frac{\theta}{2}}\sum_{1 \le e^k \le x^{1-\OurEpsilon}/d} \frac{(k+1)^{\frac{\theta}{2} - 1}}{\log (1+x/(de^k))}\ll \frac{1}{\OurEpsilon}.
\end{align*}
To estimate $S_3$ asymptotically for fixed $m$ and $d$, we introduce $\delta > 0$ and let $S_{3,\le 1-\delta}$ and $S_{3,>1-\delta}$ be the corresponding contributions of $\max\{x^{\OurEpsilon},P(m)+1\} \le p \le x^{1-\delta}$ and $x^{1-\delta} \le p \le x/d$ to \eqref{eq:S3def}.

To study $S_{3,\le 1-\delta}$ we estimate $T_m(x,x/(pd),p)$ asymptotically using \Cref{lem:sum} with $h(\cdot) =g(\cdot)\mathbf{1}_{(\cdot,m)=1}$ and \eqref{eq:flatsum} with $h=g$ and $h(\cdot) =g(\cdot)\mathbf{1}_{(\cdot,m)=1}$. We obtain that $S_{3,\le 1-\delta}$ is asymptotic to the right-hand side of \eqref{eq:S2asymp}, whose limit was already computed: 
\[\lim_{x \to \infty} S_{3,\le 1-\delta}(x,m,d)=\lim_{x \to \infty} S_{2,2,\le 1-\delta}(x,m,d)=\Lname_m \bigg( \theta \int_{\OurEpsilon}^{1-\delta}  \frac{(1-v)^{1-\theta}}{v} \widetilde{\rho_{\theta}}\bigg(\frac{1-v}{v}\bigg)dv\bigg) = \Lname_m( C_{\OurEpsilon}+O(\delta)). \]
It remains to control $S_{3,>1-\delta}$. Omitting the conditions $(h,m)=1$ and $P(h) \le p-1$ in $T_m(x,x/(pd),p)$ and using  \Cref{lem:slow} and \eqref{eq:flatsum} twice with $g=|f|^2$, we find that $S_{3,>1-\delta}$ is bounded in terms of the sum in the right-hand side of \eqref{eq:junk} which was already estimated there as $\ll \max\{\delta, 1/\log x\}^{\theta/2}$.
\section{Convergence of conditional variance} \label{sec:mc-element}
\subsection{Proof of \texorpdfstring{\Cref{thm:mc-convergence}}{Theorem \ref{thm:mc-convergence}}}
\subsubsection{Some elementary estimates}
We need to collect a few estimates before proceeding to the construction of the random measure $m_\infty$.
\begin{lem}\label{lem:truncate}
Let $f\colon \NN \to \CC$ be a function satisfying \eqref{eq:f-summability} and $|f|^2 \in \mathbf{P}_\theta$ for some $\theta >0$. Then there exists some deterministic $y_0 = y_0(f) > 0$ such that
\begin{align}\label{eq:truncate}
& \left|1 + \sum_{k =1}^2 \frac{\alpha(p)^k f(p^k)}{p^{z}}\right|^{-2} \left|1 + \sum_{k \ge 3} \frac{\alpha(p)^k f(p^k)}{p^{z}}\right|^2
= 
1 + \asymptoticO \left(  \left|\frac{f(p)}{p^{1/2}}\right|^3 + \left|\frac{f(p^2)}{p}\right|^3 +  \sum_{k \ge 3} \left|\frac{f(p^k)}{p^{k/2}}\right| \right)
\end{align}

\noindent where the implicit constants on the right-hand side are uniform in $p \ge y_0$, $\Re z \ge \frac{1}{2}$, and also for any sequence $(\alpha(p))_{p}$ satisfying $\sup_p |\alpha(p)| \le 1$. In particular, there exists some deterministic constant $C = C(y_0) \in (0, \infty)$ such that
\begin{align}\label{eq:truncate-product}
\prod_{p \ge y_0} \left|1 + \sum_{k \ge 3} \frac{\alpha(p)^k f(p^k)}{p^{kz}}\right|^2
\le C \prod_{p \ge y_0} \left|1 + \sum_{k = 1}^2 \frac{\alpha(p)^k f(p^k)}{p^{kz}}\right|^2
\end{align}

\noindent uniformly in $\Re z \ge \frac{1}{2}$.

\end{lem}

\begin{proof}
Condition \eqref{eq:f-summability} implies that there exists some $\OurEpsilon > 0$ and some $y_0> 0$  such that for any $p \ge y_0$, one has $\sum_{k \ge 1} p^{-k/2} \left|f(p^k) \right| \le 1 - \OurEpsilon$. Consider the elementary estimate
\begin{align*}
(1+A)^{-1} (1+A+B) 
& =[1 - A + A^2 + \asymptoticO(|A|^3)] (1+A+B) \\
%&= 1 - A + A^2   + O(A^3) + A + B + [-A + A^2 + O(A^3)][A+B]\\
& = 1 + B - AB + \asymptoticO(|A|^3 + |A|^2|B|)
= 1 + \asymptoticO(|A|^3 + |B|)
 \qquad \text{for $|A|+|B| \le 1 - \OurEpsilon$}.
\end{align*}

\noindent Applying this with
\begin{align*}
A = \frac{\alpha(p) f(p)}{p^{z}} + \frac{\alpha(p)^2 f(p^2)}{p^{2z}},
\qquad B = \sum_{k \ge 3} \frac{\alpha(p)^k f(p^k)}{p^{kz}}
\end{align*}

\noindent for $p \ge y_0$, we have $|A| \le  |p^{-1/2} f(p)| + |p^{-1} f(p^2)|$ and $|B| \le \sum_{k \ge 3} |p^{-k/2} f(p^k)|$ and thus \eqref{eq:truncate} holds  uniformly for all $z \in \CC$ satisfying $\Re z \ge \frac{1}{2}$ (and the bound can be made deterministic).  Using the fact that $|f(p)^2| \le p$ for any $p \ge y_0$,  we have
\begin{align*}
\sum_{p \ge y_0} \left[ \left|\frac{f(p)}{p^{1/2}}\right|^3 + \left|\frac{f(p^2)}{p}\right|^3 +  \sum_{k \ge 3} \left|\frac{f(p^k)}{p^{k/2}}\right| \right]
\le  \sum_{p \ge y_0} \left[ \left|\frac{f(p)}{p^{1/2}}\right|^3 + \left|\frac{f(p^2)}{p}\right|^2 +  \sum_{k \ge 3} \left|\frac{f(p^k)}{p^{k/2}}\right| \right] \ll 1
\end{align*}

\noindent by the assumption \eqref{eq:f-summability}. In particular, the infinite random product
\begin{align*}
\prod_{p \ge p_0} \left|1 + \sum_{k =1}^2 \frac{\alpha(p)^k f(p^k)}{p^{z}}\right|^{-2} \left|1 + \sum_{k \ge 1} \frac{\alpha(p)^k f(p^k)}{p^{z}}\right|^2
\end{align*}

\noindent exists and its value may be bounded from above uniformly in $\Re z \ge \frac{1}{2}$ by some absolute constant $C(y_0)$. This concludes the proof.
\end{proof}

\begin{lem}\label{lem:cross-moments}
Let $f \colon \NN \to \CC$. Let $z_j = \sigma_{t_j} + is_j$ for $j=1, 2$. Then
\begin{align}\label{eq:1-point}
\EE\left[ \left|1 + \sum_{k \ge 1} \frac{\alpha(p)^k f(p^k)}{p^{kz_1}}\right|^2\right]
& = \exp\left( \frac{|f(p)|^2}{p^{2\sigma_{t_1}}} + \asymptoticO\left(\frac{|f(p)|^4}{p^{2}} + \sum_{k \ge 2} \frac{|f(p^k)|^2}{p^{k}}\right) \right)
\end{align}

\noindent and
\begin{equation}\label{eq:2-point}
\begin{split}
&\EE\left[ \left|1 + \sum_{k =1}^2 \frac{\alpha(p)^k f(p^k)}{p^{kz_1}}\right|^2 \left|1 + \sum_{k =1}^2 \frac{\alpha(p)^k f(p^k)}{p^{kz_2}}\right|^2\right]\\
& = \exp \left( \frac{|f(p)|^2}{p^{2\sigma_{t_1}}} + \frac{|f(p)|^2}{p^{2\sigma_{t_2}}} + 2 \frac{|f(p)|^2}{p^{\sigma_{t_1} + \sigma_{t_2}}} \cos(|s_1 - s_2| \log p) + \asymptoticO \left( \frac{|f(p^2)|^2}{p^{2}} + \sum_{k=1}^2 \left[\left|\frac{f(p^k)}{p^{k/2}}\right|^3 + \left|\frac{f(p^k)}{p^{k/2}}\right|^4 \right]\right) \right)
\end{split}
\end{equation}

\noindent uniformly for $t_j \ge 2$, $s_j \in \RR$ and $p\ge 2$.
\end{lem}

\begin{proof}
Expanding the square, we have
\begin{equation}\label{eq:1st-mom-expand}
\begin{split}
\left|1 + \sum_{k \ge 1} \frac{\alpha(p)^k f(p^k)}{p^{kz_1}}\right|^2
& =  1 + 2\Re \sum_{k \ge 1} \frac{\alpha(p)^k f(p^k)}{p^{kz_1}} + \sum_{j, k \ge 1} 
\frac{\alpha(p)^j f(p^j)}{p^{jz_1}}
\frac{\overline{\alpha(p)^k f(p^k)}}{p^{k\overline{z_1}}}.
\end{split}
\end{equation}

\noindent Evaluating expectation term-by-term (by absolute convergence and Fubini's theorem), we obtain
\begin{equation}\label{eq:full-expansion}
\begin{split}
\EE\left[ \left|1 + \sum_{k \ge 1} \frac{\alpha(p)^k f(p^k)}{p^{kz_1}}\right|^2\right]
& = 1 + \sum_{k \ge 1} \frac{|f(p^k)|^2}{p^{2k\sigma_{t_1}}}
%= \exp\left( \frac{|f(p)|^2}{p^{2\sigma_{t_1}}}\right) \left(1 + \asymptoticO\left(\sum_{k \ge 2} \frac{|f(p^k)|^2}{p^{k}}\right) \right)
\end{split}
\end{equation}

\noindent which gives \eqref{eq:1-point} with the desired uniformity. As for the second claim, consider
\begin{align*}
&\left|1 + \sum_{k =1}^2 \frac{\alpha(p)^k f(p^k)}{p^{kz_1}}\right|^2\left|1 + \sum_{k =1}^2 \frac{\alpha(p)^k f(p^k)}{p^{kz_2}}\right|^2\\
& =  \left(1 + 2\Re \sum_{k = 1}^2 \frac{\alpha(p)^k f(p^k)}{p^{kz_1}} + \sum_{j, k = 1} ^2
\frac{\alpha(p)^j f(p^j)}{p^{jz_1}}
\frac{\overline{\alpha(p)^k f(p^k)}}{p^{k\overline{z_1}}}\right)\\
& \qquad \qquad \qquad \qquad  \times
\left(1 + 2\Re \sum_{k = 1}^2 \frac{\alpha(p)^k f(p^k)}{p^{kz_2}} + \sum_{j, k = 1} ^2
\frac{\alpha(p)^j f(p^j)}{p^{jz_2}}
\frac{\overline{\alpha(p)^k f(p^k)}}{p^{k\overline{z_2}}}\right)\\
& = 1 
+ 2\Re \left[\sum_{k = 1}^2 \left(\frac{1}{p^{kz_1}}+\frac{1}{p^{kz_2}} \right) \alpha(p)^k f(p^k)\right]
+ \sum_{j, k = 1} ^2
\left(\frac{1}{p^{jz_1 + k\overline{z_1}}} + \frac{1}{p^{jz_2 + k\overline{z_2}}}\right)
\alpha(p)^j f(p^j)\overline{\alpha(p)^k f(p^k)}\\
& \qquad + 4\left(\Re \sum_{k = 1}^2 \frac{\alpha(p)^k f(p^k)}{p^{kz_1}}\right) \left(\Re \sum_{k = 1}^2 \frac{\alpha(p)^k f(p^k)}{p^{kz_2}}\right)
+ \asymptoticO\left(
\sum_{k=1}^2 \left[\left|\frac{f(p^k)}{p^{k/2}}\right|^3 + \left|\frac{f(p^k)}{p^{k/2}}\right|^4 \right]
\right)
\end{align*}

\noindent where the implicit constant in the error term is uniform in $t_j \ge 2, s_j \in \RR, p \ge 2$, and for any sequence $(\alpha(p))_p$ satisfying $|\alpha(p)| \le 1$.
Taking $\alpha$ to be the Steinhaus random multiplicative function, we obtain 
\begin{align*}
&\EE \left[ \left|1 + \sum_{k =1}^2 \frac{\alpha(p)^k f(p^k)}{p^{kz_1}}\right|^2\left|1 + \sum_{k =1}^2 \frac{\alpha(p)^k f(p^k)}{p^{kz_2}}\right|^2\right]\\
%& = \EE\left[ 1 + \sum_{k=1}^2 \left(\frac{|f(p^k)|^2 }{p^{k(z_1 + \overline{z_1})}} + \frac{|f(p^k)|^2 }{p^{k(z_2 + \overline{z_2})}} + 4\left(\Re \frac{f(p^k) \alpha(p)^k}{p^{kz_1}}\right)\left(\Re \frac{f(p^k) \alpha(p)^k}{p^{kz_2}}\right)\right) \right] +\asymptoticO\left(
%\sum_{k=1}^2 \left[\left|\frac{f(p^k)}{p^{k/2}}\right|^3 + \left|\frac{f(p^k)}{p^{k/2}}\right|^4 \right]
%\right)\\
& = 1 + \sum_{k =1}^2 |f(p^k)|^2\left[ \frac{1}{p^{2k\sigma_{t_1}}} +  \frac{1}{p^{2k\sigma_{t_2}}}\right]
+ 2 \frac{|f(p)|^2}{p^{\sigma_{t_1} + \sigma_{t_2}}} \cos(|s_1 - s_2| \log p)\\
& \qquad  \qquad \qquad  +2 \underbrace{\frac{|f(p^2)|^2}{p^{2(\sigma_{t_1} + \sigma_{t_2})}} \cos(2|s_1 - s_2| \log p)}_{ \ll |f(p^2)|^2 / p^2}
 + \asymptoticO\left(
\sum_{k=1}^2 \left[\left|\frac{f(p^k)}{p^{k/2}}\right|^3 + \left|\frac{f(p^k)}{p^{k/2}}\right|^4 \right]
\right)
% =\exp \left( \frac{|f(p)|^2}{p^{2\sigma_{t_1}}} + \frac{|f(p)|^2}{p^{2\sigma_{t_2}}} + 2 \frac{|f(p)|^2}{p^{\sigma_{t_1} + \sigma_{t_2}}} \cos(|s_1 - s_2| \log p)\right)  \left(1 + \asymptoticO \left( \frac{|f(p)|^4 + |f(p^2)|^2}{p^{2}} \right) \right) 
\end{align*}

\noindent which implies our second estimate \eqref{eq:2-point} with the same uniformity.
\end{proof}
\begin{lem}\label{lem:martingaleL2est}
Let $g\colon \NN \to \CC$ be a function such that $g \in \mathbf{P}_\theta$ for some $\theta \in \CC$. We have
\begin{align}\label{eq:exp-sum}
\sum_{p \le y} \frac{g(p)}{p^{1+\frac{c}{\log y}}} \cos(s \log p) 
= \theta \log\left[\min \left(|s|^{-1}, \log y\right)\right]  + \asymptoticO(1)
\end{align}

\noindent uniformly in $y \ge 3$ and any compact subsets of $s \in \RR$ and $c \ge 0$. Moreover, for each fixed $c \ge 0$ there exists some constant $C_g = C_g(c) \in \CC$ such that
\begin{align} \label{eq:log-sum-twist}
\sum_{p \le y} \frac{g(p)}{p^{1+\frac{c}{\log y}}}
= \theta \log\log y+ C_g + o(1)
\end{align}

\noindent as $y \to \infty$.
\end{lem}
\begin{proof}
Without loss of generality suppose $s \ge 0$, and let $\OurEpsilon = c / \log y$ for convenience. Let us rewrite the left-hand side of \eqref{eq:exp-sum} as
\begin{align*}
& \sum_{p \le y} \frac{g(p)}{p^{1+\OurEpsilon}} \cos(s \log p) 
= \int_{2^-}^{y^+} \frac{\cos(s\log x)}{x^{1+\OurEpsilon}} d\pi_{g}(x) \\
& \qquad = \theta \int_2^y \frac{\cos(s \log x)}{x^{1+\OurEpsilon}} d\mathrm{Li}(x) + \int_{2^-}^{y^+} \frac{\cos(s\log x)}{x^{1+\OurEpsilon}}  d\Ea_g(x).
\end{align*}

\noindent Using integration by parts, it is easy to check that the residual term is equal to
\begin{align}\label{eq:pf-PNT-error}
\frac{\cos(s\log x)}{x^{1+\OurEpsilon}} \Ea_g(x)\bigg|_{2^-}^{y^+} + \int_{2^-}^{y} \Ea_g(x) d \left[ \frac{\cos(s \log x)}{x^{1+\OurEpsilon}}\right]
 = \asymptoticO\left(1 +  \frac{|\Ea_g(y)|}{y} - \int_2^y \frac{|\Ea_g(x)|}{x^2}dx \right),
\end{align}
 
\noindent which is $\asymptoticO(1)$ by the assumption that $g \in \mathbf{P}_\theta$. As for the main term
\begin{align*}
\theta \int_2^y \frac{\cos(s \log x)}{x^{1+\OurEpsilon} \log x}dx
& = \theta \int_{\log 2}^{\log y} \cos(s u) e^{-\OurEpsilon u} \frac{du}{u}
= \theta \int_{\frac{\log 2}{\log y}}^1 \cos(s u \log y) e^{-c u } \frac{du}{u},
\end{align*}
\noindent we shall estimate its size depending on the value of $s$ relative to $y$:
\begin{itemize}
\item When $1/s \ge \log y$, we have
\begin{equation}\label{eq:pf-PNT-error2}
\begin{split}
    \int_{\frac{\log 2}{\log y}}^1 \cos(s u \log y) e^{-c u } \frac{du}{u}
    &= \int_{\frac{\log 2}{\log y}}^1 \frac{du}{u}
    + \int_{\frac{\log 2}{\log y}}^1 \left[\cos(s u \log y) e^{-c u } - e^{-cu} + e^{-cu} - 1\right]\frac{du}{u}\\
    & = -\log \frac{\log 2}{\log y} 
    + \asymptoticO\left(\int_0^1 \left[(su \log y)^2 + cu\right] \frac{du}{u}\right)
    = \log \log y + \asymptoticO(1).
\end{split}
\end{equation}
    \item When $\log 2 < 1/s < \log y $, we have
\begin{equation}\label{eq:alt_bound}
\begin{split}
 \int_{\log 2}^{\log y} \cos(su) e^{-\OurEpsilon u}\frac{du}{u}
& =   \int_{\log 2}^{s^{-1}} \frac{du}{u}
+ \int_{\log 2}^{s^{-1}} \left[ \cos(su)e^{-\OurEpsilon u} - 1\right]\frac{du}{u}
+  \int_{s^{-1}}^{\log y} \frac{\cos(s u) e^{-\OurEpsilon u} du}{u}\\
&=\log\frac{1}{s} - \log \log 2 
+ \asymptoticO\left(\int_{\log 2}^{s^{-1}} [(su)^2 + \OurEpsilon u]\frac{du}{u}\right)
+ 
\int_{1}^{s \log y} e^{-\frac{\OurEpsilon}{s}  v} \cos v\frac{dv}{v}\\
&=\log\frac{1}{s} +
\int_{1}^{s \log y} e^{-\frac{\OurEpsilon}{s}  v} \cos v\frac{dv}{v} + \asymptoticO(1)
\end{split}
\end{equation}

\noindent with the desired uniformity in $y, s$ and $c$. To treat the remaining integral, we use the identity
\begin{align*}
\frac{d}{dx} \frac{e^{-bx}[a\sin(ax) - b \cos(ax)]}{a^2 + b^2}
= e^{-bx} \cos (ax)
\end{align*}

\noindent for any $a, b, x \in \RR$, and consider 
\begin{equation*}
\begin{split}
  \int_{1}^{T} e^{-\frac{\OurEpsilon}{s}  v} \cos v\frac{dv}{v}
& = e^{-\frac{\OurEpsilon}{s}  v}   \frac{ \sin v - \frac{\OurEpsilon}{s} \cos v}{1+\frac{\OurEpsilon^2}{s^2}} \frac{1}{v} \bigg|_1^T + \int_1^T  e^{-\frac{\OurEpsilon}{s}  v}   \frac{ \sin v - \frac{\OurEpsilon}{s} \cos v}{1+\frac{\OurEpsilon^2}{s^2}}\frac{dv}{v^2} 
\end{split}
\end{equation*}

\noindent which is bounded in absolute value uniformly in $s \ge 0$, $\OurEpsilon \ge 0$ and $T \ge 1$. Taking $T = s \log y$, we see that the right-hand side of \eqref{eq:alt_bound} is equal to $\log(1/s) + \asymptoticO(1)$ with the desired uniformity. 

\item When $1/s \le \log 2 < \log y$ (and in particular $\log (1/s) = \asymptoticO(1)$), we may instead consider
\begin{align*}
\int_{\log 2}^{\log y} \cos(su) e^{-\OurEpsilon u}\frac{du}{u}
&=\int_{1}^{s \log y} e^{-\frac{\OurEpsilon}{s}  v} \cos v\frac{dv}{v}
- \int_{1}^{s\log 2} e^{-\frac{\OurEpsilon}{s}  v} \cos v\frac{dv}{v}
\end{align*}

\noindent which is again bounded in absolute value with the desired uniformity.
\end{itemize}
This verifies \eqref{eq:exp-sum}. As for \eqref{eq:log-sum-twist}, one can obtain the improved asymptotics by identifying the constant $C_g$ from \eqref{eq:pf-PNT-error} and \eqref{eq:pf-PNT-error2} when $c = \OurEpsilon = 0$ and $s = 0$. This concludes the proof.
\end{proof}

\begin{lem}\label{lem:martingaleL2est-2}
Let $g\colon \NN \to \CC$ be a function such that $g \in \mathbf{P}_\theta$ for some $\theta \in \CC$. Then
\begin{align*}
&\sum_{p > y}\frac{g(p)\cos(s\log p)}{p^{1+ \frac{1}{\log y}}} = \asymptoticO(1)
\end{align*}

\noindent uniformly in $y \ge 3$ and $s$ in any compact subset of $\RR$. Moreover,
\begin{itemize}
\item For $s$ in any compact subset of $\RR \setminus \{0\}$, we have uniformly
    \begin{align*}
&\sum_{p > y}\frac{g(p)\cos(s\log p)}{p^{1+ \frac{1}{\log y}}} = o(1)
\qquad \text{as $y \to \infty$.}
\end{align*}
\item For $s = 0$, we have
\begin{align*}
    \lim_{y \to \infty} \sum_{p > y}\frac{g(p)}{p^{1+ \frac{1}{\log y}}} = \theta \int_1^\infty e^{-u} \frac{du}{u}.
\end{align*}
\end{itemize} 
\end{lem}
\begin{proof}
We have
\begin{align*}
\sum_{p > y}\frac{g(p)\cos(s\log p)}{p^{1+ \frac{1}{\log y}}}
&= \int_y^\infty \frac{\cos(s\log x)}{x^{1+\frac{1}{\log y}}} d\pi_{g}(x)\\
&=\theta \int_y^\infty \frac{\cos(s\log x)}{x^{1+\frac{1}{\log y}} }\frac{dx}{\log x}
+ \int_y^\infty  \frac{\cos(s\log x)}{x^{1+\frac{1}{\log y}} } d\Ea_g(x)\\
& = \theta\int_{1}^\infty \cos (us\log y)  e^{-u}\frac{du}{u}
- \Ea_g(y) \frac{\cos(s \log y)}{ey} -  \int_y^\infty \Ea_g(x) d \left[  \frac{\cos(s\log x)}{x^{1+\frac{1}{\log y}} } \right]\\
& = \theta\int_{1}^\infty \cos (us\log y)  e^{-u}\frac{du}{u} + \asymptoticO\left(|\Ea_g(y)| / y\right)
+ \asymptoticO\left([1+|s|]\int_y^\infty \frac{|\Ea_g(x)|}{x^2}dx\right)\\
& = \asymptoticO\left(\frac{1}{\max(1, |s| \log y)}\right)
+ \asymptoticO\left(1 / \log y \right)
+ \asymptoticO\left([1+|s|]\int_y^\infty \frac{|\Ea_g(x)|}{x^2}dx\right)
\end{align*}

\noindent and our first two claims follow with the desired uniformity. If $s = 0$, we can identify the limit (as $y \to \infty$) from the first integral in the second last line.
\end{proof}
\subsubsection{Proof of \texorpdfstring{\Cref{thm:mc-convergence}}{Theorem \ref{thm:mc-convergence}}} \label{subsubsec:pf-mc-stein}
\paragraph{Step 1: martingale convergence.}
It is straightforward to check that for any prime number $p_0 \le  y$, 
\begin{align*}
\EE\left[ m_{y, \infty}(h) \bigg| \Fa_{p_0}\right]
& = \int_I h(s) \EE\left[ \frac{|A_y(\frac{1}{2} + is)|^2}{\EE[|A_y(\frac{1}{2} + is)|^2]} \bigg| \Fa_{p_0}\right] ds \\
& =  \int_I h(s) \frac{|A_{p_0}\frac{1}{2} + is)|^2}{\EE[|A_{p_0}(\frac{1}{2} + is)|^2]} ds = m_{p_0, \infty}(h),
\end{align*}

\noindent i.e., $\left(m_{y, \infty}(h)\right)_{y}$ is a martingale with respect to the filtration $(\Fa_y)_y$. Moreover, if we let $y_0 > 0$ be the constant from \Cref{lem:truncate}, then
\begin{align*}
& \frac{\EE\left[A_y(\frac{1}{2} + is_1)|^2A_y(\frac{1}{2} + is_2)|^2\right]}{\EE\left[A_y(\frac{1}{2} + is_1)|^2\right]\EE\left[A_y(\frac{1}{2} + is_2)|^2\right]}
= \prod_{p \le y} \frac{\EE\left[\left|1 + \sum_{k \ge 1} \frac{\alpha(p)^k f(p^k)}{p^{k(1/2 + is_1)}}\right|^2\left|1 + \sum_{k \ge 1} \frac{\alpha(p)^k f(p^k)}{p^{k(1/2 + is_2)}}\right|^2\right]}{\EE\left[\left|1 + \sum_{k \ge 1} \frac{\alpha(p)^k f(p^k)}{p^{k(1/2 + is_1)}}\right|^2\right]\EE\left[\left|1 + \sum_{k \ge 1} \frac{\alpha(p)^k f(p^k)}{p^{k(1/2 + is_2)}}\right|^2\right]}\\
& \qquad \ll \prod_{y_0 \le p \le y} \frac{\EE\left[\left|1 + \sum_{k =1}^2 \frac{\alpha(p)^k f(p^k)}{p^{k(1/2 + is_1)}}\right|^2\left|1 + \sum_{k = 1}^2 \frac{\alpha(p)^k f(p^k)}{p^{k(1/2 + is_2)}}\right|^2\right]}{\EE\left[\left|1 + \sum_{k \ge 1} \frac{\alpha(p)^k f(p^k)}{p^{k(1/2 + is_1)}}\right|^2\right]\EE\left[\left|1 + \sum_{k \ge 1} \frac{\alpha(p)^k f(p^k)}{p^{k(1/2 + is_2)}}\right|^2\right]}\\
& \qquad \ll \exp \left( \sum_{p \le y} 2 \frac{|f(p)|^2}{p} \cos(|s_1 - s_2| \log p)
 + \asymptoticO \left( \frac{|f(p)|^4}{p^{2}} +\sum_{k \ge 2} \frac{|f(p^k)|^2}{p^{k}}\right)
\right) 
\end{align*}

\noindent uniformly in $y \ge y_0$ and $s_1, s_2 \in \RR$ by \Cref{lem:cross-moments}. Observe that the error term in the exponent satisfies
\begin{align*}
\sum_p \left[\frac{|f(p)|^4}{p^{2}} +\sum_{k \ge 2} \frac{|f(p^k)|^2}{p^{k}}\right]
\ll \sum_p \left[\frac{|f(p)|^3}{p^{3/2}}  +\sum_{k \ge 2} \frac{|f(p^k)|^2}{p^{k}}\right] \ll 1
\end{align*}

\noindent by the assumption \eqref{eq:f-summability}. In particular,
\begin{align*}
\EE\left[ m_{y, \infty}(h)^2\right]
& = \iint_{I \times I} h(s_1) h(s_2)  \frac{\EE\left[A_y(\frac{1}{2} + is_1)|^2A_y(\frac{1}{2} + is_2)|^2\right]}{\EE\left[A_y(\frac{1}{2} + is_1)|^2\right]\EE\left[A_y(\frac{1}{2} + is_2)|^2\right]} ds_1 ds_2\\
& \ll \|h\|_\infty^2 \iint_{I \times I}  \exp \left( \sum_{p \le y} 2 \frac{|f(p)|^2}{p} \cos(|s_1 - s_2| \log p)\right)  ds_1 ds_2\\
& \ll \iint_{I \times I} \frac{ds_1 ds_2}{|s_1 - s_2|^{2\theta}}
\end{align*}

\noindent by \Cref{lem:martingaleL2est} (with $g = |f|^2$ and $c = 0$), and this integral is bounded uniformly in $y$ for any fixed $\theta \in (0, \frac{1}{2})$. Therefore, we can invoke the $L^2$ martingale convergence theorem (see e.g., \cite[Corollary 9.23]{Kal2021}), which says that there exists some nontrivial random variable $m_{\infty}(h)$ such that
\begin{align}\label{eq:martingale-convergence}
m_{y, \infty}(h) \to m_{\infty}(h) \qquad \text{as} \quad  y \to \infty
\end{align}

\noindent in the sense of almost sure convergence and convergence in $L^2$.

\begin{rem}\label{rem:missing-abstract}
The fact that there exists some some Radon measure $m_\infty(ds)$ on $\RR$ such that the limits in \eqref{eq:martingale-convergence} can be realised as integrals against $m_\infty(ds)$ (which is part of the statement of \Cref{thm:mc-convergence}) does not follow from martingale convergence theorem but abstract theory of probability/functional analysis. Since this is not needed for the rest of our proof, we refer the interested readers to \Cref{sec:convergence_measure} for a brief discussion of convergence of random measures and the references therein.
\end{rem}
\paragraph{Step 2: truncation.}
Our next goal is to show that
\begin{lem} \label{lem:mcL2truncate}
Under the same setting as \Cref{thm:mc-convergence}, we have
\begin{align}
\lim_{y \to \infty} \EE\left[ \left| m_{\infty, y}(h) - m_{y,y}(h)\right|^2 \right] = 0.
\end{align}
\end{lem}

\begin{proof}
Let $y_0$ be the constant in \Cref{lem:truncate} and $y \ge y_0$, and write $z_j = \sigma_y + is_j$ for $j = 1, 2$. Then
\begin{align*}
 \EE\left[ \left| m_{\infty, y}(h) - m_{y,y}(h)\right|^2 \right] 
& = \iint_{I \times I} h(s_1)h(s_2) ds_1 ds_2  \frac{\EE\left[A_y(\sigma_y + is_1)|^2A_y(\sigma_y+ is_2)|^2\right]}{\EE\left[A_y(\sigma_y + is_1)|^2\right]\EE\left[A_y(\sigma_y + is_2)|^2\right]} \left(K_{y}(s_1, s_2) - 1\right)
\end{align*}

\noindent where
\begin{align*}
K_y(s_1, s_2)
& := \prod_{p > y} \frac{\EE\left[ \left|1 + \sum_{k \ge 1}\frac{\alpha(p)^k f(p^k)}{p^{kz_1}}\right|^2 \left|1 + \sum_{k \ge 1} \frac{\alpha(p)^k f(p^k)}{p^{kz_2}}\right|^2\right]}{\EE\left[ \left|1 + \sum_{k \ge 1}\frac{\alpha(p)^k f(p^k)}{p^{kz_1}}\right|^2\right]\EE\left[ \left|1 + \sum_{k \ge 1}\frac{\alpha(p)^k f(p^k)}{p^{kz_2}}\right|^2\right]}\\
& \ll \prod_{p > y} \frac{\EE\left[ \left|1 + \sum_{k = 1}^2\frac{\alpha(p)^k f(p^k)}{p^{kz_1}}\right|^2 \left|1 + \sum_{k =1}^2 \frac{\alpha(p)^k f(p^k)}{p^{kz_2}}\right|^2\right]}{\EE\left[ \left|1 + \sum_{k \ge 1}\frac{\alpha(p)^k f(p^k)}{p^{kz_1}}\right|^2\right]\EE\left[ \left|1 + \sum_{k \ge 1}\frac{\alpha(p)^k f(p^k)}{p^{kz_2}}\right|^2\right]}\\
& \ll\exp \left (2\sum_{p > y} \frac{|f(p)|^2}{p^{1 + \frac{1}{\log y}}} \cos(|s_1 - s_2| \log p)\right) 
\end{align*}

\noindent by \Cref{lem:truncate} and \Cref{lem:cross-moments}. Applying \Cref{lem:martingaleL2est-2} with $g = |f|^2$ and $s = |s_1 - s_2|$, we see that $|K_y(s_1, s_2) - 1|$ is uniformly bounded in $y$ and $s_1, s_2 \in I$, and $\lim_{y \to \infty} |K_y(s_1, s_2) - 1| = 0$ for any fixed $s_1 \ne s_2$. Meanwhile,
\begin{align*}
& \frac{\EE\left[A_y(\sigma_y + is_1)|^2A_y(\sigma_y+ is_2)|^2\right]}{\EE\left[A_y(\sigma_y + is_1)|^2\right]\EE\left[A_y(\sigma_y + is_2)|^2\right]} 
= \prod_{p \le y} \frac{\EE\left[\left|1 + \sum_{k \ge 1} \frac{\alpha(p)^k f(p^k)}{p^{kz_1}}\right|^2\left|1 + \sum_{k \ge 1} \frac{\alpha(p)^k f(p^k)}{p^{kz_2}}\right|^2\right]}{\EE\left[\left|1 + \sum_{k \ge 1} \frac{\alpha(p)^k f(p^k)}{p^{kz_1}}\right|^2\right]\EE\left[\left|1 + \sum_{k \ge 1} \frac{\alpha(p)^k f(p^k)}{p^{kz_2}}\right|^2\right]}\\
& \qquad \ll \prod_{y_0 \le p \le y}  \frac{\EE\left[\left|1 + \sum_{k =1}^2 \frac{\alpha(p)^k f(p^k)}{p^{kz_1}}\right|^2\left|1 + \sum_{k = 1}^2 \frac{\alpha(p)^k f(p^k)}{p^{kz_2}}\right|^2\right]}{\EE\left[\left|1 + \sum_{k \ge 1} \frac{\alpha(p)^k f(p^k)}{p^{kz_1}}\right|^2\right]\EE\left[\left|1 + \sum_{k \ge 1} \frac{\alpha(p)^k f(p^k)}{p^{kz_2}}\right|^2\right]}\\
& \qquad \ll  \exp \left( \sum_{p \le y} 2 \frac{|f(p)|^2}{p^{1+\frac{1}{\log y}}} \cos(|s_1 - s_2| \log p)\right) 
\end{align*}

\noindent by \Cref{lem:truncate} and \Cref{lem:cross-moments}. The above bound is $\ll |s_1 - s_2|^{-2\theta}$ uniformly in $y$ and $s_1, s_2 \in I$ by \Cref{lem:martingaleL2est}, and is integrable in particular. Therefore, we conclude by dominated convergence that
\begin{align*}
&\lim_{y \to \infty} \EE\left[ \left| m_{\infty, y}(h) - m_{y,y}(h)\right|^2 \right] \\
& = \iint_{I \times I} \lim_{y \to \infty} h(s_1)h(s_2) ds_1 ds_2  \frac{\EE\left[A_y(\sigma_y + is_1)|^2A_y(\sigma_y+ is_2)|^2\right]}{\EE\left[A_y(\sigma_y + is_1)|^2\right]\EE\left[A_y(\sigma_y + is_2)|^2\right]} \left(K_{y}(s_1, s_2) - 1\right) = 0.
\end{align*}
\end{proof}

\paragraph{Step 3: comparing two truncated Euler products.}
We now show that the new approximating sequence $(m_{y, y}(h))_y$ has the limit as our martingale construction.
\begin{lem}\label{lem:mcL2compare}
Under the same setting as \Cref{thm:mc-convergence}, we have
\begin{align}\label{eq:final-2ndmom}
\lim_{y \to \infty} \EE\left[ \left| m_{y, \infty}(h) - m_{y,y}(h)\right|^2 \right] = 0.
\end{align}
\end{lem}

\begin{proof}
Expanding the square, we have
\begin{align}\label{eq:2nd-moment-expand}
\EE\left[ \left| m_{y, \infty}(h) - m_{y,y}(h)\right|^2 \right] 
= \EE\left[ \left| m_{y, \infty}(h)\right|^2 \right] 
-2 \EE\left[  m_{y, \infty}(h) m_{y,y}(h) \right] 
+ \EE\left[ \left|m_{y,y}(h)\right|^2 \right]
\end{align}

\noindent where (for $t_1, t_2 \in \{y, \infty\}$)
\begin{align}\label{eq:truncated-DCT}
\EE\left[ m_{y, t_1}(h) m_{y, t_2}(h) \right]
= \iint_{I \times I} h(s_1) h(s_2) \frac{\EE\left[A_y(\sigma_{t_1} + is_1)|^2A_y(\sigma_{t_2}+ is_2)|^2\right]}{\EE\left[A_y(\sigma_{t_1} + is_1)|^2\right]\EE\left[A_y(\sigma_{t_2} + is_2)|^2\right]}  ds_1 ds_2
\end{align}

\noindent and the ratio of expectations inside the integrand is
\begin{align*}
\ll \|h\|_\infty^2  \exp \left( \sum_{p \le y} 2 \frac{|f(p)|^2}{p^{1+\frac{1}{2}[\log^{-1} t_1 + \log^{-1} t_2]}} \cos(|s_1 - s_2| \log p)\right) 
\ll |s_1 - s_2|^{-2\theta}
\end{align*} 

\noindent uniformly in $s_1, s_2 \in I$ by \Cref{lem:martingaleL2est}, and it has a pointwise limit
\begin{align*}
\lim_{y \to \infty} \frac{\EE\left[A_y(\sigma_{t_1} + is_1)|^2A_y(\sigma_{t_2}+ is_2)|^2\right]}{\EE\left[A_y(\sigma_{t_1} + is_1)|^2\right]\EE\left[A_y(\sigma_{t_2} + is_2)|^2\right]} 
= \prod_p \frac{\EE\left[ \left|1 + \sum_{k \ge 1} \frac{\alpha(p)^k f(p^k)}{p^{k(1/2+is_1)}}\right|^2 \left|1 + \sum_{k \ge 1} \frac{\alpha(p)^k f(p^k)}{p^{k(1/2+is_2)}}\right|^2\right]}{\EE\left[ \left|1 + \sum_{k \ge 1} \frac{\alpha(p)^k f(p^k)}{p^{k(1/2 + is_1)}}\right|^2 \right]\EE\left[ \left|1 + \sum_{k \ge 1} \frac{\alpha(p)^k f(p^k)}{p^{k(1/2+is_2)}}\right|^2 \right]}
\end{align*}

\noindent for any $s_1 \ne s_2$. Thus, all the expectations appearing on the right-hand side of \eqref{eq:2nd-moment-expand} converge to the same limit
\begin{align*}
\iint_{I \times I} h(s_1) h(s_2) ds_1 ds_2  \prod_p \frac{\EE\left[ \left|1 + \sum_{k \ge 1} \frac{\alpha(p)^k f(p^k)}{p^{k(1/2+is_1)}}\right|^2 \left|1 + \sum_{k \ge 1} \frac{\alpha(p)^k f(p^k)}{p^{k(1/2+is_2)}}\right|^2\right]}{\EE\left[ \left|1 + \sum_{k \ge 1} \frac{\alpha(p)^k f(p^k)}{p^{k(1/2 + is_1)}}\right|^2 \right]\EE\left[ \left|1 + \sum_{k \ge 1} \frac{\alpha(p)^k f(p^k)}{p^{k(1/2+is_2)}}\right|^2 \right]}
\end{align*}

\noindent as $y \to \infty$ by dominated convergence, and we conclude that \eqref{eq:final-2ndmom} holds.
\end{proof}

\paragraph{Step 4: full support.}
\begin{proof}[Proof of \Cref{thm:mc-convergence}]
Combining the martingale convergence in Step 1, as well as \Cref{lem:mcL2truncate} and \Cref{lem:mcL2compare} from Step 2 and 3 respectively, we have proved the $L^2$ convergence \eqref{eq:thm:mc-convergence}. We just need to verify the final claim that $m_\infty(ds)$ is supported on $\RR$ almost surely, and it is sufficient to check that $\PP(m_\infty(I) = 0) = 0$ for any compact interval $I$ with positive Lebesgue measure $|I| > 0$.

Note that for any fixed $p_0 > 0$, the sequence of random variables
\begin{align*}
\int_I m_{(p_0, y], \infty}(ds)
:= \int_I \prod_{p_0 < p \le y} \frac{ \left|1 + \sum_{k \ge 1} \frac{\alpha(p)^k f(p^k)}{p^{k(1/2+is)}}\right|^2}{\EE\left[\left|1 + \sum_{k \ge 1} \frac{\alpha(p)^k f(p^k)}{p^{k(1/2+is)}}\right|^2\right]} ds
\end{align*}

\noindent is again an $L^2$-bounded martingale with respect to the filtration $(\Fa_y)_y$, and thus converges to some almost sure limit $m_{(p_0, \infty]}(I)$ as $y \to \infty$, and this limit satisfies the trivial comparison inequality
\begin{align*}
C_{\min}(I) m_{(p_0, \infty]}(I) 
\le m_\infty(I) 
\le C_{\max}(I) m_{(p_0, \infty]}(I) 
\end{align*}

\noindent almost surely with
\begin{align*}
C_{\min}(I) := \left\{ \min_{s \in I}  \prod_{p \le p_0} \frac{ \left|1 + \sum_{k \ge 1} \frac{\alpha(p)^k f(p^k)}{p^{k(1/2+is)}}\right|^2}{\EE\left[\left|1 + \sum_{k \ge 1} \frac{\alpha(p)^k f(p^k)}{p^{k(1/2+is)}}\right|^2\right]} \right\} , \quad
C_{\max}(I) := \left\{ \max_{s \in I}  \prod_{p \le p_0} \frac{ \left|1 + \sum_{k \ge 1} \frac{\alpha(p)^k f(p^k) }{p^{k(1/2+is)}}\right|^2}{\EE\left[\left|1 + \sum_{k \ge 1} \frac{\alpha(p)^k f(p^k) }{p^{k(1/2+is)}}\right|^2\right]} \right\} 
\end{align*}

\noindent which are some random but finite values (the minimum and maximum are attained because we are dealing with finite product over $p \le p_0$, and any infinite sums that we encounter are absolutely convergent). In other words, $m_\infty(I) = 0$ if and only if $m_{(p_0, \infty]}(I) = 0$.

By construction, $m_{(p_0, \infty]}(I)$ only depends on $\alpha(p)$ for $p > p_0$, i.e., it is measurable with respect to the $\sigma$-algebra $\Ga_{p_0} := \sigma\left(\alpha(p), p > p_0\right)$. In particular, we have
\begin{align*}
\{m_\infty(I) = 0\} = \{m_{(p_0, \infty]}(I)\} \in \Ga_{p_0}.
\end{align*}

\noindent Since this is true for arbitrary values of $p_0$, we have $\{m_\infty(I) = 0\}  \in \bigcap_{p_0}\Ga_{p_0}$. Given that $(\alpha(p))_{p}$ are i.i.d. random variables, the tail $\sigma$-algebra $\bigcap_{p_0}\Ga_{p_0}$ is trivial by Kolmogorov $0-1$ law  \cite[Theorem~4.13]{Kal2021}. In particular, we have
\begin{align}\label{eq:01-conclusion}
\PP(m_\infty(I) = 0) \in \{0, 1\}.
\end{align}

\noindent On the other hand, recall from \eqref{eq:martingale-convergence} that the convergence of $m_{y, \infty}(I)$ to $m_\infty(I)$ holds in $L^2$. This implies $\EE[m_\infty(I)] = |I| > 0$, and in particular $\PP(m_\infty(I) > 0) > 0$. Combining this with \eqref{eq:01-conclusion} we deduce that $\PP(m_\infty(I) = 0)  = 1$, which concludes our proof.
\end{proof}

\begin{rem}\label{rem:support-ppt}
Let us discuss some further support properties of the measure $m_\infty(ds).$

By Fatou's lemma (or uniform integrability by our $L^2$ construction), we have $\EE\left[m_{\infty}(\{s_0\})\right] = 0 $ for any fixed $s_0 \in \RR$. Thus $\PP(m_\infty(\{s_0\}) = 0) = 1$ and this can be extended to a fixed subset of $\RR$ with at most countably many points.

It is not difficult to refine our analysis and show that $m_\infty(ds)$ does not contain any Dirac mass with probability $1$. By stationarity, let us prove that this is the case when we restrict ourselves to the interval $[0,1]$. Since $\EE[m_\infty([0,1])] = 1$, we know $\PP\left(m_\infty([0,1]) < \infty\right) = 1$. Moreover, if we denote by $\Da$ the collection of dyadic points on $[0,1]$, then we have $\PP\left(m_\infty(\Da) = 0\right) = 1$ by our previous observation.  However, on the event $\{m_\infty([0,1]) < \infty \} \cap \{m_\infty(\Da) = 0\}$, we have
\begin{align*}
    m_\infty([0,1])^2
    \ge \sum_{j = 0}^{2^n - 1} m_{\infty}([j2^{-n}, (j+1)2^{-n}])^2
    \ge \left(\sup_{s_0 \in [0,1]} m_\infty(\{s_0\})\right)  m_\infty([0,1])
\end{align*}

\noindent for any $n \in \NN$, and
\begin{align*}
\EE\left[ \sum_{j = 0}^{2^n - 1} m_{\infty}([j2^{-n}, (j+1)2^{-n}])^2\right]
\ll \sum_{j=0}^{2^n - 1} \iint_{[j2^{-n}, (j+1)2^{-n}]^2} \frac{ds_1 ds_2}{|s_1 - s_2|^{2\theta}}
\ll 2^{-n(1-2\theta)} \xrightarrow[n \to \infty]{\theta \in (0, \frac{1}{2})} 0.
\end{align*}

\noindent In particular, $\EE\left[ \left(\sup_{s_0 \in [0,1]} m_\infty(\{s_0\})\right)  m_\infty([0,1])\right] = 0$. By the full support property of $m_\infty$, we have $\PP(m_\infty([0,1]) > 0) = 1$ and conclude that $\PP\left(\sup_{s_0 \in [0,1]} m_\infty(\{s_0\}) > 0 \right) = 0$.
\end{rem}

\subsection{Proof of \texorpdfstring{\Cref{lem:limit-cond-var}}{Lemma \ref{lem:limit-cond-var}}} \label{sec:var-Tauberian}
To deduce the convergence of $U_x$, we just need one final probabilistic ingredient, namely a generalisation of the Hardy--Littlewood--Karamata Tauberian theorem.
\begin{thm}[{\cite[Theorem A.2]{BW2023}}]
\label{theo:tauberian}
Let $\nu(d \cdot)$ be a nonnegative random measure on $\RR_{\ge 0}$, $\nu(t) := \int_0^t \nu(ds)$, and suppose the Laplace transform \begin{align*}
    \hat{\nu}(\lambda) := \int_0^\infty e^{-\lambda s} \nu(ds)
\end{align*}

\noindent exists almost surely for any $\lambda > 0$. If
\begin{itemize}\setlength\itemsep{0em}
\item $\rho \in [0, \infty)$ is fixed;
\item $L\colon (0, \infty) \to (0, \infty)$ is a deterministic slowly varying function at $0$; and 
\item $C_\nu$ is some nonnegative (finite) random variable,
\end{itemize}
then we have
\begin{align}\label{eq:tauberian}
    \frac{\lambda^\rho}{L(\lambda^{-1})} \hat{\nu}(\lambda) \xrightarrow[\lambda \to 0^+]{p} C_\nu
    \qquad \Rightarrow \qquad 
    \frac{t^{-\rho}}{L(t)} \nu(t) \xrightarrow[t \to \infty]{p} \frac{C_\nu}{\Gamma(1+\rho)}.
\end{align}
\end{thm}

\begin{proof}[Proof of \Cref{lem:limit-cond-var}]
We first show that
\begin{align}\label{eq:con-var1}
    \int_\RR \frac{|A_\infty(\sigma_x+it)|^2}{\EE\left[|A_\infty(\sigma_x+it)|^2\right]} \frac{dt}{\left|\sigma_x + it\right|^2} 
    = \int_\RR \frac{m_{\infty, x}(dt)}{\left|\sigma_x + it\right|^2} 
    \xrightarrow[x \to \infty]{L^1} \int_\RR \frac{m_\infty(dt)}{|1/2 + it|^2}.
\end{align}

\noindent To see this, consider the upper bound
\begin{align*}
\EE\left|\int_\RR \frac{m_{\infty, x}(dt)}{\left|\sigma_x + it\right|^2}  -  \int_\RR \frac{m_\infty(dt)}{|1/2 + it|^2}\right|
& \le \sum_{n \in \ZZ} 
\EE\left|\int_{n}^{n+1} \frac{m_{\infty, x}(dt)}{\left|1/2 + it\right|^2}  -  \int_{n}^{n+1} \frac{m_\infty(dt)}{|1/2 + it|^2}\right|\\
& \quad + \sum_{n \in \ZZ} 
\EE\left|\int_{n}^{n+1} \left[\frac{1}{|\sigma_x + it|^{2}} - \frac{1}{|1/2 + it|^{2}}\right] m_{\infty, x}(dt)\right|
\end{align*}

\noindent where the summand on the right-hand side is $\ll (1+n^2)^{-1}$ uniformly in $x$ (hence summable in particular), and by dominated convergence we just have to verify that
\begin{align*}
    (a)& \quad \EE\left|\int_{n}^{n+1} \frac{m_{\infty, x}(dt)}{\left|1/2 + it\right|^2}  -  \int_{n}^{n+1} \frac{m_\infty(dt)}{|1/2 + it|^2}\right| \to 0\\
    \qquad \text{and} \qquad
    (b)& \quad \EE\left|\int_{n}^{n+1} \left[\frac{1}{|\sigma_x + it|^{2}} - \frac{1}{|1/2 + it|^{2}}\right] m_{\infty, x}(dt)\right| \to 0
\end{align*}

\noindent for each $n \in \ZZ$. But (a) immediately follows from the $L^2$-convergence in \Cref{thm:mc-convergence} (by taking $h(t) := |1/2 + it|^{-2}$ on $I = [n, n+1]$), whereas (b) is a simple consequence of the bound
\begin{align*}
    \EE\left|\int_{n}^{n+1} \left[\frac{1}{|\sigma_x + it|^{2}} - \frac{1}{|1/2 + it|^{2}}\right] m_{\infty, x}(dt)\right|
    \le \frac{1}{\log x}\EE\left[\int_{n}^{n+1} \frac{m_{\infty, x}(dt)}{|1/2 + it|^{4}} \right] \le \frac{16}{\log x} \to 0.
\end{align*}

Next, we recall that $\EE[|A_\infty(\sigma_x + it)|^2]$ is independent of $t \in \RR$, see \eqref{eq:full-expansion} from the proof of \Cref{lem:cross-moments}. This combined with \eqref{eq:log-sum-twist} in \Cref{lem:martingaleL2est} implies that there exists some constant $C \in (0, \infty)$ such that $\EE[|A_\infty(\sigma_x + it)|^2] \sim C \log^{\theta} x$ as $x \to \infty$. In other words, we have (by Plancherel's identity)
\begin{align*}
    \frac{2\pi}{C \log^\theta x} \int_1^\infty \frac{|s_t|^2}{t^{1+ \frac{1}{\log x}}} dt 
    = \frac{1}{C \log^\theta x} \int_{\RR} \frac{|A_\infty(\sigma_x+it)|^2}{\left|\sigma_x + it\right|^2} dt
    \xrightarrow[x \to \infty]{L^1} \int_\RR \frac{m_\infty(dt)}{|1/2 + it|^2}.
\end{align*}

\noindent We now apply \Cref{theo:tauberian}: taking $\lambda = 1/\log x$ and $\rho = \theta$, we see that
\begin{align*}
\frac{2\pi}{C \log^\theta x} \int_1^\infty \frac{|s_t|^2}{t^{1+ \frac{1}{\log x}}} dt 
=  \frac{2\pi}{C \log^\theta x} \int_0^\infty e^{-t / \log x} |s_{e^{t}}|^2dt 
\xrightarrow[x \to \infty]{L^1} \int_\RR \frac{m_\infty(dt)}{|1/2 + it|^2}
\end{align*}

\noindent (and the convergence also holds in probability in particular) implies that
\begin{align*}
\frac{2\pi}{C \log^\theta x} \int_1^x \frac{|s_t|^2}{t} dt 
=  \frac{2\pi}{C \log^\theta x} \int_0^{\log x} |s_{e^{t}}|^2 dt 
\xrightarrow[x \to \infty]{p} \frac{1}{\Gamma(1+\theta)}\int_\RR \frac{m_\infty(dt)}{|1/2 + it|^2},
\end{align*}

\noindent and the last convergence also holds in $L^1$ since
\begin{align}\label{eq:U_x-ui}
\frac{2\pi }{C \log^\theta x} \int_1^x \frac{|s_t|^2}{t} dt  \le 
\frac{2\pi e}{C \log^\theta x} \int_1^\infty \frac{|s_t|^2}{t^{1+ \frac{1}{\log x}}} dt 
\end{align}

\noindent and the right-hand side is uniformly integrable (by its $L^1$-convergence as $x \to \infty$). In particular,
\begin{align*}
\lim_{x \to \infty}\frac{2\pi}{C \log^\theta x} \EE\left[\int_1^x \frac{|s_t|^2}{t} dt\right]
= \frac{1}{\Gamma(1+\theta)} \EE\left[\int_\RR \frac{m_\infty(dt)}{|1/2 + it|^2} \right]
= \frac{1}{\Gamma(1+\theta)}\int_{\RR} \frac{dt}{\frac{1}{4} + t^2} = \frac{2\pi}{\Gamma(1+\theta)}.
\end{align*}

\noindent Comparing this with the normalisation condition $\EE[U_x] = 1$, we conclude that \eqref{eq:Uxlim} holds.
\end{proof}

\subsection{Uniform integrability of \texorpdfstring{$|S_x|^{2q}$}{Sx2q}} \label{subsec:mom-convergence}
\begin{proof}[Proof of \Cref{cor:mom-convergence}]
Given $q \ge 0$, if we want to show $\lim_{x \to \infty}\EE[|S_x|^{2q}] = \EE [|\sqrt{V_{\infty}} G|^{2q}]$ using \Cref{thm:main} it suffices to prove that $|S_x|^{2q}$ is uniformly integrable, see \cite[Lemma 5.11]{Kal2021}. This means proving 
\[     \sup_{x \ge 3} \EE\left[|S_x|^{2q} 1_{\{|S_x|^{2q} > K\}}\right]\xrightarrow[K \to \infty]{} 0. \]
If $q \in [0,1)$ then this is straightforward because
\begin{align*}
    \sup_{x \ge 3} \EE\left[|S_x|^{2q} 1_{\{|S_x|^{2q} > K\}}\right]
    & \le \sup_{x \ge 3} \EE\left[|S_x|^{2q} \left(|S_x| / K^{1/2q}\right)^{2 - 2q}\right]
    = K^{1 - \frac{1}{q}}
\end{align*}
\noindent using the fact that $\EE[|S_x|^2] \equiv 1$. For $q=1$ we argue directly: since $\EE [|S_x|^2]\equiv 1$ and $\EE [|G|^2]\equiv 1$, the claim $\lim_{x \to \infty}\EE[|S_x|^{2q}] = \EE [|\sqrt{V_{\infty}} G|^{2q}]$ for $q=1$ is equivalent to saying $\EE V_{\infty} \equiv 1$. From \Cref{thm:mc-convergence} and the identity $\int_{\RR} ds/|1/2+is|^2 = 2\pi$, $\EE V_{\infty} \equiv 1$ holds and we are done.
\end{proof}
To extend the proof of \Cref{cor:mom-convergence} to $q \in (1, 2)$ when $\theta \in (0, \tfrac{1}{2})$, it suffices to verify $\sup_{x \ge 3}\EE[|S_{x}|^4]< \infty$. This is plausible because we expect the size of $\EE[|S_x|^4]$ to be comparable to that of $\EE[|U_x|^2]$, which we know is uniformly bounded due to the fact that $\lim_{x \to \infty} \EE[|U_x|^2] = \Lsum$ from \Cref{lem:conv}. Indeed it may even be possible to check that $\lim_{x \to \infty} \EE[|S_{x}|^4] = 2\Lsum = \Gamma(3) \EE[V_\infty^2]$ via direct computation, by adapting the analysis in \Cref{subsec:l2computation}. More generally, one should be able to compute $\lim_{x \to \infty} \EE[|S_{x}|^{2k}]$ for any positive integer $k<1/\theta$.

The convergence $\EE[|S_x|^{2q}] \to \Gamma(1+q) \EE[V_\infty^q]$ is expected to be true for all $q \in [0, 1/\theta)$ (and should continue to hold for $\theta \in [\tfrac{1}{2}, 1)$ if one could verify distributional convergence in this range). This threshold should not be surprising as it corresponds to the integrability of $V_\infty$.\footnote{We do not establish this fact in the current paper, but for the special case where $f(n) = \theta^{\Omega(n)/2}$ this is essentially verified in \cite[Theoerem 1.8]{sw}} It is reasonable to believe that in the aforementioned range for $(q, \theta)$, we have
\begin{align*}
    \EE[|S_x|^{2q}] \ll 1 + \EE\left[\left(\int_\RR \frac{m_{y, \infty}(ds)}{|\tfrac{1}{2} + is|^2}\right)^q\right]
\end{align*} 

\noindent where $m_{y, \infty}(ds)$ is defined in \eqref{eq:mc-measure} and $y = y(x)$ is some suitable function that goes to infinity as $x \to \infty$. With techniques of Gaussian approximation, one could then translate moment bounds for Gaussian multiplicative chaos to the uniform integrability of $|S_x|^{2q}$ for all fixed $q \in [1, 1/\theta).$ We refer the interested readers to \cite{ShortNote} for details of the approximation procedure and ingredients from the theory of Gaussian multiplicative chaos, and \cite{Har2020} for an alternative and more self-contained analysis.

\section{The Rademacher case}\label{sec:rad}
\subsection{Main adaptations}
We explain how to prove \Cref{thm:mainRad} by modifying the proof of \Cref{thm:main}. The proof relies on the (real-valued) martingale central limit theorem. The following is a special case of \cite[Theorem 3.2 and Corollary 3.1]{HH1980}.
\begin{thm}\label{thm:mCLT}
For each $n$, let $(M_{n,j})_{j\le k_n}$ be a real-valued, mean-zero and square integrable martingale with respect to the filtration $(\Fa_{j})_j$, and $\Delta_{n, j} := M_{n, j} - M_{n, j-1}$ be the corresponding martingale differences. Suppose the following conditions are satisfied.
\begin{itemize}
    \item[(a)] The conditional covariance converges, i.e.,
    \begin{align}
        V_{n} := \sum_{j =1}^{k_n} \EE[|\Delta_{n, j}|^2 | \Fa_{j-1}] & \xrightarrow[n \to \infty]{p} V_\infty.
    \end{align}

    \item[(b)] The conditional Lindeberg condition holds: for any $\delta > 0$,
    \begin{align*}
        \sum_{j = 1}^{k_n} \EE[|\Delta_{n, j}|^2 1_{\{|\Delta_{n, j}| > \delta\}} | \Fa_{j-1}] \xrightarrow[n \to \infty]{p} 0.
    \end{align*}
\end{itemize}

\noindent Then
\begin{align*}
    M_{n, k_n} \xrightarrow[n \to \infty]{d} \sqrt{V_\infty} \ G
\end{align*}

\noindent where $G \sim \Na(0, 1)$ is independent of $V_\infty$, and the convergence in law is also stable.
\end{thm}

Let
\begin{align}
\Fa_p^\RadSym := \sigma(\Rad(q), q \le p),\qquad \Fa_{p^-}^\RadSym := \sigma(\Rad(q), q <p).
\end{align}
We use \Cref{lem:mCLT} and follow the 4 steps described in \Cref{sec:mainproof}. Each step will possibly require (in addition to  $f \in \Ma^{\Sf}$) different restrictions on $f$. The conditions in \Cref{thm:mainRad} will be easily seen to imply all these restrictions. 
Given $f \in \Ma^{\Sf}$ and $\OurEpsilon>0$, we define 
\[ S^{\RadSym}_{x,\OurEpsilon} :=\frac{1}{\sqrt{\sum_{n \le x} f(n)^2}}\sum_{\substack{n\le x\\ P(n) \ge x^{\OurEpsilon}}}\Rad(n) f(n)\]
and write
\begin{align}\label{eq:eps-truncated-SRad}
S^{\RadSym}_{x,\OurEpsilon} = \sum_{x^{\OurEpsilon}\le p\le x} Z^{\RadSym}_{x,p}, \qquad  Z^{\RadSym}_{x,p} := \frac{1}{\sqrt{\sum_{n \le x} |f(n)|^2}}\sum_{\substack{n\le x\\ P(n)=p}} \Rad(n) f(n).
\end{align}
\begin{lem}\label{lem:negRad}
		Let $\theta>0$. Suppose $f\in \Ma^{\Sf}$ satisfies  $f^2 \in \Ma_{\theta}$. Then $\limsup_{\OurEpsilon\to 0^+} \limsup_{x \to \infty}\EE[(S^{\RadSym}_x - S^{\RadSym}_{x,\OurEpsilon})^2]=0$.
\end{lem}
The proof of \Cref{lem:negRad} is identical to that of \Cref{lem:neg}, once we change $S_x$ to $S^{\RadSym}_x$ and $S^{\RadSym}_{x,\OurEpsilon}$ to $S_{x,\OurEpsilon}$. 
%(in fact, it can be simplified: the expression $A_{2,\OurEpsilon}$ in \eqref{eq:A12} is identically $0$ since $f$ is supported on squarefrees).
\begin{lem}\label{lem:lindRad}
	Let $\theta >0$. Suppose $f\in \Ma^{\Sf}$ satisfies $f^2 \in \Ma_{\theta}$ and $\sum_{p \le x} f(p)^4 \ll x^2(\log x)^{-4\theta-4}$. Then $\limsup_{x \to \infty}\sum_{x^{\OurEpsilon}\le p\le x} \EE [(Z^{\RadSym}_{x,p})^4] =0$ for every $\OurEpsilon>0$.
\end{lem}
The proof of \Cref{lem:lindRad} requires the following generalisation of \eqref{eq:orthRad}:
\begin{equation}\label{eq:genorth} \EE \left[ \prod_{i=1}^{k}\Rad(n_i)\right] = \mathbf{1}_{\prod_{i=1}^{k}n_i=\Box} \prod_{i=1}^{k} \mu^2(n_i),
\end{equation}
where $\Box$ stands for a perfect square. We also need the following lemma.
\begin{lem}\label{lem:param}
Let $m \in \NN$. Given squarefree $a,b,c,d \in \NN$ with $abcd=m^2$, there exist unique squarefree positive integers $(g_i)_{i=1}^{6}$ which are pairwise coprime except possibly for the pair $(g_1,g_2)$ and such that 
\begin{equation}\label{eq:param}
a = g_1 g_3 g_4, \, b = g_1 g_5 g_6,\, c = g_2 g_3 g_5,\, d = g_2 g_4 g_6
\end{equation}
and $abcd=(\prod_{i=1}^{6} g_i)^2$. Conversely, given squarefree positive integers $(g_i)_{i=1}^{6}$ which are pairwise coprime except possibly for $(g_1,g_2)$, and which satisfy $\prod_{i=1}^{6} g_i=m$, we can let $a$, $b$, $c$ and $d$ as in \eqref{eq:param} and obtain squarefree solutions to $abcd = m^2$.
\end{lem}
\begin{proof}[Proof of \Cref{lem:param}]
The converse is trivial because if we define $a,b,c,d$ as in \eqref{eq:param} then $abcd=(\prod_{i=1}^{6}g_i)^2=m^2$ and the assumptions on the $g_i$ imply $a$, $b$, $c$ and $d$ are squarefree. For the first part, the properties of $g_i$ and \eqref{eq:param} force us to define
\[g_1=(a,b),\, g_2=(c,d),\, g_3=\left( \frac{a}{g_1}, \frac{c}{g_2}\right), \, g_4=\left( \frac{a}{g_1}, \frac{d}{g_2}\right)\, g_5=\left( \frac{b}{g_1},\frac{c}{g_2}\right),\,g_6=\left( \frac{b}{g_1}, \frac{d}{g_2}\right).\]
The squarefreeness of $a,b,c,d$ implies  $(g_i,g_j)=1$ when $i \neq j$ except possibly when $\{i,j\}=\{1,2\}$. To verify \eqref{eq:param} for these $g_i$, let
$a'=a/g_1$, $b'=b/g_1$, $c'=c/g_2$ and $d'=d/g_2$. Then we have $(a',b')=(c',d')=1$ and 
\[ a'' b''c'' d''=\left(\frac{m}{g_1 g_2 g_3 g_4 g_5 g_6} \right)^2\]
with 
\[ a'' = \frac{a'}{(a',c')(a',d')},\, b''= \frac{b'}{(b',c')(b',d')}, \, c'' =\frac{c'}{(c',a')(c',b')},\, d''= \frac{d'}{(d',a')(d',b')}.\]
Since $a'',b'',c'',d''$ are squarefree and pairwise coprime, $a''b''c''d''$ being a perfect square implies $a''=b''=c''=d''=1$, which is equivalent to \eqref{eq:param}.
\end{proof}
\begin{proof}[Proof of \Cref{lem:lindRad}]
Let $f \in \Ma^{\Sf}$. By \eqref{eq:genorth},
\begin{equation}\label{eq:4thsumRad}
\sum_{x^{\OurEpsilon}\le p\le x} \EE\left[(Z^{\RadSym}_{x,p})^4\right]=(\sum_{n\le x}f(n)^2)^{-2}\sum_{x^{\OurEpsilon}\le p\le x}  \sum_{\substack{abcd=\Box\\ a,b,c,d \le x \\ P(a)=P(b)=P(c)=P(d)=p}} f(a)f(b)(c)f(d).
\end{equation}
The inner sum in the right-hand side of \eqref{eq:4thsumRad} can be written as 
\[\sum_{\substack{abcd=\Box\\ a,b,c,d \le x \\ P(a)=P(b)=P(c)=P(d)=p}} f(a)f(b)f(c)f(d)=\sum_{\substack{m \le x^2\\P(m)=p}}h^{\RadSym}(m), \quad h^{\RadSym}(m):=\sum_{\substack{abcd=m^2\\ P(a)=P(b)=P(c)=P(d)\\a,b,c,d\le x}} f(a)f(b)f(c)f(d).\]
We claim that we have the pointwise bound
\begin{equation}\label{eq:hpointwise} 0\le h^{\RadSym}(m) \le (f^2 * f^2 * f^2 *f^2 * f^2 * f^2)(m) \cdot \mathbf{1}_{P(m)^2 \mid\mid  m}.
\end{equation}
Indeed, if $f(a)f(b)f(c)f(d) \neq 0$ and $abcd=m^2$ then $a,b,c,d$ are squarefree, and, in the notation of \Cref{lem:param}, there are $(g_i)_{i=1}^{6}$ such that  $m=\prod_{i=1}^{6} g_i$ and $f(a)f(b)f(c)f(d) = \prod_{i=1}^{6} f(g_i)^2$. From \eqref{eq:hpointwise},
\begin{equation*} \sum_{\substack{abcd=\Box\\ a,b,c,d \le x \\ P(a)=P(b)=P(c)=P(d)=p\\ a,b,c,d \text{ squarefree}}} f(a)f(b)f(c)f(d)\le \sum_{\substack{m \le x^2\\P(m)=p\\ P(m)^2 \mid\mid  m}}\fname^{\RadSym}(m), \qquad \fname^{\RadSym} := f^2 * f^2 * f^2 *f^2 * f^2 * f^2.
\end{equation*}
The rest of the proof continues as in the proof of \Cref{lem:lind}, the only difference being that $\fname^{\RadSym}(p) = 6f(p)^2$ (as opposed to $\fname(p)=4|f(p)|^2$). 
\end{proof}

Let
\[ T^{\RadSym}_{x,\OurEpsilon} :=\sum_{x^{\OurEpsilon}\le p\le x} \EE\left[ (Z^{\RadSym}_{x,p})^2 \mid \Fa^{\RadSym}_{p^-}\right]\]
and
\begin{align}
	U^{\RadSym}_x := \bigg(\sum_{n\le x}f(n)^2\left(\frac{1}{n}-\frac{1}{x}\right)\bigg)^{-1}\int_{1}^{x} \frac{(s^{\RadSym}_t)^2}{t}dt, \qquad  s^{\RadSym}_t := \frac{1}{\sqrt{t}} \sum_{n\le t}  \Rad(n)f(n).
\end{align}
Recall $C_{\OurEpsilon}$ was defined in \eqref{eq:Ceps}.
\begin{proposition}\label{prop:l2Rad}
	Let $\theta \in (0,\frac{1}{2})$. Suppose $f\in \Ma^{\Sf}$ satisfies $f^2 \in \Ma_{\theta}$. Then $\lim_{x \to \infty}\EE[ (T^{\RadSym}_{x,\OurEpsilon}-C_{\OurEpsilon} U^{\RadSym}_x)^2]=0$.
\end{proposition}
Given $f \in \Ma^{\Sf}$ we define
\[ \Lsum^{\RadSym}:= \sum_{\substack{g_3,g_4,g_5,g_6\in \NN \\ \text{pairwise coprime}}} \frac{f^2(g_3) f^2(g_4)f^2(g_5)f^2(g_6)}{\max\{g_3 g_4, g_5 g_6\} \max\{g_3 g_5,g_4g_6\}}  \prod_{p\mid g_3 g_4g_5g_6}(1+f^2(p)/p)^{-2}.\]
\Cref{prop:l2Rad} will follow from
\begin{lem}\label{lem:convRad}
	Let $\theta \in (0,\frac{1}{2})$. Suppose $f \in \Ma^{\Sf}$ satisfies $f^2 \in \Ma_{\theta}$. Then $\Lsum^{\RadSym}$ converges and
	\begin{align}\label{eq:limits}
\lim_{x \to \infty}\EE[ (U^{\RadSym}_x)^2]= \Lsum^{\RadSym},\qquad \lim_{x \to \infty}\EE [T^{\RadSym}_{x,\OurEpsilon} U^{\RadSym}_x] = C_{\OurEpsilon}\Lsum^{\RadSym},\qquad \lim_{x \to \infty}
\EE [(T^{\RadSym}_{x,\OurEpsilon})^2] =C^2_{\OurEpsilon}\Lsum^{\RadSym}.
	\end{align}
\end{lem}
To see $\Lsum^{\RadSym}$ converges when $f \in \Ma^{\Sf}$ satisfies $f^2 \in \Ma_{\theta}$, recall that  $\sum_{n \le T} f^2(n) \ll T (\log T)^{\theta-1+o(1)}$ by \Cref{lem:upper}. Given $h_3,\ldots,h_6 \in \NN$, the contribution of $2^{h_i-1} \le g_i <2^{h_i}$ ($3 \le i \le 6$)  to $\Lsum^{\RadSym}$ can be seen to be, once we discard the coprimality conditions and use $\prod_{p\mid g_3 g_4g_5g_6}(1+f^2(p)/p)^{-2} \le 1$, at most
\[\ll 2^{h_3+h_4+h_5+h_6-\max\{h_3+h_4,h_5+h_6\}-\max\{h_3+h_5,h_4+h_6\}} (h_3 h_4 h_5 h_6)^{\theta-1+o(1)}\]
which converges when summing over $h_i\in \NN$. The proof of \eqref{eq:limits} is based on the following identities, which use the notation introduced in \eqref{eq:Fmx}.
\begin{lem}\label{lem:idenRad}
	Let $f \in \Ma^{\Sf}$. Denote $M_1 = \max\{g_3g_4,g_5g_6\}$, $M_2 = \max\{g_3 g_5,g_4 g_6\}$ and $G=\prod_{i=3}^{6}g_i$. We have
	\begin{align}\label{eq:UxrepRad}
		\EE [(U^{\RadSym}_x)^2] &= \sum_{\substack{g_3,\ldots,g_6 \in \NN \\ \text{pairwise coprime}}} \frac{f^2(g_3)f^2(g_4)f^2(g_5)f^2(g_6)}{M_1 M_2} S_1(x,G,M_1)S_1(x,G,M_2),\\
\notag		\EE [ T^{\RadSym}_{x,\OurEpsilon}U^{\RadSym}_x] &= \sum_{\substack{g_3,\ldots,g_6 \in \NN \\ \text{pairwise coprime}}} \frac{f^2(g_3)f^2(g_4)f^2(g_5)f^2(g_6)}{M_1 M_2} S_1(x,G,M_1) \\ 
\notag &\qquad  \sum_{\substack{g_2 \le x/M_2\\(g_2,G)=1}} f^2(g_2) M_2 \sum_{\max\{x^{\OurEpsilon},P(g_2G)+1\} \le p \le x/(g_2 M_2)} f^2(p)/\sum_{n \le x} f^2(n),\\
\notag		\EE [ (T^{\RadSym}_{x,\OurEpsilon})^2] &= \sum_{\substack{g_3,\ldots,g_6 \in \NN \\ \text{pairwise coprime}}} \frac{f^2(g_3)f^2(g_4)f^2(g_5)f^2(g_6)}{M_1 M_2}\\
\notag & \qquad \prod_{i=1}^{2} \sum_{\substack{g \le x/M_i\\ (g,G)=1}} f^2(g) \sum_{\max\{x^{\OurEpsilon},P(gG)+1\}\le p \le x/(gM_i)} f^2(p) M_i/(\sum_{n \le x} f^2(n)).
\end{align}
\end{lem}
We leave the deduction of \eqref{eq:limits} from \Cref{lem:idenRad} to the reader, as it is a matter of taking $x \to \infty$ and using dominated convergence as in the Steinhaus case. The derivation of \Cref{lem:idenRad} is similar to that of \Cref{lem:iden1}, \Cref{lem:TxUxS2} and \Cref{lem:txsquared}. We include the full details of \eqref{eq:UxrepRad}, the proofs of the other parts are omitted as they use the same ideas.
\begin{proof}[Proof of \eqref{eq:UxrepRad}]
	By definition,
	\[\EE [(U^{\RadSym}_x)^2 ]=F_1(x)^{-2} S^{\RadSym}, \qquad S^{\RadSym}:=\int_{1}^{x}\int_{1}^{x} \frac{\EE[(s^{\RadSym}_t)^2(s^{\RadSym}_r)^2]}{tr}dtdr.\]
	By \eqref{eq:genorth},
	\[\EE [(s^{\RadSym}_t)^2(s^{\RadSym}_r)^2 ]=\frac{1}{tr} \sum_{\substack{n_1,n_2 \le t \\ m_1, m_2 \le r\\ n_1 n_2 m_1 m_2=\Box }}f(n_1)f(n_2)f(m_1)f(m_2).\]
	It follows that
	\begin{align*}
		S^{\RadSym} = \sum_{\substack{n_1,n_2,m_1,m_2 \le x \\ n_1 n_2  m_1 m_2=\Box}} f(n_1)f(m_1)f(n_2)f(m_2) \left( \frac{1}{\max\{n_1,n_2\}}-\frac{1}{x}\right) \left( \frac{1}{\max\{m_1,m_2\}}-\frac{1}{x} \right).
	\end{align*}
Using \Cref{lem:param}, we may replace the $n_i$ and $m_i$ with squarefree $(g_i)_{i=1}^{6}$ such that
\[ n_1 = g_1 g_3 g_4,\, n_2 = g_1 g_5 g_6,\, m_1 = g_2g_3g_5,\, m_2 = g_2 g_4 g_6\]
and the $g_i$ are pairwise coprime except possibly for the pair $g_1,g_2$. Summing first over $g_3,\ldots,g_6$ and then over $g_1,g_2$, this allows us to express $S^{\RadSym}$ as
\[  \sum_{\substack{g_3,\ldots,g_6 \in \NN \\\ \text{pairwise coprime} }} \frac{f^2(g_3)f^2(g_4)f^2(g_5)f^2(g_6)}{\max\{g_3g_4,g_5g_6\}\max\{g_3g_5,g_4g_6\}} F_{g_3g_4g_5g_6}\bigg(\frac{x}{\max\{g_3g_4,g_5g_6\}}\bigg)F_{g_3g_4g_5g_6}\bigg(\frac{x}{\max\{g_3g_5,g_4g_6\}}\bigg).\]
To conclude we divide by $F_1(x)^2$.
\end{proof}

\subsection{Convergence of conditional variance in the Rademacher case}
Let us continue to write $\sigma_t = \frac{1}{2} \left(1 + \frac{1}{\log t}\right)$, but now define
\begin{align*}
    A_y^\RadSym(s) := \prod_{p \le y} \left(1 + \frac{\Rad(p)f(p)}{p^{ks}}\right)
    \qquad \text{and} \qquad 
    m_{y, t}^\RadSym(ds) := \frac{|A_y^\RadSym(\sigma_t + is)|^2}{\EE\left[|A_y^\RadSym(\sigma_t + is)|^2\right]}ds, \qquad s \in \RR.
\end{align*}

\noindent The following is the analogue of \Cref{thm:mc-convergence}.
\begin{thm} \label{thm:mc-convergence-rad}
Let $\theta \in (0, \frac{1}{2})$, and $f\in \Ma^{\Sf}$ be a function such that $f^2 \in \mathbf{P}_\theta$ and $\sum_p p^{-3/2} |f(p)|^3 < \infty$.
Then there exists a nontrivial random Radon measure $m_\infty^\RadSym(ds)$ on $\RR$  such that the following are true: for any bounded interval $I$ and any test function $h \in C(I)$, we have
\begin{align}\label{eq:thm:mc-convergenceRad}
   (i) \quad m_{y, \infty}^\RadSym(h) \xrightarrow[y \to \infty]{L^2} m_\infty(h)
    \qquad \text{and} \qquad
  (ii) \quad   m_{\infty, t}^\RadSym(h) \xrightarrow[t \to \infty]{L^2} m_\infty(h)
\end{align}

\noindent and in particular the above convergence also holds in probability. Moreover,  the limiting measure $m_\infty^\RadSym(ds)$ is supported on $\RR$ almost surely.
\end{thm}

In order to establish \Cref{thm:mc-convergence-rad}, we need
the following analogue of \Cref{lem:cross-moments}, the proof of which is straightforward and omitted here.
\begin{lem}\label{lem:cross-moments-rad}
Let $f \colon \NN \to \RR$. Let $z_j = \sigma_{t_j} + is_j$ for $j=1, 2$. We have
\begin{align} \label{eq:1-point-rad}
\EE\left[ \left|1 +  \frac{\Rad(p)f(p)}{p^{z_1}}\right|^2\right]
& = 1 +\frac{f(p)^2}{p^{2\sigma_{t_1}}}
= \exp \left( \frac{f(p)^2}{p^{2\sigma_{t_1}}} + \asymptoticO \left(\frac{f(p)^4}{p^{2}}\right)\right)
\end{align}

\noindent and
\begin{equation}\label{eq:2-point-rad}
\begin{split}
&\EE\left[ \prod_{j=1}^2 \left|1 + \frac{\Rad(p)f(p)}{p^{z_j}}\right|^2 \right]\\
& = \exp \Bigg\{ 
f(p)^2\left[ \frac{1}{p^{2\sigma_{t_1}}} + \frac{1}{p^{2\sigma_{t_2}}}
+ \frac{4\cos(s_1 \log p) \cos(s_2 \log p)}{p^{\sigma_{t_1} + \sigma_{t_2}}}\right] + \asymptoticO\left(
\left|\frac{f(p)}{p^{1/2}}\right|^3 + \left|\frac{f(p)}{p^{1/2}}\right|^4 \right)
\Bigg\}\\
& = \exp \Bigg\{ 
f(p)^2\left[ \frac{1}{p^{2\sigma_{t_1}}} + \frac{1}{p^{2\sigma_{t_2}}}
+ 2\frac{\cos((s_1 + s_2) \log p) + \cos((s_1 - s_2) \log p)}{p^{\sigma_{t_1} + \sigma_{t_2}}}\right] + \asymptoticO\left(
\left|\frac{f(p)}{p^{1/2}}\right|^3 + \left|\frac{f(p)}{p^{1/2}}\right|^4 \right)
\Bigg\}\\
\end{split}
\end{equation}

\noindent uniformly in $t_j \ge 2$, $s_j \in \RR$.
\end{lem}

\begin{proof}[Proof of \Cref{thm:mc-convergence-rad}]
Most of the arguments in the 4-step approach in \Cref{subsubsec:pf-mc-stein} will go through ad verbatim. As such, we shall only highlight the necessary changes below. Without loss of generality (by the triangle inequality of the $L^2$-norm), it will be convenient for us to assume that $h \in C(I)$ where $I = [n, n+1]$ for some $n \in \ZZ$.

\paragraph{Step 1: martingale convergence.} In view of the new estimates in \Cref{lem:cross-moments-rad}, we have
\begin{align*}
    \EE[m_{y, \infty}^\RadSym(h)^2]
    & \ll \|h\|_\infty^2 \iint_{I \times I} \exp \left(2\sum_{p \le y} \frac{f(p)^2}{p} \left[\cos((s_1 + s_2) \log p) + \cos((s_1 - s_2)\log p)\right]\right) ds_1 ds_2\\
    & \ll \iint_{I \times I} \frac{ds_1 ds_2}{|s_1 + s_2|^{2\theta}|s_1 - s_2|^{2\theta} }
\end{align*}

\noindent where the last line follows from \Cref{lem:martingaleL2est} (with $c = 0$, and $s \in \{s_1 + s_2, s_1 - s_2\}$). The above integral is obviously finite if $|s_1 + s_2|$ is bounded away from $0$, and so it is sufficient to consider the case where $s_1, s_2 \in I := [0, 1]$. But
\begin{align*}
     \iint_{[0,1]^2} \frac{ds_1 ds_2}{|s_1 - s_2|^{2\theta} |s_1 + s_2|^{2 \theta}}
     & \ll \int_0^1 \frac{ds_1}{s_1^{2\theta}}\int_0^1 \frac{ds_2}{|s_1 - s_2|^{2\theta}}
     = \int_0^1 \frac{ds_1}{s_1^{2\theta}} \left[ |1-s_1|^{1-2\theta} + |s_1|^{1-2\theta}\right] \ll 1.
\end{align*}

\noindent Thus we can still apply $L^2$ martingale convergence theorem as before.

\paragraph{Step 2: truncation.}
Using \Cref{lem:cross-moments-rad}, we have
\begin{align*}
K_y^\RadSym(s_1, s_2)
& := \prod_{p > y} \frac{\EE\left[ \left|1 +\frac{\Rad(p)f(p)}{p^{z_1}}\right|^2 \left|1 +\frac{\Rad(p) f(p)}{p^{z_2}}\right|^2\right]}{\EE\left[ \left|1 +\frac{\Rad(p)f(p)}{p^{z_1}}\right|^2 \right]\EE\left[ \left|1 +\frac{\Rad(p)f(p)}{p^{z_2}}\right|^2\right]}\\
& = \exp \left(2\sum_{p > y} \frac{f(p)^2}{p^{1 + \frac{1}{\log y}}} \left[\cos((s_1 + s_2) \log p) + \cos((s_1 - s_2)\log p)\right]
+ \asymptoticO\left(\sum_{p > y} \left|\frac{ f(p)}{p^{1/2}}\right|^3\right)\right).
\end{align*}

\noindent By \Cref{lem:martingaleL2est-2} (with $g = |f|^2$ and $s  \in \{s_1 + s_2, s_1 - s_2\}$), $K_y^\RadSym(s_1, s_2)$ is uniformly bounded for $y \ge 3$ and $s_1, s_2 \in I$, and converges to $1$ for almost every $s_1, s_2 \in I$. Also,
\begin{align*}
\frac{\EE\left[A_y^\RadSym(\sigma_y + is_1)|^2A_y^\RadSym(\sigma_y+ is_2)|^2\right]}{\EE\left[A_y^\RadSym(\sigma_y + is_1)|^2\right]\EE\left[A_y^\RadSym(\sigma_y + is_2)|^2\right]}
& \ll   \exp \left(2\sum_{p \le y} \frac{f(p)^2}{p^{1+\frac{1}{\log y}}} \left[\cos((s_1+s_2) \log p) + \cos((s_1 - s_2)\log p)\right]\right)\\
& \ll |s_1+s_2|^{-2\theta} |s_1 - s_2|^{-2\theta}
\end{align*}

\noindent by \Cref{lem:martingaleL2est} which is integrable as explained in Step 1 earlier. Hence $\lim_{y \to \infty} \EE\left[ \left| m_{\infty, y}(h) - m_{y,y}(h)\right|^2 \right] = 0$ by an application of dominated convergence, similar to that in the proof of \Cref{lem:mcL2truncate}.

\paragraph{Step 3: comparing two truncated Euler products.}
We would like to use the second moment method and show that each expectation on the right-hand side of \eqref{eq:2nd-moment-expand} (in the Rademacher case) converges to the same limit as $y \to \infty$. One can study \eqref{eq:truncated-DCT} using dominated convergence as before, but now with the new estimate
\begin{align*}
&\frac{\EE\left[A_y^\RadSym(\sigma_{t_1} + is_1)|^2A_y^\RadSym(\sigma_{t_2}+ is_2)|^2\right]}{\EE\left[A_y^\RadSym(\sigma_{t_1} + is_1)|^2\right]\EE\left[A_y^\RadSym(\sigma_{t_2} + is_2)|^2\right]}\\
& \quad \ll \exp \left( \sum_{p \le y}  \frac{2f(p)^2}{p^{1+\frac{1}{2}[\log^{-1} t_1 + \log^{-1} t_2]}} \left[\cos((s_1 + s_2) \log p) + \cos((s_1 - s_2)\log p)\right]\right)
\ll  |s_1+s_2|^{-2\theta} |s_1 - s_2|^{-2 \theta}
\end{align*}

\noindent where $t_1, t_2 \in \{y, \infty\}$, using \Cref{lem:martingaleL2est-2} as in Step 2.

\paragraph{Step 4: full support.}
The argument here does not rely on any estimates for particular random multiplicative function, and the proof is complete.
\end{proof}

Following the same arguments in \Cref{sec:var-Tauberian}, we can immediately deduce
\begin{lem}
Under the same setting as \Cref{thm:mc-convergence-rad}, we have
\begin{align}
    U_x^{\RadSym} \xrightarrow[x \to \infty]{p} \frac{1}{2\pi} \int_{\RR} \frac{m_\infty^\RadSym(ds)}{|\tfrac{1}{2} + is|^2}.
\end{align}
\end{lem}
We have now collected all the necessary ingredients for  \Cref{thm:mainRad} and \Cref{cor:mom-convergenceRad}, the proof of which will be omitted as one may follow the same arguments as in the Steinhaus setting.

\appendix

\section{Mean values of nonnegative multiplicative functions}\label{app:mean}
The well-known lemma below gives alternative expressions for the main term in the right-hand side of \eqref{eq:flatsum}.
\begin{lem}\label{lem:slow}
Fix $\theta>0$. Let $g$ be a nonnegative multiplicative function satisfying  \eqref{eq:loglim} and \eqref{eq:series}. Let
\[ E(t)= \sum_{p \le t} \frac{(g(p)-\theta)\log p}{p}.\]
Then $E(t)=o(\log t)$ as $t \to \infty$, and
\begin{align*}	
\frac{e^{-\gamma \theta}}{\Gamma(\theta)} \frac{x}{\log x} \prod_{p\le x} \left(\sum_{i=0}^{\infty} \frac{g(p^i)}{p^{i}}\right) &\sim   \frac{x (\log x)^{\theta-1}}{\Gamma(\theta)} \prod_{p\le x} \left(\sum_{i=0}^{\infty} \frac{g(p^i)}{p^{i}}\right)\left(1-\frac{1}{p}\right)^{\theta}\sim   \frac{x (\log x)^{\theta-1}}{\Gamma(\theta)} L_g(\log x)
\end{align*}
holds as $x \to \infty$, 	where $L_g\colon [0,\infty) \to (0,\infty)$ is a continuous slowly varying function given by
	\begin{align*} L_g(\log x) &:= C_g \exp\left(\int_{2}^{x} \frac{E(t)dt}{t\log^2 t}\right),\\
 C_g &:= \prod_{p} \exp\left( \frac{\theta}{p}\right)\left(1-\frac{1}{p}\right)^{\theta}\exp\left( -\frac{g(p)}{p}\right) \left(\sum_{i=0}^{\infty}  \frac{g(p^i)}{p^i} \right) \in (0, \infty).
 \end{align*}
\end{lem}
\begin{proof}
Since $\sum_{p \le t}(\log p)/p \sim \log t$ as $t \to \infty$ by Mertens' Theorems \cite[Theorem~2.7]{MV}, we have $E(t)=o(\log t)$ by \eqref{eq:loglim}. As
\[ \prod_{p \le x} \left(1-\frac{1}{p}\right)\sim \frac{e^{-\gamma}}{\log x}\]
by Mertens' Theorems \cite[Theorem~2.7]{MV}, it follows that as $x \to \infty$,
\[ 	\frac{e^{-\gamma \theta}}{\Gamma(\theta)} \frac{x}{\log x} \prod_{p\le x} \left(\sum_{i=0}^{\infty} \frac{g(p^i)}{p^{i}}\right) \sim   \frac{x (\log x)^{\theta-1}}{\Gamma(\theta)} L_{g,0}(\log x)\]
where
	\[ L_{g,0}(\log x): =  \prod_{p\le x} \left(\sum_{i=0}^{\infty} \frac{g(p^i)}{p^{i}}\right)\left(1-\frac{1}{p}\right)^{\theta}.\]
We decompose $L_{g,0}$ as $L_{g,0}=  L_{g,1} L_{g,2} L_{g,3}$ where
	\begin{align*}
		L_{g,1}(\log x)& =	 \exp\left( \sum_{p \le x} \frac{g(p)-\theta}{p} \right),\\  L_{g,2}(\log x) &= \prod_{p \le x} \exp\left( \frac{\theta}{p}\right)\left(1-\frac{1}{p}\right)^{\theta}, \quad L_{g,3}(\log x)= \prod_{p \le x} \exp\left( -\frac{g(p)}{p}\right) \left(\sum_{i=0}^{\infty}  \frac{g(p^i)}{p^i} \right).
	\end{align*}
Using $\log(1+t)=t+O(t^2)$ for $|t|\le 1/2$ we see that $L_{g,2}(\log x)$ tends to the positive and finite limit $\prod_{p}\exp(\theta/p)(1-1/p)^{\theta}$ as $x \to \infty$.  Similarly, making also use of \eqref{eq:series}, $L_{g,3}(\log x)$ tends to the positive and finite limit $\prod_{p} \exp( -g(p)/p) (\sum_{i=0}^{\infty} g(p^i)/p^i)$. These two products multiply to $C_g$. To treat $L_{g,1}$ we write
\[ \sum_{p \le x} \frac{g(p)-\theta}{p} =\int_{2^{-}}^{x} \frac{dE(t)}{\log t} = \frac{E(x)}{\log x}  +\int_{2}^{x} \frac{E(t)dt}{t\log^2 t} =o(1)+\int_{2}^{x} \frac{E(t)dt}{t\log^2 t}  .\]
As $E$ is a step function, $\int_{2}^{x} \frac{E(t)dt}{t\log^2 t}$ is continuous. Demonstrating that $L_g$ (or $L_{g,1}$) is slowly varying amounts to showing that
\[ \int_{e^x}^{e^{xu}} \frac{E(t)dt}{t\log^2 t}=\int_{\log x}^{\log x +\log u} \frac{E(e^{e^v})dv}{e^v}\]
tends to $0$ as $x \to \infty$. This follows from $E(t)=o(\log t)$. 
\end{proof}
Given $\theta>0$, Wirsing \cite{Wirsing} considered a family of multiplicative functions, which we denote  $\Ma'_{\theta}$. It consists of all nonnegative multiplicative $g$-s that satisfy $\sum_{p \le t}g(p) \sim \theta t/\log t$ as $t \to \infty$ and such that there exists $A \in (0,2)$ for which $g(p^i)= O(A^i)$ holds for all $p,i \ge 2$.
Wirsing proved the following.
\begin{lem}[{\cite[Satz 1]{Wirsing}}]\label{lem:u}
	Fix $\theta>0$.  Let $g\in \Ma'_{\theta}$. Then \eqref{eq:flatsum} holds as $x \to \infty$ for every $h$ of the form $g(\cdot)\mathbf{1}_{(\cdot,d)=1}$ for some $d \in \NN$.
\end{lem}
We claim \Cref{lem:u} implies
\begin{equation}\label{eq:inc} \Ma'_{\theta} \subseteq \Ma_{\theta}.
\end{equation}
To see \eqref{eq:inc}, observe that $\sum_{p \le t} g(p) \sim \theta t/\log t$ implies $\sum_{p\le t} g(p)(\log p)/p \sim \theta \log t$ by integration by parts. Moreover, the condition $g(p^i)=O(A^i)$ implies $\sum_{p, i\ge 2} g(p^i)/p^i \ll_A \sum_{p} 1/p^2 < \infty$ and similarly $\sum_{p, i \ge 2:\, p^i \le t} g(p^i) \ll_A t/\log t$.

A special case of the main result of de Bruijn and van Lint \cite{dBvL} says the following.
\begin{lem}[\cite{dBvL}]\label{lem:sum}
 Fix $\theta>0$. Let $g \in \Ma_{\theta}$. Fix $\OurEpsilon>0$. Denote $u=\log x/ \log y$.  Then uniformly for $u\in [\OurEpsilon,1/\OurEpsilon]$,
\[\sum_{n \le x,\, P(n)\le y} g(n) \sim  \frac{\Gamma(\theta)\rho_{\theta}(u)}{u^{\theta-1}} \sum_{n \le x} g(n)\]
holds as $x \to \infty$.
\end{lem}
In view of \eqref{eq:inc}, \Cref{lem:sum} holds for all $g \in \Ma'_{\theta}$.

The following is a logarithmic version of \Cref{lem:u} which is deduced from the Hardy--Littlewood--Karamata Tauberian theorem, see \cite[Theorem~2]{Geluk} or \cite[Theorem~3]{Fainleib} for full details.
\begin{lem}\label{lem:log}
Fix $\theta>0$. Let $g$ be a nonnegative multiplicative function satisfying \eqref{eq:loglim} and \eqref{eq:series}. Then
\begin{equation}\label{eq:logsum}\sum_{n \le x} \frac{g(n)}{n} \sim \frac{e^{- \gamma\theta}}{\Gamma(\theta+1)}  \prod_{p\le x} \left(\sum_{i=0}^{\infty} \frac{g(p^i)}{p^{i}}\right)
\end{equation}
	holds as $x \to \infty$.
\end{lem}
\begin{rem}
Slightly weaker versions of \Cref{lem:log} appear in de Bruijn and van Lint \cite[Equation~(1.8)]{dBvL1st}.
and Wirsing \cite[Hilfssatz~6]{Wirsing}. 
If $g \in \Ma_{\theta}$ then \eqref{eq:logsum} can be deduced from \eqref{eq:flatsum} (with $h=g$) by integration by parts. This requires showing that
\begin{equation}\label{eq:ibp}
\int_{0}^{\log x} v^{\theta-1} L_g(v) dv \sim \frac{(\log x)^{\theta}L_g(\log x)}{\theta}
\end{equation}
holds as $x \to \infty$, where $L_g$ is the slowly varying function in \Cref{lem:slow}. The asymptotic relation \eqref{eq:ibp} is a special case of Karamata's Theorem \cite[Proposition~IV.5.1]{Korevaar}, which can also be verified directly using the formula for $L_g$ given in \Cref{lem:slow}. \end{rem}
The following upper bound can sometimes act as a substitute for \Cref{lem:u} and \Cref{lem:sum}.
\begin{lem}[{\cite[Theorem~2.14]{MV}}]\label{lem:upper}
Let $g$ be a nonnegative multiplicative function $g$ satisfying $\sum_{p \le t} g(p) \le A t/\log t$ for $t \ge 2$ and $\sum_{p,i \ge 2} g(p^i)\log(p^i)/p^i \le A$. Then 
	\begin{equation*}
		\sum_{n \le x} g(n)\ll_A \frac{x}{\log x}\sum_{n \le x}\frac{g(n)}{n}
	\end{equation*}
 holds for $x \ge 2$, and the implied constant depends on $A$ only.
\end{lem}

\section{Results from probability theory}
\subsection{Stable convergence}
\begin{definition}\label{def:stable}
    Let $(\Omega, \Fa, \PP)$ be a probability space, and $\Ga \subset \Fa$ be a sub-$\sigma$-algebra. We say a sequence of random variables $Z_n$ converges $\Ga$-stably as $n \to \infty$ if $(Y, Z_n)$ converges in distribution as $n \to \infty$ for any bounded $\Ga$-measurable random variable $Y$. When $\Ga = \Fa$, we say $Z_n$ converges stably.
\end{definition}
In the language of weak convergence of probability measures, the stable convergence in \Cref{thm:main} can be rephrased in the following way: if $Y$ is a real-valued random variable that is measurable with respect to $\Fa_\infty := \sigma(\alpha(p), p \ge 2)$, then for any bounded continuous function $h\colon \RR \times \RR \to \RR$, we have
\begin{align}\label{eq:stable2}
    \lim_{x \to \infty} \EE\left[h(Y, S_x)\right]
    = \EE\left[h\left(Y, \sqrt{V_\infty} G\right) \right].
\end{align}

\noindent By a standard density argument (and up to renaming of random variables), \Cref{thm:main} can be equivalently rephrased as
\begin{align} \label{eq:stable3}
    \lim_{x \to \infty} \EE\left[Y \widetilde{h}(S_x)\right]
    = \EE\left[Y\widetilde{h}(\sqrt{V_\infty} G) \right]
\end{align}

\noindent for any bounded continuous function $\widetilde{h}\colon \CC \to \RR$, and any bounded $\Fa_\infty$-measurable random variable $Y$ (for example $Y = \widehat{h}\left(\int_{\RR} \frac{m_{\infty}(ds)}{|1/2 + is|^2}\right)$ where $\widehat{h}\colon \RR \to \RR$ is some bounded measurable function). By \eqref{eq:stable3}, 
\begin{align*}
\lim_{x \to \infty}\EE\left[ Y \exp\left(i (a_1 \Re S_x + a_2 \Im S_x)\right) \right]
= \EE\left[Y \exp \left(- \frac{1}{4}(a_1^2+a_2^2) \int_\RR\frac{m_{\infty}(ds)}{2\pi |1/2 + is|^2} \right) \right] 
\end{align*}
\noindent holds for all $a_1, a_2 \in \RR$. Alternatively, one can state the stable convergence in terms of probability distribution functions: for any $A_1 \subset \RR$, and $A_2 \subset \CC$ which is a continuity set with respect to the distribution of $\sqrt{V_\infty} G$, we have
\begin{align}\label{eq:stable4}
\lim_{x \to \infty} \PP(Y \in A_1, S_x \in A_2)
= \PP(Y \in A_1, \sqrt{V_\infty} G \in A_2).
\end{align}

\noindent In particular, \eqref{eq:stable4} holds if the boundary of $A_2$ has zero Lebesgue measure (because of the existence of probability density function for $G \sim \Na_{\CC}(0,1)$). 

Specialising \eqref{eq:stable4} to $A_1 = \RR$, we obtain $S_x \xrightarrow[x \to \infty]{d} \sqrt{V_\infty} G$. As such, stable convergence is a strengthening of the usual convergence in distribution. The additional information we obtain here is that the size of the fluctuation of $S_x$ (i.e., conditional variance) is still determined by our randomness $(\alpha(p))_p$ in a measurable way in the large-$x$ limit, even though the overall fluctuation is not (similar to what one observes in the vanilla central limit theorem for i.i.d.~random variables). We refer the interested readers to \cite{HL2015} for more discussions and examples of stable limit theorems.

\subsection{Complexifying the martingale central limit theorem}\label{app:cmCLT}
We shall deduce \Cref{lem:mCLT} from the  real-valued martingale central limit theorem, stated in \Cref{thm:mCLT}.
\begin{proof}[Proof of \Cref{lem:mCLT}]
By the Cram\'er--Wold device \cite[Corollary~6.5]{Kal2021}, it suffices to show that the following claim is true: for any $a_1, a_2 \in \RR$, we have
\begin{align}\label{eq:pf-cramerwold}
a_1 \Re M_{n, k_n} + a_2 \Im M_{n, k_n}  \xrightarrow[n \to \infty]{d} \sqrt{V_\infty} W
\end{align}

\noindent where $W \sim \Na\left(0, \dfrac{a_1^2 + a_2^2}{2}\right)$ is independent of $V_\infty$, and that the convergence is stable.

To do so, let us consider $T_{n, j} := a_1 \Re M_{n, j} + a_2 \Im M_{n, j}$. Then $(T_{n, j})_{j \le k_n}$ is a martingale with respect to $(\Fa_{j})_{j \le k_n}$, and we write $\Delta^T_{n, j} := T_{n, j} - T_{n, j-1}$. Using the fact that
\begin{align*}
2 \left(\Re \Delta_{n, j}^T\right)^2 & = \Re \left[ |\Delta_{n, j}^T|^2 + \left(\Delta_{n, j}^T\right)^2\right], \\
2 \left(\Im \Delta_{n, j}^T\right)^2 & = \Re \left[ |\Delta_{n, j}^T|^2 - \left(\Delta_{n, j}^T\right)^2\right], \\
2\left(\Re \Delta_{n, j}^T\right)\left(\Im \Delta_{n, j}^T\right) & = \Im \left[\left(\Delta_{n, j}^T\right)^2\right],
\end{align*}

\noindent and $|\Delta^T_{n, j}|^2 = a_1^2 \left(\Re \Delta_{n, j}^T\right)^2 + 2a_1a_2 \left(\Re \Delta_{n, j}^T\right)\left(\Im \Delta_{n, j}^T\right) + a_2^2 \left(\Im \Delta_{n, j}^T\right)^2$, we see that
\begin{align*}
\sum_{j=1}^{k_n} \EE\left[ |\Delta_{n, j}^T|^2 | \Fa_{j-1}\right]
& = \frac{a_1^2 + a_2^2}{2} \  \Re \sum_{j=1}^{k_n} \EE\left[ |\Delta_{n, j}^T|^2 | \Fa_{j-1}\right]
+ \frac{a_1^2 - a_2^2}{2} \  \Re \sum_{j=1}^{k_n} \EE\left[ \left(\Delta_{n, j}^T\right)^2 | \Fa_{j-1}\right]\\
& \qquad + a_1 a_2 \Im \sum_{j=1}^{k_n} \EE\left[ \left(\Delta_{n, j}^T\right)^2 | \Fa_{j-1}\right]\\
& \xrightarrow[n \to \infty]{p} \frac{a_1^2+a_2^2}{2} V_\infty.
\end{align*}

\noindent Moreover, for any $\delta > 0$ we have
\begin{align*}
 \sum_{j \ge 1} \EE[|\Delta_{n, j}^T|^2 1_{\{|\Delta_{n, j}^T| > \delta\}} | \Fa_{j-1}] 
\le  \sum_{j \ge 1} \EE\left[2(a_1^2 + a_2^2)|\Delta_{n, j}|^2 1_{\{\sqrt{2 (a_1^2 + a_2^2)}|\Delta_{n, j}| > \delta\}} | \Fa_{j-1}\right] 
\xrightarrow[n \to \infty]{p} 0,
\end{align*}

\noindent i.e., the conditional Lindeberg condition is also satisfied. Therefore, we can apply \Cref{thm:mCLT} and obtain \eqref{eq:pf-cramerwold}, which concludes our proof.
\end{proof}

\subsection{Convergence of random measures} \label{sec:convergence_measure}
Let $\mathfrak{M}$ be the space of (nonnegative) Radon measures on $\RR$ equipped with the vague topology -- this is the coarsest topology such that the evaluation map $\pi_f\colon \nu \mapsto \nu(f)$ is continuous for all 
functions $f\colon \RR \to \RR$ that are continuous and compactly supported (we denote this class of test functions by $C_c(\RR)$). Equivalently, this topology is characterised by the following convergence criterion: we say a sequence of Radon measures $\nu_n$ converges to $\nu$ if and only if $\lim_{n \to \infty} \nu_n(f) = \nu(f)$ for all $f \in C_c(\RR)$ (see e.g., \cite[Lemma 4.1]{Kal2017}).

It is well known that $\mathfrak{M}$ is a Polish space \cite[Theorem 4.2]{Kal2017}, i.e., there exists a metric that metrises the topology of vague convergence and turns $\mathfrak{M}$ into a complete separable metric space. A very natural choice would be the so-called Prokhorov metric, see \cite[Lemma 4.6]{Kal2017}. As such, we may view random measures as abstract Polish-space-valued random variables and speak of random measures converging in probability. In particular, using the subsequential characterisation of limits and a diagonal argument, it is not difficult to show that the following three statements are equivalent (see \cite[Lemma 4.8]{Kal2017}):
\begin{itemize}
    \item $\nu_n$ converges in probability to $\nu$ with respect to e.g., the Prokhorov metric.
    \item For any subsequence $(n_m)_m$ there exists a further subsequence $(n_{m_k})_k$ such that $\nu_{n_{m_k}}$ converges to $\nu$ with respect to the vague topology almost surely.
    \item For all $f \in C_c(\RR)$, we have $\nu_n(f) \xrightarrow[n \to \infty]{p} \nu(f).$
\end{itemize}

Going back to \Cref{rem:missing-abstract}, since the proof of \Cref{thm:mc-convergence} in \Cref{subsubsec:pf-mc-stein} shows that \eqref{eq:thm:mc-convergence} holds whenever $h \in C(I)$ for any compact interval $I$, the statement is true for all $h \in C_c(\RR)$ in particular. Thus we may conclude that there exists a random Radon measure $m_\infty(ds)$ such that for each fixed $h \in C_c(\RR)$ the martingale limit $m_\infty(h)$ is realised as the integral $m_\infty(h) = \int_\RR h(s) m_\infty(ds)$.

\bibliographystyle{abbrv}
\bibliography{references}
\end{document}